\documentclass[12pt,a4paper]{article}
\usepackage{cmap}
\usepackage[cp1251]{inputenc}
\usepackage[T1]{fontenc}
\usepackage{mathrsfs}
\usepackage[english]{babel}
\usepackage{latexsym,amsmath,amssymb,amsthm}
\usepackage[nodisplayskipstretch]{setspace}
\usepackage{enumitem}
\setlist{nolistsep,parsep=0pt}
\usepackage{float}
\usepackage[top=1cm,bottom=1.2cm,left=1.2cm,right=1.2cm]{geometry}
\usepackage{setspace}
\usepackage{graphicx}
\usepackage[section]{placeins}
\usepackage{cite}
\usepackage[square,numbers,sort&compress]{natbib}
\usepackage[colorlinks,citecolor=blue,linkcolor=blue,urlcolor=black]{hyperref}
\usepackage{titlesec}
\usepackage{euscript}
\usepackage{verbatim}
\sffamily
\titlelabel{\thetitle.\quad}

\allowdisplaybreaks[4]
\numberwithin{equation}{section}

\theoremstyle{definition}
 \newtheorem{definition}{Definition}[section]
 \newtheorem{remark}{Remark}[section]

\theoremstyle{plain}
 \newtheorem{theorem}{Theorem}[section]
 \newtheorem{lemma}{Lemma}[section]
 \newtheorem*{lemma*}{Lemma}
 \newtheorem{corollary}{Corollary}[section]

\begin{document}

\newcommand{\coker}{\mathop{\mathrm{coker}}}

\newcommand{\spL}{\mathrm{L}}
\newcommand{\const}{\mathrm{const}}
\newcommand{\mT}{{\mathcal{T}}}

\newcommand{\Es}{\EuScript}

\newcommand{\mE}{{\mathcal{E}}}
\newcommand{\mV}{{\mathcal{V}}}
\newcommand{\mI}{{\mathcal{I}}}

\newcommand{\R}{{\mathbb R}}
\newcommand{\N}{{\mathbb N}}
\newcommand{\mC}{{\mathbb C}}
\newcommand{\Z}{{\mathbb Z}}

\newcommand{\Span}{\mathop{\mathrm{span}}}

\newcommand{\dotpl}{\mathop{\dot +}}
\newcommand{\timesO}{\mathop{\times}}

\newcommand{\T}{{\mathop{\mathrm{T}}}} 
\newcommand{\sgn}{\mathop{\mathrm{sgn}}}
\newcommand{\range}{\mathop{\mathrm{range}}}
\newcommand{\rank}{\mathop{\mathrm{rank}}}
\newcommand{\res}{\mathop{\mathrm{Res}}}
\newcommand{\ii}{\mathrm{i}}
\newcommand{\dd}{\textup{d}}
\newcommand{\la}{\lambda}
\renewcommand{\Im}{\operatorname{Im}}
\renewcommand{\Re}{\operatorname{Re}}
\newcommand{\grad}{\mathop{\mathrm{grad}}}

\title{\textbf{Existence and qualitative behavior of solutions of abstract differential-algebraic equations}}
\author{\large \textbf{Maria Filipkovska} \\
 \it\small Institute of Analysis and Scientific Computing, Technische Universit\"{a}t Wien, \\
 \it\small Wiedner Hauptstra{\ss}e 8-10, 1040 Wien, Austria \\
 \it\small  B. Verkin Institute for Low Temperature Physics and Engineering \\
 \it\small  of the National Academy of Sciences of Ukraine,\; Nauky Ave. 47, 61103 Kharkiv, Ukraine\\
 \small mariia.filipkovska@asc.tuwien.ac.at;\quad filipkovskaya@ilt.kharkov.ua 
 }%

\date{ }
\maketitle

 \begin{abstract}
Abstract differential-algebraic equations (ADAEs) of a semilinear type are studied.
Theorems on the existence and uniqueness of solutions and the maximal interval of existence, on the global solvability of the ADAEs, the boundedness of solutions and the blow-up of solutions are presented. Previously, an ADAE is reduced to a system of explicit differential equations and algebraic equations by using projectors. The number of equations of the system depends on the index of the characteristic pencil of the ADAE. We consider the pencil of an arbitrarily high index.
 \end{abstract}

{\small\textbf{Keywords:}\; abstract differential-algebraic equation; degenerate evolution equation; higher index; global existence; boundedness; blow-up, maximal interval

\medskip MSC2020: 34A09, 35A01, 35R20, 34A12, 34C11, 58C15 }

 \section{Introduction}\label{Intro}

A system consisting of differential and algebraic equations can be represented in the form of an abstract evolution equation which is often called a differential-algebraic equation (DAE), when it is considered in finite-dimensional spaces, and an abstract differential-algebraic equation (ADAE), when it is considered in infinite-dimensional spaces. If the system contains a PDE, then it is often called a partial differential-algebraic equation (PDAE).
ADAEs is a general class of equations that is intensively used both in solving practical problems and in theoretical research. A significant feature of these equations is that any type of a PDE can be represented as an ADAE in appropriate infinite-dimensional spaces, possibly with a complementary boundary condition. It is often much easier to study the corresponding ADAE than the original PDE. Moreover, this approach allows us to cover wider classes of PDEs, including systems of PDEs and algebraic equations (PDAEs).

In the present paper, we study both the existence and uniqueness of solutions and their behavior with increasing time (boundedness and blow-up of solutions). From a practical point of view, it is important not only to prove that a solution exists on some local interval, but also to determine the maximal interval of its existence. If the solution is global (it exists on an arbitrarily large time interval), then the corresponding real-world system will operate over a sufficiently long duration without interruption. If the solution is ``blow-up'' in finite time (it exists on a finite time interval and is unbounded), then the operation of the corresponding real system will be interrupted. For example, the disruption of a semiconductor is described as the blow-up of a solution of an appropriate initial value or initial boundary value problem.

DAEs and ADAEs arise from the modelling of systems and processes in control problems, mechanics, robotics, gas industry, chemical kinetics, radio engineering, economics and other fields. It is well known that DAEs are used in modelling various objects whose structure is described by directed graphs, e.g., electrical, gas and neural networks. The application of DAEs is described in gas network theory, e.g., in \cite{GNM,Brenan-C-P,  Fil.Sing-GN}, and in electrical circuit theory is shown, e.g., in
\cite{Vlasenko1,Riaza,FR999,Rutkas2007,Fil.CombMeth,Fil.NCombMeth,Brenan-C-P, Lamour-Marz-Tisch}. The DAEs of higher index arise, for example, in electrical engineering, gas industry, robotics, control problems and chemical kinetics (see, e.g., \cite{Brenan-C-P,Vlasenko1,Riaza,FR999,Rutkas2007,Lamour-Marz-Tisch}).

In this paper, we deal with the abstract differential-algebraic equation
 \begin{equation}\label{ADAE}
\frac{d}{dt}[Ax]+Bx=f(t,x),\qquad t>0,
 \end{equation}
where $A$ and $B$ are closed linear operators from $X$ into $Y$ with domains $D_A$ and $D_B$ respectively, $D=D_A\cap D_B\ne \{0\}$ is a lineal  (linear manifold), $X$ and $Y$ are Banach spaces (BSs), $D_A$ and $D_B$ are dense in $X$  (i.e., $\overline{D_A}=\overline{D_B}=X$), $\range(A)=\overline{\range(A)}$ and $f\in C(\R_+\times D,Y)$, $\R_+=[0,\infty)$.
The operator $A$, as well as $B$, can be noninvertible (degenerate). Therefore, the ADAE \eqref{ADAE} is also called a degenerate (or abstract) differential equation and  degenerate (or abstract) evolution equation. According to the DAE terminology, the equation of the type \eqref{ADAE} is called a semilinear ADAE.

The main results of the paper are theorems on the existence and uniqueness of solutions, the maximal interval of existence and the qualitative behaviour of solutions, including theorems on the global solvability (the maximal interval of existence is infinite), on the blow-up of solutions (the maximal interval is finite and the norm of a solution tends to infinity) and on the boundedness of solutions (the Lagrange stability).

The local solvability of the semilinear ADAE ${\frac{d}{dt}(Ax)+Bx(t)=F(t,Kx)}$, where $A$, $B$ and $K$ are closed linear operators in complex Banach spaces, has been studied in \cite{FR999}. More precisely, this ADAE was transformed into the one ${\frac{d}{dt}(Tv)+v(t)=f(t,Nv)}$ with the bounded operator ${T=A(\la A+B)^{-1}}$ by the change of variable ${x=e^{\la t}(\la A+B)^{-1}v}$, where $\la$ is a regular point (here ${N=K(\la A+B)^{-1}}$ and ${f(t,y)=e^{-\la t}F(t,e^{\la t}y)}$). Then conditions for the existence and uniqueness of local solutions of the ADAE with the bounded operator $T$ was found.
Here the assumption similar to that made for the pencil of the operators $A$, $B$ in the present paper (i.e., $\la A+B$ is a regular pencil of index $\nu\in {\mathbb N}\cup \{0\}$; see Section \ref{Sec_Pencil_ind}) was made for the operator $T$. Namely, it is supposed that $\zeta=0$ is a pole of multiplicity $m$ for the resolvent $(T-\zeta I)^{-1}$.  The similar approach was used in \cite{FP86}, where the simple pole case has been considered.
For a semilinear ADAE of the form \eqref{ADAE} with the regular pencil $\la A+B$ of index 2 (in terms of Definition \ref{Def_Pencil_ind}), the existence and uniqueness of a nonlocal (i.e., on a given finite interval in $\R$) solution was studied in \cite{Rut-Khudoshin} by using global restrictions (e.g., a global Lipschitz condition and the global boundedness of the norms of partial derivatives). In the present paper, such global restrictions are not used, which allows one to weaken restrictions on the nonlinear right-hand side of the ADAE.

Also, the results on the local solvability of semilinear ADAEs of the type \eqref{ADAE} have been obtained in the works of R.E. Showalter (see, e.g., \cite{Showalter72}). In \cite{Showalter72} certain coercive estimates and Lipschitz conditions are used.

For linear ADAEs in Hilbert spaces, the condition for the existence and uniqueness of solutions on a finite interval in $\R$ have been presented, e.g., in \cite[Section 12.3.3]{Lamour-Marz-Tisch} (see also references on p.~577). The Galerkin method has been used to obtain these results.
The existence and uniqueness of local solutions of certain semi-explicit and semilinear ADAEs of index 2 (in the sense that a stable semi-discretization leads to the finite-dimensional DAE of tractability index 2) has been studied in \cite{Heiland}.

\paragraph{The paper has the following structure.}  In Section~\ref{Intro}, the problem statement and some definitions are given. Sections \ref{Sec_Ind_DirDec_Proj} and \ref{Sec-Nonloc impl func} contain the preliminary information and  constructions. Section \ref{Sec_Ind_DirDec_Proj} provides the definition of the index of a regular pencil of operators, the direct decompositions of the domain $D$ and the spaces $X$, $Y$ and the associated projectors.
These projectors are used to reduce the ADAE to a system of explicit differential equations and algebraic equations. In Section \ref{Sec-Nonloc impl func} nonlocal (global) implicit function theorems, which are applied in the proofs of the main theorems, are given. The main results including the theorems mentioned above are presented in Section \ref{Sec-Main}.

\paragraph{Notation used throughout the paper:}
 \begin{itemize}[itemsep=2pt,parsep=0pt,topsep=3pt,leftmargin=1cm]
\item $I_X$ denotes the identity operator in the space $X$; $I$ denotes an identity operator (it will be clear from the context in what space);
\item $\overline{D}$ denotes the closure of a set $D$; $\partial D$ is the boundary of $D$; $D^c$ is the complement of $D$; $D_A$ is the domain of a linear operator $A$;
\item  $\spL(X,Y)$ is the space of bounded linear operators from $X$ to $Y$; $\spL(X,X)=\spL(X)$; similarly, $C((a,b),(a,b))=C(a,b)$;
\item $\Span\{x_i\}_{i=0,...,N}$ or $\Span\{x_i,\; i=0,...,N\}$ is the linear span of a system $\{x_i\}_{i=0,...,N}$;
\item $L_1\dotpl L_2$, where $L_1$ and $L_2$ are linear spaces, is the direct sum of $L_1$ and $L_2$; $\dotpl\limits_{i=1}^n L_i=L_1\dotpl \ldots \dotpl L_n$;
\item $X^*$ is the conjugate (or adjoint, or dual) space of $X$;
    $\delta_{ij}$ is the Kronecker delta;
\item $\EuScript A^{(-1)}$ is the semi-inverse operator of an operator $\EuScript A$ \,($A^{-1}$ is the inverse operator of $A$) (the definition of a semi-inverse operator can be found in \cite{Faddeev}; see also \cite{Nashed76} where such an operator is called an algebraic generalized inverse and, in Banach spaces, a topological generalized inverse);
\item the partial derivative $\frac{\partial}{\partial x}$ is denoted by $\partial_x$; $\partial_{(x_1,x_2,...,x_n)}=\frac{\partial}{\partial(x_1,x_2,...,x_n)}= \left(\frac{\partial}{\partial x_1},\frac{\partial}{\partial x_2},...,\frac{\partial}{\partial x_n}\right)$;
\item in some cases, the time derivative $\frac{d}{dt}$ is denoted by a dot~$\dot{ }$~;
\item if $V$ is a continuously differentiable functional, defined in a Banach space $X$, then the gradient of $V$ is denoted by $\partial_x V(x)$ or $\grad V(x)$; 
\item $\int\limits^{\infty}U(u)du=\infty$ \,(${<\infty}$) means that the integral $\int\limits_{u_0}^{\infty}U(u)du=\infty$ \,(${<\infty}$), i.e., does not converge (i.e., converges) for any $u_0>0$;
\item $\|\cdot\|$ denotes an appropriate norm in a space, unless it is explicitly stated which norm is considered.
 \end{itemize}

For simplicity of the presentation of further results, we assume the following. If $t\in [a,b]$, $a,b\in \R$, $a\ne b$, then by an open neighborhood $U_\varepsilon(a)$ of the point $a$ we mean a semi-open interval $[a,a+\varepsilon)$, $0<\varepsilon<b-a$, and, similarly, by an open neighborhood $U_\varepsilon(b)$ we mean a semi-open interval $(b-\varepsilon,b]$, $0<\varepsilon<b-a$.

 \paragraph{Definitions used throughout the paper.}\,

The function $x(t)$ is called a \emph{solution} of the ADAE \eqref{ADAE} on an interval $J\subseteq \R_+$ if $x\in C(J,X)$, $(Ax)\in C^1(J,Y)$ (\,if $J=\R_+$, then $(Ax)\in C^1((0,\infty),Y)$\,) and $x(t)$ belongs to $D$ and satisfies equation \eqref{ADAE} for every $t\in J$.

If a solution $x(t)$ exists on the whole interval $[t_0,\infty)$ (where $t_0$ is a given initial value), then it is called \emph{global}.

If a solution $x(t)$ is global and bounded, i.e., $\sup\limits_{t\in [t_0,\infty)} \|x(t)\|<\infty$, then it is called \emph{Lagrange stable}.

A solution $x(t)$ \emph{blows up in finite time} (or \emph{has a finite escape time}) and is also called \emph{Lagrange unstable} if it exists on some finite interval $[t_0,T)$ and $\lim\limits_{t\to T-0} \|x(t)\|=\infty$.

The equation \eqref{ADAE} is called \emph{Lagrange stable} (respectively, \emph{Lagrange unstable}) \emph{for an initial point} or initial values if the solution of the initial value problem for \eqref{ADAE} is Lagrange stable (respectively, Lagrange unstable) for this initial point or the corresponding initial values.

An equation is said to be \emph{Lagrange stable} if each its solution is Lagrange stable.

 \smallskip
Let $x(t)$ be a solution of a differential equation (possibly a DAE) on an interval $J\subseteq \R_+$. The interval $J$ is called a \emph{maximal interval of existence} of $x(t)$ if there does not exist an extension of $x(t)$ over an interval $J_1\supsetneqq J$ so that this extension remains a solution of the equation.  Let the solution $x(t)$ belong to some set $D$ for all $t\in J$. The interval $J$ is called a \emph{maximal interval of existence} of $x(t)$ \emph{in $D$} if there does not exist an extension of $x(t)$ over an interval $J_2\supsetneqq J$ so that this extension remains a solution of the equation and belongs to $D$ for all $t\in J_2$.

Notice that we will extend the solution only to the right and, for brevity, call $J$ a maximal interval, although it could more properly be called the right maximal interval.

 \smallskip
A mapping $f(t,x)$ of a set $J\times M$ into $Y$ is said to \emph{satisfy locally a Lipschitz condition} (or to \emph{be locally Lipschitz continuous}) \emph{with respect to $x$ on} $J\times M$ if for each (fixed) $(t',x')\in J\times M$ there exist open neighborhoods $U(t')$, $\widetilde{U}(x')$ of the points $t'$, $x'$ and a constant $L\ge 0$ such that $\|f(t,x_1)-f(t,x_2)\|\le L\|x_1-x_2\|$ for any $t\in U(t')$ and $x_1,x_2\in \widetilde{U}(x')$.
As usual, the mapping $f(t,x)$ \emph{satisfies a Lipschitz condition} (or to \emph{be Lipschitz continuous}) \emph{with respect to $x$ on} $J\times M$ if there exists a constant $l\ge 0$ such that  $\|f(t,x_1)-f(t,x_2)\|\le l\|x_1-x_2\|$ for any $t\in J$ and $x_1,x_2\in M$.

We denote by $\rho(M_1,M_2)=\inf\limits_{x_1\in M_1,\, x_2\in M_2}\|x_2-x_1\|$  the distance between the closed sets $M_1$, $M_2$ in $X$; $\rho(x,M)=\inf\limits_{y\in M}\|x-y\|$ denotes the distance from the point $x\in X$ to the set $M$.

 \section{Index of a regular pencil of operators, direct decompositions of spaces and the associated projectors}\label{Sec_Ind_DirDec_Proj}

Consider an operator pencil
\begin{equation}\label{Pencil}
P(\la)=\la A+B\colon D\to Y,
\end{equation}
where $A$ and $B$ are the operators defined in  \eqref{ADAE}, and $\la\in\mC$ is a parameter.

A point $\la$ such that there exists the bounded linear operator
$$
R(\la)=P^{-1}(\la)=(\la A+B)^{-1}
$$
defined on $Y$ is called a \emph{regular point} of the pencil $P(\la)$. The operator $R(\la)$ is called the \emph{resolvent} of $P(\la)$. By  $\varrho=\varrho(A,B):=\{\la\in\mC \mid \exists\, (\la A+B)^{-1}\in \spL(Y,X)\}$ (cf. \cite{Rutkas2007}) we denote the set of the regular points $\la$ defined above.
The pencil $P(\la)$ is said to be \emph{regular} if $\varrho\ne\emptyset$ and to be \emph{singular} (nonregular, irregular) if $\varrho=\emptyset$. The set $\varrho$ is open, and the resolvent as the operator function $R\colon \varrho\to \spL(Y,X)$ is holomorphic on $\varrho$.
Indeed, for each $\la_*\in \varrho$, it follows from the relation  $R(\la)=R(\la_*)\big[I_Y-(\la_*-\la)AR(\la_*)\big]^{-1}$ and inclusion $AR(\la_*)\in \spL(Y)$ that $R(\la)=\sum\limits_{k=0}^{\infty} (\la_*-\la)^k R(\la_*)\big(AR(\la_*)\big)^k$ for $|\la-\la_*|< \big\|AR(\la_*)\big\|_{\spL(Y)}^{-1}$ (cf. \cite{Rutkas2007}).
Moreover, the point $\la=\infty$ is said to be a \emph{regular point} of $P(\la)$ if in some neighborhood of the infinity ($|\la|>M$) all the points $\la$ are regular and $\|R(\la)\|\le C |\la|^{-1}$ for $|\la|>M$ \cite{Rutkas2007}.

In this work, the Banach spaces (BSs) $X$, $Y$ are, in general, complex.
If $X$, $Y$ are real BSs, then the pencil $P(\la)$ is called  \emph{regular} if $\varrho=\varrho(\tilde{A},\tilde{B})=\{\la\in\mC \mid \exists\, (\la \tilde{A}+\tilde{B})^{-1}\in \spL(\tilde{Y},\tilde{X})\}\ne \emptyset$, where the operators $\tilde{A}\colon D_{\tilde A}\to \tilde{Y}$, $\tilde{B}\colon D_{\tilde B}\to \tilde{Y}$ and the complex BSs $\tilde{X}$, $\tilde{Y}$ are the complex extensions of $A$, $B$ and the complexifications of $X$, $Y$, respectively,
and $\la \tilde{A}+\tilde{B}\colon \tilde{D}=D_{\tilde A}\cap D_{\tilde B}\to \tilde{Y}$. In this case, a point $\la\in \varrho=\varrho(\tilde{A},\tilde{B})$ is called a \emph{regular point} of $P(\la)$; the pencil $P(\la)$ is singular if $\varrho=\emptyset$; the same definition of the regularity of the point $\la=\infty$ holds as above, with the exception that $R(\la)$ is replaced by
$$
R_\mC(\la)=(\la\tilde{A}+\tilde{B})^{-1}.
$$
Notice that the operators ${\tilde A}^*$, ${\tilde B}^*$ from $\tilde{Y}^*$ into $\tilde{X}^*$ (with the domains $D_{{\tilde A}^*}$ and $D_{{\tilde B}^*}$ respectively) induce the adjoint operators  $A^*$, $B^*$, respectively, from $Y^*$ into $X^*$ (with the domains $D_{A^*}$, $D_{B^*}$ respectively) (see, e.g., \cite[p.~370--373]{Kantorovich-A}).

The set ${\sigma:=\mC\setminus \varrho}$ is called the \emph{spectrum} (or the finite spectrum) of $P(\la)$. If $\la=\infty$ is not a regular point of $P(\la)$, then the set $\widetilde{\sigma}=\sigma\cup\{\infty\}$ is the \emph{extended spectrum} of $P(\la)$ (or the full spectrum \cite{Rutkas2007}).

Note that the lineal $D_A$ equipped with the graph norm of the operator $A$, i.e., $\|x\|_{X_A}=\|x\|_X +\|Ax\|_Y$, forms the Banach space $X_A$, and $A$ is bounded as an operator from $X_A$ into $Y$. The operators $B$ and $P(\la)$ (for each $\la$) as the operators from $X_A$ into $Y$ with the domain $D$ are closed and, if $D_B\supseteq D_A$, are bounded.
The set of all regular points of $P(\la)$ is called the \emph{resolvent set} of the pencil.

The pencil \eqref{Pencil}, associated with the linear part $\frac{d}{dt}[Ax]+Bx$ of the equation \eqref{ADAE}, is called \emph{characteristic}. If the characteristic pencil $P(\la)$ is  regular (respectively, singular), then equation \eqref{ADAE} is called \emph{regular}  (respectively, \emph{singular}, or \emph{nonregular}, or \emph{irregular}).

  \subsection{Index of the pencil}\label{Sec_Pencil_ind}

\begin{definition}\label{Def_Pencil_ind}
Let the following conditions hold:
 \begin{enumerate}[label={\upshape(\alph*)},ref={\upshape(\alph*)}, itemsep=3pt,parsep=0pt,topsep=3pt,leftmargin=1cm]
\item\label{IndexP-1} The pencil $P(\la)=\la A+B$ is regular for all $\la$ from some neighborhood of the infinity, i.e., there exists a number $R>0$ such that each $\la$ with $|\la|>R$ is a regular point of $P(\la)$.

\item\label{IndexP-2} The point $\la=\infty$ is a pole of the resolvent $R(\la)$ of order $r$ or a removable singularity of $R(\la)$ (if the singularity is removable, we will say that it is a pole of order $0$). This is equivalent to the fact that the resolvent
    $$
  \widehat{R}(\mu)=(A+\mu B)^{-1}\quad \text{of the pencil}\quad \widehat{P}(\mu)=A+\mu B
    $$
    has a pole of order $\nu$ at the point $\mu=0$, where $\nu=r+1$ if $\la=\infty$ is the pole of $R(\la)$ of order $r$ and $\nu=1$ if $\la=\infty$ is a removable singularity.
 \end{enumerate}
Then $P(\la)$ is called a \emph{regular pencil of index $\nu$} ($\nu\in{\mathbb N}$).

 \smallskip
If there exists the inverse operator $A^{-1}\in\spL(Y,X)$ (or $\mu=0$ is a regular point of the pencil $A+\mu B$) and $D_B\supseteq D_A$, then $P(\la)$ is a \emph{regular pencil of index $0$}.
\end{definition}

As above, if the spaces $X$, $Y$ are real, then the analogous definitions of the index of $P(\la)$  are introduced by going over to the complex extensions $\tilde{A}$, $\tilde{B}$ of $A$, $B$ and the complexifications $\tilde{X}$, $\tilde{Y}$ of $X$, $Y$. In particular, the order $\nu=r+1$ of the pole of $(\tilde{A}+\mu \tilde{B})^{-1}$ at the point $\mu=0$ is the index of $P(\la)$.

 \smallskip
The above definition can be reformulated in the following way (cf. \cite{Fil.NCombMeth}). First, note that $\widehat{R}(\mu)=\mu^{-1}R(\mu^{-1})$.

Let condition \ref{IndexP-1} hold and $\nu\in{\mathbb N}$ be the least number such that for some constants $C,\, R>0$ the estimate
\begin{equation}\label{pencil_index}
\|R(\la)\|\le C|\la|^{\nu-1},\quad |\la|\ge R,
\end{equation}
or the equivalent estimate
\begin{equation}\label{pencil_indexInv}
\|\widehat{R}(\mu)\|\le \dfrac{C}{|\mu|^\nu},\quad 0<|\mu|\le\dfrac{1}{R},
\end{equation}
holds, then $P(\la)$ is a \emph{regular pencil of index $\nu$}.

Notice that for a regular pencil $P(\la)$ acting in finite-dimensional spaces, there is always a number $\nu\in{\mathbb N}$ for which the conditions \eqref{pencil_index} and \eqref{pencil_indexInv} are satisfied. The use of estimates \eqref{pencil_index} and \eqref{pencil_indexInv} for obtaining the direct decompositions \eqref{DYrr} (see below) and the projectors onto the subspaces from \eqref{DYrr}, as well as references to sources, are described in details in \cite{Vlasenko1}.

 \subsection{Direct decompositions of the operator domains and spaces and the associated projectors}\label{Sec_DirDec}

Let $P(\la)$ be a regular pencil of index $\nu$, i.e., the above conditions \ref{IndexP-1} and \ref{IndexP-2} (or \eqref{pencil_index}) hold. Then by \cite[Theorem 2.1]{Rutkas2007} or \cite[Section 2.3.1]{Vlasenko1},
there exists the pair of mutually complementary projectors $P_k\colon D\to D_k$ (by construction, $P_k$, $k=1,2$, have the domain $D_A$, $P_1D_A=P_1D$,\, $P_iP_jx= \delta_{ij}P_ix$, $(P_1+P_2)x=x$, $x\in D_A$) and the pair of mutually complementary projectors $Q_k\colon Y\to Y_k$ ($Q_iQ_j= \delta_{ij}Q_i$, $Q_1+Q_2=I_Y$), $k=1,2$, which generate the decompositions of $D$ and $Y$ into the direct sums
 \begin{equation}\label{DYrr}
D=D_1\dotpl D_2,\quad Y=Y_1\dotpl Y_2,\quad  D_k=P_kD,\quad Y_k=Q_kY,\quad k=1,2,
 \end{equation}
such that the pairs $\{D_k,Y_k\}$ are invariant under the operators  $A$, $B$, i.e., $A D_k\subset Y_k$ and $B D_k\subset Y_k$, $k=1,2$. The restricted operators
\begin{equation}\label{A_k,B_k}
A_k:=A\big|_{D_k}\colon D_k\to Y_k,\quad B_k:=B\big|_{D_k}\colon D_k\to Y_k, \quad {k=1,2},
\end{equation}
are such that there exist $A_1^{-1} \in \spL(Y_1,\overline{D_1})$ and $B_2^{-1}\in \spL(Y_2,\overline{D_2})$. The projectors can be constructively determined by using contour integration \cite{Rutkas2007,Vlasenko1} or residues \cite{Fil.NCombMeth,Fil.CombMeth}. Here $Q_k\in \spL(Y)$, $k=1,2$. In addition, if we use the norm  $\|\cdot\|_{X_A}$ of the space $X_A$ defined above, then $P_k\in \spL(X_A)$, $k=1,2$, and they generate the direct decomposition $X_A=D_1\dotpl P_2D_A$, where $D_1=P_1D=P_1D_A$ (see \cite{Rutkas2007}).

Recall that $A\in \spL(X_A,Y)$, and if $D_B\supseteq D_A$, then $B\in \spL(X_A,Y)$ as well. Generally, there is a method to go from equation \eqref{ADAE} with the closed operators to a similar equation with bounded operators without using the norm $\|\cdot\|_{X_A}$. Take a regular point $\la_*\in \varrho$ and make the change of variable $x=R(\la_*)v$ (cf. \cite{Rutkas2007}), then equation \eqref{ADAE} passes to the equation
 \begin{equation*}
\frac{d}{dt}[\widetilde{A}v]+\widetilde{B}v=\widetilde{f}(t,v),\qquad t>0,
 \end{equation*}
where $\widetilde{A}=AR(\la_*)\in \spL(Y)$, $\widetilde{B}=BR(\la_*)\in \spL(Y)$ and $\widetilde{f}(t,v)=f(t,R(\la_*)v)$.

If the pencil $P(\la)$ has index 1, the projectors mentioned above allow one to reduce equation \eqref{ADAE} to a system of an explicit ODE and an algebraic equation. However, for the pencil of index higher than 1, additional decompositions of the lineal $D_2$, the subspace $Y_2$ and the projectors $P_2$, $Q_2$ are required. Below we provide another method to construct the projectors.

In what follows, we suppose that $\overline{D}=X$. Denote $D_{20}=\ker A\cap D$. In particular, if $D_B\supseteq D_A$, then $\overline{D}=X$ and $D_{20}=\ker A$.

Assume that $\dim D_{20}=n$ and denote by $\{\varphi_1^1,...,\varphi_n^1\}$ a basis of $D_{20}$. Thus, ${D_{20}=\Span\{\varphi_i^1\}_{i=1,...,n}}$ is closed.

Note that if the operator $AR(\la_*)$ is completely continuous (compact) for some $\la_*\in \varrho$, then the resolvent $R(\la)$ is completely continuous for each $\la\in \varrho$ (since $R(\la)=R(\la_*)+(\la_*-\la)R(\la)AR(\la_*)$), and the same holds for the resolvent $\widehat{R}(\mu)$ \,($\mu=\la^{-1}$).
Therefore, in this case the number of linearly independent eigenvectors of the operator $\widehat{P}(\mu)$ that correspond to the eigenvalue $\mu=0$ is finite \cite{Keldysh1971} (recall that $\varphi\in D$ is an eigenvector of $\widehat{P}(\mu)$, associated with the eigenvalue $\mu=0$, if $\widehat{P}(0)\varphi=A\varphi=0$~), and hence $\dim D_{20}<\infty$.

Using \cite{Keldysh1971}, we obtain that there exists a \emph{canonical system $\{\varphi_i^j\}_{i=1,...,n}^{j=1,...,m_i}$} ($1\le m_i\le \nu$) \emph{of eigenvectors and adjoined vectors} (or \emph{root vectors}) \emph{of the pencil $\widehat{P}(\mu)=A+\mu B$ that correspond to the eigenvalue $\mu=0$}, for which the vectors $\varphi_i^j\in D$ satisfy the equalities
\begin{equation}\label{eavA}
A\varphi_i^1=0,\quad A\varphi_i^j=-B\varphi_i^{j-1},\quad i=1,...,n,\quad j=2,...,m_i,\quad \max\limits_{i=1,...,n}\{m_i\}=\nu \quad (\nu= r+1),
\end{equation}
where $\nu$ is the index of $P(\la)$.  The number $m_i-1$ is called the \emph{order} of the adjoined vector $\varphi_i^{m_i}$, and $m_i$ is called the \emph{multiplicity} of the eigenvector $\varphi_i^1$. If the eigenvector $\varphi_i^1$ does not have adjoined vectors (in this case, $\langle B\varphi_i^1,q_i^1\rangle=1$, where $q_i^1\in \ker A^*$),  then its multiplicity $m_i=1$ and the corresponding chain of root vectors consists only of $\varphi_i^1$. The vectors $\{\varphi_i^j\}_{i=1,...,n}^{j=1,...,m_i}$ and $\{B\varphi_i^j\}_{i=1,...,n}^{j=1,...,m_i}$ form the bases of $D_2$ and $Y_2$, respectively, and the set $D_2$ (note that $D_2$ is closed) is called the \emph{root subspace of $\widehat{P}(\mu)$ that correspond to} $\mu=0$, and $\sum_{i=1}^n m_i$ is the root number of $\widehat{P}(\mu)$ (cf. \cite{Gokhberg-Krein}). Thus, the projectors $P_k$, $Q_k$, $k=1,2$, can be obtained by the formulas:
\begin{equation}\label{rrProj}
P_2x=\sum\limits_{i=1}^n \sum\limits_{j=1}^{m_i}\langle x,p_i^j\rangle \varphi_i^j, \quad Q_2y=\sum\limits_{i=1}^n \sum\limits_{j=1}^{m_i}\langle y,q_i^j\rangle B\varphi_i^j, \quad P_1=I_X-P_2,\quad Q_1=I_Y-Q_2,
\end{equation}
where\, ${p_i^j=B^*q_i^j\in X^*}$,\, ${j=1,...,m_i}$, ${i=1,...,n}$, $x\in X$, $y\in Y$, and the bounded linear functionals $q_i^j\in D^*=D_{A^*}\cap D_{B^*}$ are chosen such that
\begin{equation}\label{eavA*}
A^*q_i^{m_i}=0,\quad A^*q_i^j=-B^*q_i^{j+1},\quad j=1,...,m_i-1,\quad i=1,...,n,
\end{equation}
and such that the systems $\{B\varphi_i^j\}_{i=1,...,n}^{j=1,...,m_i}$ and $\{q_i^j\}_{i=1,...,n}^{j=1,...,m_i}$ are biorthogonal, i.e.,
\begin{equation}\label{biorthogonal}
\langle B\varphi_i^j,q_k^l\rangle = \delta_{ik}\delta_{jl}.
\end{equation}

If $m_i=1$ (for some $i$), then \eqref{eavA*} becomes $A^*q_i^{m_i}=0$. Note that since $\langle B\varphi_i^{m_i},q_i^{m_i}\rangle=1$  \,($q_i^{m_i}\in \ker A^*$), then $B\varphi_i^{m_i}\not\in \range(A)$, $i=1,...,n$.
The formulas \eqref{rrProj} and relations \eqref{biorthogonal}, \eqref{eavA*}
follow from the required properties of the operators $P_k$, $Q_k$, $k=1,2$, namely, they are the pairs of mutually complementary projectors and $\{D_k,Y_k\}$, where $D_k=P_kD$ and $Y_k=Q_kY$, are invariant under $A$, $B$, i.e.,
$$
AP_kx=Q_kAx,\quad  BP_kx=Q_kBx,\quad  x\in D,\quad  k=1,2.
$$

Denote $\hat{q}_i^j=q_i^{m_i+1-j}$, then
$$
A^*\hat{q}_i^1=0,\quad A^*\hat{q}_i^j=-B^*\hat{q}_i^{j-1},\quad i=1,...,n,\; j=2,...,m_i,
$$
and \emph{$\{\hat{q}_i^j=q_i^{m_i+1-j}\}_{i=1,...,n}^{j=1,...,m_i}$ is a canonical system of eigenvectors and adjoined vectors of the pencil $A^*+\mu B^*$ that correspond to the eigenvalue ${\mu=0}$} (this system is determined uniquely when specifying the system ${\{\varphi_i^j\}_{i=1,...,n}^{j=1,...,m_i}}$) \cite{Keldysh1971}. Therefore, ${\{\hat{q}_i^1=q_i^{m_i}\}_{i=1,...,n}}$ forms the basis of $D_{20}^*= \ker A^* \cap D^*$ and hence ${\dim D_{20}^*=n}$. If $D_B\supseteq D_A$, then ${\{\hat{q}_i^1\}_{i=1,...,n}}$ is the basis of $\ker A^*$ and ${\dim\ker A^*=n}$.

The projectors \eqref{rrProj} generate the decompositions of $D$, $X$ and $Y$ into the direct sums
 \begin{equation}\label{XDYrr}
\begin{split}
 & D=D_1\dotpl D_2,\quad D_k=P_kD,\quad X=X_1\dotpl X_2,\quad X_k=P_kX, \quad D_2=X_2,\quad D_1=X_1\cap D, \\
 & Y=Y_1\dotpl Y_2,\quad Y_k=Q_kY\quad (Y_2=BD_2),\quad k=1,2,
\end{split}
 \end{equation}
where the direct decompositions of $D$ and $Y$ are the same as \eqref{DYrr} (obviously, $Q_k\in \spL(Y)$, $P_k\in \spL(X)$ and $P_k\in \spL(X_A)$, $k=1,2$).

For the pencil $P(\la)$ of index 2, a similar construction of the projectors was used in \cite{Rut-Khudoshin} with the reference to \cite{Khudoshin99}. If $D_B\supseteq D_A$, then, according to \cite{Vainberg-Trenogin}, the system $\{\varphi_i^j\}_{i=1,...,n}^{j=1,...,m_i}$ can be referred to as $(-B)$-Jordan set of the operator $A$.

Recall that ${D_{20}=\ker A\cap D=\Span\{\varphi_i^1\}_{i=1,...,n}}$.
Note that if $P(\la)$ has index 1, then $D_2=D_{20}$ and $Y_1=\range\big(A\big|_D\big)$ (if in addition $D_B\supseteq D_A$, then $D_2=\ker A$ and $Y_1=\range(A)$\,).

Define by ${D_{2s}=X_{2s}}$ the linear span of the adjoined vectors of order $s$, ${s=1,...,\nu-1}$ (${\nu= r+1}$), i.e.,
 $$  
{D_{2s}:=\Span\{\varphi_i^{s+1},\text{ where $i\in \{1,...,n\}$ are those indices for which } s+1\le m_i\le \nu\}},
 $$
and let ${Y_{2s}:=BD_{2s}}$, ${s=0,...,\nu-1}$. Then ${BD_{2(s-1)}\supseteq AD_{2s}}$, ${s=1,...,\nu-1}$.
Further, define by ${D_{20}^{(j)}=X_{20}^{(j)}}$ the linear span of the eigenvectors of multiplicity $j$ (${j=1,...,\nu}$) and by
${D_{2s}^{(j)}=X_{2s}^{(j)}}$ the linear span of the adjoined vectors of order~$s$ to which the eigenvectors of multiplicity~$j$ correspond (${j=s+1,...,\nu}$, ${s=1,...,\nu-1}$), i.e.,
 $$
D_{2s}^{(j)}:=\Span\{\varphi_i^{s+1},\text{ where $i\in\{1,...,n\}$ are those indices for which } m_i=j\},\, j=s+1,...,\nu,\, s=0,...,\nu-1.
 $$
Also, define $Y_{2s}^{(j)}:=BD_{2s}^{(j)}$, ${j=s+1,...,\nu}$, ${s=0,...,\nu-1}$. Then
$$
{D_2=\dotpl\limits_{s=0}^{\nu-1}D_{2s}},\quad {Y_2=\dotpl\limits_{s=0}^{\nu-1}Y_{2s}},\qquad
{D_{2s}=D_{2s}^{(s+1)}\dot{+} ...\dot{+} D_{2s}^{(\nu)}},\quad {Y_{2s}=Y_{2s}^{(s+1)}\dot{+} ...\dot{+} Y_{2s}^{(\nu)}},\quad {s=0,...,\nu-1}.
$$
Introduce the projectors ${P_{2s}\colon D\to D_{2s}=X_{2s}}$ (generally, ${P_{2s}\colon X\to X_{2s}}$) and ${Q_{2s}\colon Y\to Y_{2s}}$, ${s=0,...,{\nu-1}}$
 \,(obviously, ${P_{2i}P_{2j}=\delta_{ij}P_{2i}}$, ${Q_{2i}Q_{2j}=\delta_{ij}Q_{2i}}$, ${P_2=\sum\limits_{s=0}^{\nu-1} P_{2s}}$ and ${Q_2=\sum\limits_{s=0}^{\nu-1} Q_{2s}}$).
Also, introduce the projectors ${P_{2s}^{(j)}\colon D\to D_{2s}^{(j)}}$ (generally, $P_{2s}^{(j)}\colon X\to X_{2s}^{(j)}$) and ${Q_{2s}^{(j)}\colon Y\to Y_{2s}^{(j)}}$,  ${s=0,...,\nu-1}$, ${j=s+1,...,\nu}$
 \,(obviously, the projectors $P_{2s}^{(j)}$, as well as $Q_{2s}^{(l)}$, are pairwise disjoint, and ${P_{2s}=\sum\limits_{j=s+1}^{\nu} P_{2s}^{(j)}}$, ${Q_{2s}=\sum\limits_{j=s+1}^{\nu} Q_{2s}^{(j)}}$).
The introduced projectors can be obtained from \eqref{rrProj} as the appropriate partial sums, 
\begin{align*}
P_{2s} &=\sum\limits_{i\in \{1,...,n\}\,:\, m_i=s+1,...,\,\nu} \langle\cdot,B^*q_i^{s+1}\rangle \varphi_i^{s+1}=\sum\limits_{j=s+1}^{\nu} P_{2s}^{(j)}, \quad &
P_{2s}^{(j)} &=\sum\limits_{i\in \{1,...,n\}\,:\, m_i=j} \langle\cdot,B^*q_i^{s+1}\rangle \varphi_i^{s+1}, \\
Q_{2s} &=\sum\limits_{i\in \{1,...,n\}\,:\, m_i=s+1,...,\,\nu} \langle\cdot,q_i^{s+1}\rangle B\varphi_i^{s+1}=\sum\limits_{j=s+1}^{\nu} Q_{2s}^{(j)}, \quad &
Q_{2s}^{(j)} &=\sum\limits_{i\in \{1,...,n\}\,:\, m_i=j} \langle\cdot,q_i^{s+1}\rangle B\varphi_i^{s+1},
\end{align*}
$j=s+1,...,\nu$,\, $s=0,...,\nu-1$, and are such that ${Q_{2(\nu-1)}=Q_{2(\nu-1)}^{(\nu)}}$, ${P_{2(\nu-1)}=P_{2(\nu-1)}^{(\nu)}}$;\,
${Q_{2s}Bx=BP_{2s}x}$,\, ${Q_{2s}^{(j)}Bx=BP_{2s}^{(j)}x}$,\, ${j=s+1,...,\nu}$, ${s=0,...,\nu-1}$, ${x\in D_B}$;\,
${AP_{20}x=0}$,\, ${Q_{2s}Ax=AP_{2(s+1)}x}$, ${s=0,...,\nu-2}$,\, ${Q_{2(\nu-1)}Ax=0}$,\, ${x\in D}$.

Denote
 \begin{equation*}  \begin{split}
&D_{2\Sigma}=D_2\setminus D_{20}= \Span\{\varphi_i^j, \text{ where $i\in \{1,...,n\}$ are those indices for which $m_i\ne 1$,\, $j=2,...,m_i$ }\}, \\
& Y_{2*}=\Span\{B\varphi_i^{m_i},\; i=1,...,n\}, \\
&Y_{2\Sigma}=Y_2\setminus Y_{2*}= \Span\{B\varphi_i^j, \text{ where $i\in \{1,...,n\}$ are those indices for which $m_i\ne 1$,\, $j=1,...,m_i-1$} \}.
 \end{split}  \end{equation*}
Then $D_2$ and $Y_2$ can be represented as the direct sums
 \begin{equation}\label{DY2rr}
D_2=D_{2\Sigma}\dotpl D_{20},\quad  Y_2=Y_{2\Sigma}\dotpl Y_{2*},
 \end{equation}
and the restricted operators $A_{20}:=A\big|_{D_{20}}$ and ${A_{2\Sigma}:=A\big|_{D_{2\Sigma}}\colon D_{2\Sigma}\to Y_{2\Sigma}}$ are such that $A_{20}=0$ and there exists
$$
{A_{2\Sigma}^{-1}\in \spL(Y_{2\Sigma},D_{2\Sigma})}.
$$
Notice that ${D_{2\Sigma}=D_2\setminus D_{20}= \dotpl\limits_{s=1}^{\nu-1}D_{2s}}$,\,  ${Y_{2*}=\dotpl\limits_{s=0}^{\nu-1}Y_{2s}^{(s+1)}= \dotpl\limits_{s=0}^{\nu-1}BD_{2s}^{(s+1)}}$, \, ${Y_{2\Sigma}=\range(A_2)}$ and ${Y_1=\range(A_1)}$, where ${A_k=A\big|_{D_k}\colon D_k\to Y_k}$, $k=1,2$ (they were introduced in \eqref{A_k,B_k}), and, accordingly, $\range\big(A\big|_D\big)=Y_1\dotpl Y_{2\Sigma}$.
Here it is supposed that the pencil index ${\nu\ge 2}$. If ${\nu=1}$, then ${D_2=D_{20}= \Span\{\varphi_i^1\}_{i=1,...,n}}$,\, ${Y_2=B D_2}$ and ${Y_1=AD}$.
As mentioned above, the operators $A_1=A\big|_{D_1}$ and $B_2=B\big|_{D_2}$ have the inverses
  $$
A_1^{-1}\in \spL(Y_1,\overline{D_1}), \qquad  B_2^{-1}\in \spL(Y_2,D_2).
 $$
The direct decompositions \eqref{DY2rr} generate the pairs $P_{2\Sigma}$, $P_{20}$ and $Q_{2\Sigma}$, $Q_{2*}$ of the mutually complementary projectors $P_{2\Sigma}\colon D\to D_{2\Sigma}$, $P_{20}\colon D\to D_{20}$  \,(${P_{2\Sigma}+P_{20}=P_2}$) and $Q_{2\Sigma}\colon Y\to Y_{2\Sigma}$, $Q_{2*}\colon Y\to Y_{2*}$ \,(${Q_{2\Sigma}+Q_{2*}=Q_2}$),
which can be obtained from \eqref{rrProj} as appropriate partial sums, namely,
 \begin{align*}
P_{20} &=\sum\limits_{i=1}^n \langle \cdot,B^*q_i^1\rangle \varphi_i^1, &
P_{2\Sigma} &=\sum\limits_{i\in \{1,...,n\}\,:\, m_i\ne 1\,} \sum\limits_{j=2}^{m_i}\langle \cdot,B^*q_i^j\rangle \varphi_i^j= \sum\limits_{s=1}^{\nu-1}P_{2s},   \\
Q_{2*} &=\sum\limits_{i=1}^n \langle \cdot,q_i^{m_i}\rangle B\varphi_i^{m_i}= \sum\limits_{s=0}^{\nu-1}Q_{2s}^{(s+1)}, &
Q_{2\Sigma} &=\sum\limits_{i\in \{1,...,n\}\,:\, m_i\ne 1\,}  \sum\limits_{j=1}^{m_i-1}\langle \cdot,q_i^j\rangle B\varphi_i^j= \sum\limits_{s=0}^{\nu-2}\sum\limits_{l=s+2}^{\nu}Q_{2s}^{(l)},
 \end{align*}
and are such that
 $$
Q_{2\Sigma}Ax=AP_{2\Sigma}x,\quad Q_{2*}Ax=0=AP_{20}x,\quad  x\in D.
 $$
Generally, the projectors $P_{20}$, $P_{2\Sigma}$ are defined on $X$ (recall that $P_2\colon X\to X_2=D_2$): $P_{20}X=X_{20}=D_{20}$,  $P_{2\Sigma}X=X_{2\Sigma}=D_{2\Sigma}$ and $X_2=X_{20}\dotpl X_{2\Sigma}$.

It follows from the above that any $x\in X$ and any $y\in Y$ can be uniquely represented as
\begin{align}
&x=x_1+x_2, &  &x_2=x_{2\Sigma}+x_{20}, &  &x_i=P_ix,\; i=1,2,\quad x_{2\Sigma}=P_{2\Sigma}x,\quad  x_{20}=P_{20}x;  \label{xrr-ind}\\
&y=y_1+y_2,&   &y_2=y_{2\Sigma}+y_{2*}, &  &y_i=Q_iy,\; i=1,2,\quad y_{2\Sigma}=Q_{2\Sigma}y,\quad  y_{2*}=P_{2*}y.  \label{yrr-ind}
\end{align}
Using the projectors $Q_1$, $Q_{2\Sigma}$, $Q_{2*}$, we reduce the DAE \eqref{ADAE} to the equivalent system
 \begin{align}
&Q_1A\dot{x}_1+Q_1Bx_1=Q_1f(t,x), \label{ADAE_DE1} \\
&Q_{2\Sigma}A\dot{x}_{2\Sigma}+Q_{2\Sigma}B(x_{2\Sigma}+x_{20})=Q_{2\Sigma}f(t,x), \label{ADAE_DE2} \\
&Q_{2*}B(x_{2\Sigma}+x_{20})=Q_{2*}f(t,x), \label{ADAE_AE}
 \end{align}
where $x=x_1+x_{2\Sigma}+x_{20}\in D$ (see \eqref{xrr-ind}).
The system \eqref{ADAE_DE1}, \eqref{ADAE_DE2} is equivalent to the equation
  \begin{equation}\label{ADAE_DE1+2}
\frac{d}{dt}[A(x_1+x_{2\Sigma})]+(Q_1+Q_{2\Sigma})B(x_1+x_{2\Sigma}+x_{20})= (Q_1+Q_{2\Sigma})f(t,x_1+x_{2\Sigma}+x_{20}).
 \end{equation}

Denote
\begin{equation}\label{Q_2Sigma-P_2wedge}
\begin{split}
Q_{2\Sigma,s}=Q_{2s}-Q_{2s}^{(s+1)}= \sum\limits_{l=s+2}^{\nu}Q_{2s}^{(l)},\quad
P_{2\wedge,s}=P_{2s}-P_{2s}^{(s+1)}= \sum\limits_{l=s+2}^{\nu}P_{2s}^{(l)},\quad P_{2\wedge}=\sum\limits_{s=0}^{\nu-2}P_{2\wedge,s}  \\
Y_{2\Sigma,s}=Q_{2\Sigma,s}Y, \quad D_{2\wedge,s}=P_{2\wedge,s}D,\quad {s=0,...,\nu-2},\,
\end{split}
\end{equation}
then $Q_{2s}Ax=\big(Q_{2s}-Q_{2s}^{(s+1)}\big)Ax=Q_{2\Sigma,s}Ax=AP_{2(s+1)}x$, $Q_{2\Sigma,s}Bx=BP_{2\wedge,s}x$ and there exists $A_{2(s+1)}^{-1}\in \spL(Y_{2\Sigma,s},D_{2(s+1)})$, where $A_{2(s+1)}=A\big|_{D_{2(s+1)}}= Q_{2\Sigma,s}A\big|_{D_{2(s+1)}}$, ${s=0,...,\nu-2}$.
Obviously, ${Q_{2\Sigma}=\sum\limits_{s=0}^{\nu-2}Q_{2\Sigma,s}}$, $Q_{2\Sigma,s}Q_{2\Sigma,k}=\delta_{sk}Q_{2\Sigma,s}$, and
$P_{2\wedge}\colon D\to D_{2\wedge}$ (generally, $P_{2\wedge}\colon X\to {X_{2\wedge}=D_{2\wedge}}$, $P_{2\wedge,s}\colon X\to {X_{2\wedge,s}=D_{2\wedge,s}}$), $P_{2\wedge,s}P_{2\wedge,k}=\delta_{sk}P_{2\wedge,s}$.
Note that
$P_2= \sum\limits_{s=0}^{\nu-1}P_{2s}^{(s+1)}+ \sum\limits_{s=0}^{\nu-2}P_{2\wedge,s}$,\,
$P_{2\Sigma}= \sum\limits_{s=1}^{\nu-1}P_{2s}^{(s+1)}+ \sum\limits_{s=1}^{\nu-2}P_{2\wedge,s}$
(accordingly, ${D_{2\Sigma}=\dotpl\limits_{s=1}^{\nu-1}D_{2s}^{(s+1)}} \dotpl\dotpl\limits_{s=1}^{\nu-2}D_{2\wedge,s}$\,) and  $P_{2(\nu-1)}^{(\nu)}=P_{2(\nu-1)}$. The projectors generate the direct decompositions  ${Y_{2\Sigma}=\dotpl\limits_{s=0}^{\nu-2}Y_{2\Sigma,s}}$, ${D_{2\wedge}=\dotpl\limits_{s=0}^{\nu-2}D_{2\wedge,s}}$.

  \smallskip
Define
\begin{equation*}
Y_{2*}^{(1)}=Y_{20}^{(1)}, \quad
Y_{2*}^{(2)}=\dotpl\limits_{s=1}^{\nu-1}Y_{2s}^{(s+1)},\quad
D_{2\Sigma}^{(1)}=\dotpl\limits_{s=1}^{\nu-1}D_{2s}^{(s+1)}, \quad
D_{2\Sigma}^{(2)}=\dotpl\limits_{s=1}^{\nu-2}D_{2\wedge,s}= \dotpl\limits_{s=1}^{\nu-2}\dotpl\limits_{l=s+2}^{\nu}D_{2s}^{(l)},
\end{equation*}
then $Y_{2*}=Y_{2*}^{(1)}\dotpl Y_{2*}^{(2)}$ and $D_{2\Sigma}=D_{2\Sigma}^{(1)}\dotpl D_{2\Sigma}^{(2)}$.
The corresponding projectors $Q_{2*}^{(i)}\colon Y\to Y_{2*}^{(i)}$, ${P_{2\Sigma}^{(i)}\colon X\to D_{2\Sigma}^{(i)}}$,  $i=1,2$ (\,$Q_{2*}=Q_{2*}^{(1)}+Q_{2*}^{(2)}$, $P_{2\Sigma}=P_{2\Sigma}^{(1)}+P_{2\Sigma}^{(2)}$\,), take the form
\begin{equation}\label{Q_2*_(i),P_2Sigma_(i)}
Q_{2*}^{(1)}=Q_{20}^{(1)}, \quad
Q_{2*}^{(2)}=\sum\limits_{s=1}^{\nu-1}Q_{2s}^{(s+1)}, \quad P_{2\Sigma}^{(1)}=\sum\limits_{j=1}^{\nu-1}P_{2j}^{(j+1)},  \quad
P_{2\Sigma}^{(2)}=\sum\limits_{j=1}^{\nu-2} P_{2\wedge,j}= \sum\limits_{j=1}^{\nu-2} \sum\limits_{l=j+2}^{\nu} P_{2j}^{(l)}.
\end{equation}
Note that $Q_{2*}^{(1)}Bx=BP_{20}^{(1)}x$ and $Q_{2*}^{(2)}Bx=BP_{2\Sigma}^{(1)}x$, and recall that $P_{20}=P_{20}^{(1)}+ P_{2\wedge,0}$.

 \smallskip
We successively apply the projectors $Q_1$, $Q_{20}$, $Q_{2\Sigma,s}$, $s=1,...,\nu-2$, and $Q_{2s}^{(s+1)}$, $s=1,...,\nu-1$, to the ADAE \eqref{ADAE} and obtain the equivalent system (which has $2\nu-1$ equations and $2\nu-1$ unknowns)
\begin{align}
&[Q_1Ax]'+Q_1Bx=Q_1f(t,x), &\Leftrightarrow& \;
  [AP_1x]'+BP_1x=Q_1f(t,x),  \label{1_1}\\
&[Q_{20}Ax]'+Q_{20}Bx=Q_{20}f(t,x), &\Leftrightarrow& \;
  [AP_{21}x]'+BP_{20}x=Q_{20}f(t,x),  \label{2_20}\\
& \qquad  \vdots && \qquad \vdots \nonumber \\
&[Q_{2\Sigma,s}Ax]'\!+\!Q_{2\Sigma,s}Bx=Q_{2\Sigma,s}f(t,x),\!\!\! &\Leftrightarrow& \;
  [AP_{2(s+1)}x]'\!+\!B\big(P_{2s}\!-\!P_{2s}^{(s+1)}\big)x=Q_{2\Sigma,s}f(t,x),\; s=1,...,\nu\!-\!2 \label{2_2Sigma_s}\\
& \qquad  \vdots && \qquad \vdots
                         \nonumber \\
&Q_{2s}^{(s+1)}Bx=Q_{2s}^{(s+1)}f(t,x), &\Leftrightarrow& \;
  BP_{2s}^{(s+1)}x=Q_{2s}^{(s+1)}f(t,x),\; s=1,...,\nu-1,  \label{3_2*_s}\\
& \qquad  \vdots && \qquad \vdots  \nonumber
\end{align}
where a prime $'$ denotes the time derivative $\frac{d}{dt}$.
The system \eqref{1_1}--\eqref{3_2*_s} can be rewritten  as
\begin{align}
& \frac{d}{dt}[AP_1x]+BP_1x=Q_1f(t,x) 
\tag{\ref{1_1}}\\
& \frac{d}{dt}\big[A\big(P_{2\wedge,1}+P_{21}^{(2)}\big)x\big]+ BP_{20}\,x= Q_{20}f(t,x)  \tag{\ref{2_20}}\\
& \qquad\qquad  \vdots  \nonumber \\
& \frac{d}{dt}\big[A\big(P_{2\wedge,s+1}+P_{2(s+1)}^{(s+2)}\big)x\big] +BP_{2\wedge,s}\,x= Q_{2\Sigma,s}f(t,x), \quad  s=1,...,\nu-3, \label{2_2Sigma_s-}\\
& \qquad\qquad  \vdots  \nonumber \\
& \frac{d}{dt}\big[AP_{2(\nu-1)}x\big] +BP_{2\wedge,\nu-2}\,x= Q_{2\Sigma,\nu-2}f(t,x),  \label{2_2Sigma_nu-2}\\
& \qquad\qquad  \vdots  \nonumber \\
&BP_{2s}^{(s+1)}x=Q_{2s}^{(s+1)}f(t,x), \quad  s=1,...,\nu-2,  \label{3_2*_s-}\\
& \qquad\qquad  \vdots  \nonumber \\
&BP_{2(\nu-1)}x=Q_{2(\nu-1)}f(t,x).   \label{3_2*_nu-1}
\end{align}

  \section{Nonlocal (global) implicit function theorems}\label{Sec-Nonloc impl func}

In Section \ref{Sec-Main}, which contain the main results, we will consider equations of the type \eqref{F} presented below and find some arguments of the left-hand sides of the equations as implicit functions of others.
In order to simplify the proofs of the main results, we provide the theorems on the existence and uniqueness of a nonlocal (global) implicit function. They are based on the results from the previous papers of the author, namely from \cite{RF1,Fil.MPhAG,Fil.Sing-GN} (which use, in particular, the classical results of the theory of implicit functions).

Consider the equation
\begin{equation}\label{F}
F(t,x,y)=0,
\end{equation}
where $F\colon \mT\times D_X\times D_Y\to Z$, $\mT\subseteq \R$ is an interval, $D_X \subseteq X$ and $D_Y \subseteq Y$ are open sets and $X,\,Y,\,Z$ are Banach spaces.

 \begin{theorem}[existence and uniqueness of a nonlocal implicit function]\label{Th-Nonloc_ImplFunc}
Let $F\in C(\mT\times D_X\times D_Y,Z)$ and let there exist an open set $M_X\subseteq D_X$ and a set $M_Y\subseteq D_Y$ such that the following holds:
 \begin{enumerate}[label={\upshape\arabic*.}, ref={\upshape\arabic*}, itemsep=4pt,parsep=0pt,topsep=4pt,leftmargin=0.6cm]
\item\label{Sogl_F}
 For each ${t\in \mT}$, $x\in M_X$ there exists a unique ${y\in M_Y}$ such that $(t,x,y)$ satisfies \eqref{F}.

\item\label{Inv-Lipsch_F}
 A function $F(t,x,y)$ satisfies locally a Lipschitz condition with respect to $(x,y)$ on $\mT\times D_X\times D_Y$.  \\
 For any fixed $(t',x',y')\in \mT\times M_X\times M_Y$ satisfying \eqref{F}, there exist open neighborhoods
 $U_\delta(t',x')= U_{\delta_1}(t')\times U_{\delta_2}(x')\subset \mT\times D_X$ and $U_\varepsilon(y')\subset D_Y$ and an operator $W_{t',x',y'}\in \spL(Y,Z)$ having the inverse $W_{t',x',y'}^{-1}\in \spL(Z,Y)$ such that for each ${(t,x)\in U_\delta(t',x')}$ and each $y^1,\, y^2\in U_\varepsilon(y')$ the mapping $F$ satisfies the inequality
  \begin{equation}\label{Contractive_F}
 \big\|F(t,x,y^1)- F(t,x,y^2)- W_{t',x',y'}\big[y^1-y^2\big]\big\|\le k(\delta,\varepsilon)\big\|y^1-y^2\big\|,
   \end{equation}
 where $k(\delta,\varepsilon)$ is such that $\lim\limits_{\delta,\,\varepsilon\to 0} k(\delta,\varepsilon)< \|W_{t',x',y'}^{-1}\|^{-1}$  \,\textup{(}the numbers $\delta, \varepsilon>0$ depend on the choice of $t',x'$\textup{)}.
 \end{enumerate}
Then equation \eqref{F} has a unique solution $y=\eta(t,x)$, i.e., $F(t,x,\eta(t,x))=0$, for each $(t,x)\in \mT\times M_X$ and the function $\eta\in C(\mT\times M_X,M_Y)$ satisfies locally a Lipschitz condition with respect to $x$ on $\mT\times M_X$.
 \end{theorem}

 \begin{proof}
First, we need the following lemma.

 \begin{lemma}\label{Lem-ImplicitFunc}
For any fixed $(t',x',y')\in \mT\times M_X\times M_Y$ satisfying \eqref{F}, there exist open neighborhoods
$U_r(t',x')=U_{r_1}(t')\times U_{r_2}(x')\subset U_\delta(t',x')$ and $U_\rho(y')\subset U_\varepsilon(y')$ \,$($\,$r,\rho>0$ depend on the choice of $t',\,x',\,y'$\,$)$ and a unique function $\upsilon\in C\big(U_r(t',x'),U_\rho(y')\big)$ which is a solution of equation \eqref{F} with respect to $y$, i.e., $F(t,x,\upsilon(t,x))=0$, for each $(t,x)\in U_r(t',x')$ and satisfies the equality $\upsilon(t',x')=y'$ and a Lipschitz condition with respect to $x$ on $U_r(t',x')$ .
 \end{lemma}

 \begin{proof}[Proof of Lemma \ref{Lem-ImplicitFunc}]
The proof of the existence of a unique solution of equation \eqref{F} with respect to $y$ in a certain neighbourhood of a point $(t',x',y')$ satisfying the equation is similar to the proof of the implicit function theorem without the requirement of differentiability \cite[p.~420]{Trenogin}. The Lipschitz continuity of the implicit function with respect to $x$ follows from the local Lipschitz continuity of $F$ with respect to $(x,y)$ and the contractivity of the mapping $\widehat{F}$ defined in \eqref{Lem-F} below. Note that $U_{r_1}(t')$ can be a semi-open interval, as indicated in the introduction. In general, the idea of the proof is the same as for \cite[Lemma 3.1]{Fil.Sing-GN}. However, since in \cite{Fil.Sing-GN} nonregular (singular) semilinear DAEs in finite-dimensional spaces were considered, then the form of components of the phase variable $x$, as well as direct decompositions of spaces and the associated projectors, is different from that of the present paper. The proof of \cite[Lemma 3.1]{Fil.Sing-GN} is quite cumbersome due to the larger number of arguments of an implicit function and some features of the problem, so it will be easier to follow the proof of the lemma given above.
Thus, for clarity, we provide the complete proof of the present lemma.

Rewrite equation \eqref{F} as
 \begin{equation}\label{Lem-F}
y=\widehat{F}(t,x,y),\qquad  \widehat{F}(t,x,y):=y-W_{t',x',y'}^{-1}F(t,x,y),
 \end{equation}
where $(t,x)\in U_\delta(t',x')$, $y\in U_\varepsilon(y')$ and $W_{t',x',y'}$, $U_\delta(t',x')$, $U_\varepsilon(y')$ are defined in condition \ref{Inv-Lipsch_F}.

Take an arbitrary fixed point $(t',x',y')\in \mT\times M_X\times M_Y$ satisfying \eqref{F}.  Due to \eqref{Contractive_F}, there exists a number $\tilde{\delta}\in (0,\delta)$ and, accordingly, $\tilde{\delta_i}\in (0,\delta_i)$, $i=1,2$, and a number $\tilde{\varepsilon}\in (0,\varepsilon)$ such that for every $(t,x)\in \overline{U_{\tilde{\delta}}}(t',x') = \overline{U_{\tilde{\delta_1}}}(t')\times \overline{U_{\tilde{\delta_2}}}(x') \subset U_\delta(t',x')$
and every $y^1,\, y^2\in \overline{U_{\tilde{\varepsilon}}}(y')\subset U_\varepsilon(y')$ the following holds:
$\|\widehat{F}(t,x,y^1)- \widehat{F}(t,x,y^2)\|= \|y^1-y^2-W_{t',x',y'}^{-1} \big[F(t,x,y^1)-F(t,x,y^2)\big]\|\le k(\delta_0,\varepsilon_0) \|W_{t',x',y'}^{-1}\|\,\|y^1-y^2\|$, where $k(\delta_0,\varepsilon_0)\|W_{t',x',y'}^{-1}\|\le C<1$ for every $\delta_0\in (0,\tilde{\delta}]$, $\varepsilon_0\in (0,\tilde{\varepsilon}]$.
Hence, for any fixed point $(t',x',y')\in \mT\times M_X\times M_Y$ satisfying \eqref{F} there exist  (closed) neighborhoods $\overline{U_{\tilde{\delta}}}(t',x')$ and $\overline{U_{\tilde{\varepsilon}}}(y')$  \,($\tilde{\delta}\in (0,\delta)$, $\tilde{\varepsilon}\in (0,\varepsilon)$) such that $\widehat{F}$ is a contractive mapping with respect to $y$, uniformly in $(t,x)$, on $\overline{U_{\tilde{\delta}}}(t',x')\times \overline{U_{\tilde{\varepsilon}}}(y')$, i.e.,
 \begin{equation}\label{ContractiveMap-hatF}
\|\widehat{F}(t,x,y^1)- \widehat{F}(t,x,y^2)\|\le C\|y^1-y^2\|,\quad \text{where $C<1$ is a constant,}
 \end{equation}
for every $(t,x)\in \overline{U_{\tilde{\delta}}}(t',x')$ and every $y^i\in \overline{U_{\tilde{\varepsilon}}}(y')$, $i=1,2$.

Again, we take an arbitrary fixed $(t',x',y')\in \mT\times M_X\times M_Y$ satisfying \eqref{F}.
Since $F\in C(\mT\times D_X\times D_Y,Z)$, then $F(t,x,y')\to F(t',x',y')=0$ as $(t,x)\to (t',x')$, and therefore there exists a closed neighborhood $\overline{U_{\delta_*}}(t',x')= \overline{U_{\delta_1^*}}(t')\times \overline{U_{\delta_2^*}}(x') \subseteq \overline{U_{\tilde{\delta}}}(t',x')$, where
$\delta_*\in (0,\tilde{\delta}]$, $\delta_i^*\in (0,\tilde{\delta_i}]$, $i=1,2$, such that $\|W_{t',x',y'}^{-1}\|\, \|F(t,x,y')\|\le (1-C)\tilde{\varepsilon}$ \,(where $C$ is the constant from \eqref{ContractiveMap-hatF}) for every $(t,x)\in \overline{U_{\delta_*}}(t',x')$. Consequently, for each (fixed) $(t,x)\in \overline{U_{\delta_*}}(t',x')$ and every $y\in\overline{U_{\tilde{\varepsilon}}}(y')$ we have $\|\widehat{F}(t,x,y)-y'\|\le \|\widehat{F}(t,x,y)-\widehat{F}(t,x,y')\|+ \|W_{t',x',y'}^{-1}\|\, \|F(t,x,y')\| \le C\tilde{\varepsilon} +(1-C)\tilde{\varepsilon}= \tilde{\varepsilon}$.
Hence, for each (fixed) $(t,x)\in \overline{U_{\delta_*}}(t',x')$, \,$\widehat{F}(t,x,y)$ maps  $\overline{U_{\tilde{\varepsilon}}}(y')$ into itself.

It follows from the foregoing that, by the Banach's fixed point theorem and the theorem on continuous dependence of a fixed point on a parameter (see, e.g., \cite[Theorems 46, $46_2$]{Schwartz1}), the mapping $\widehat{F}(t,x,y)$ as a function of $y$, depending on the parameter $(t,x)$, has a unique fixed point $\upsilon_{(t,x)}=\upsilon(t,x)$ (i.e., $\widehat{F}(t,x,\upsilon(t,x))=\upsilon(t,x)$)  in $\overline{U_{\tilde{\varepsilon}}}(y')$ for each $(t,x)\in \overline{U_{\delta_*}}(t',x')$, which satisfies the equality $\upsilon(t',x')=y'$ and depends continuously on $(t,x)$. The continuity of the function $\upsilon\colon \overline{U_{\delta_*}}(t',x')\to \overline{U_{\tilde{\varepsilon}}}(y')$ is proved in the same way as in \cite[Theorem $46_2$]{Schwartz1}.

Let us prove that $\upsilon(t,x)$ satisfies a Lipschitz condition with respect to $x$ on $U_r(t',x')$, where the neighborhood $U_r(t',x')$ is specified below.
It follows from condition \ref{Inv-Lipsch_F} that $F(t,x,y)$ satisfies locally a Lipschitz condition with respect to $(x,y)$ on $\mT\times D_X\times D_Y$. Therefore, there exists an open neighborhood $U(t',x',y')=\widehat{U}(t')\times \widetilde{U}(x',y')$ and a constant ${L\ge 0}$ such that for any $(t,x^i,y^i)\in U(t',x',y')$, $i=1,2$, the following holds:
 \begin{equation}\label{LocLip_f}
\|F(t,x^1,y^1)- F(t,x^2,y^2)\|\le L\|(x^1,y^1)-(x^2,y^2)\|.
 \end{equation}
Choose numbers ${r\in (0,\delta_*]}$, ${r_i\in (0,\delta_i^*]}$, ${i=1,2}$, and a number ${\rho\in (0,\tilde{\varepsilon}]}$ so that $U_r(t',x')=U_{r_1}(t')\times U_{r_2}(x')\subset \overline{U_{\delta_*}}(t',x')$,   \,${U_\rho(y')\subset \overline{U_{\tilde{\varepsilon}}}(y')}$,   \,$U_r(t',x')\times U_\rho(y')\subseteq U(t',x',y')$
and $\upsilon\colon U_r(t',x')\to U_\rho(y')$.
Then, carrying out certain transformations and using \eqref{ContractiveMap-hatF}, \eqref{LocLip_f}, we obtain that the relation
$\|\upsilon(t,x^1)-\upsilon(t,x^2)\|\le C\|\upsilon(t,x^1)-\upsilon(t,x^2)\|+L\,\|W_{t',x',y'}^{-1}\|\, \|x^1-x^2\|$ and, hence, the relation
 \begin{equation*}
\|\upsilon(t,x^1)-\upsilon(t,x^2)\|\le \widehat{L}\,\|x^1-x^2\|, \qquad \widehat{L}=L\,\|W_{t',x',y'}^{-1}\|/(1-C)\ge 0,
 \end{equation*}
hold for any $(t,x^i)\in U_r(t',x')$, ${i=1,2}$. Hence, $\upsilon(t,x)$ satisfies a Lipschitz condition with respect to $x$ on $U_r(t',x')$.
 \end{proof}

It follows from Lemma~\ref{Lem-ImplicitFunc} (we also take into account condition \ref{Sogl_F}\,) that in some open neighborhood $U_r(t',x')=U_{r_1}(t')\times U_{r_2}(x')$ of each (fixed) point $(t',x')\in \mT\times M_X$ equation \eqref{F} has a unique solution ${y=\upsilon_{t',x'}(t,x)}$ such that $\upsilon_{t',x'}\in C\big(U_r(t',x'), U_\rho(y')\big)$, where $U_\rho(y')$ is an open neighborhood of the point $y'\in M_Y$ such that $(t',x',y')$ satisfies \eqref{F}, and it satisfies a Lipschitz condition with respect to $x$ and the equality $\upsilon_{t',x'}(t',x')=y'$.
Further, we define the function
\begin{equation}\label{eta_y}
\eta\colon \mT\times M_X\to M_Y
\end{equation}
by $\eta(t,x)= \upsilon_{t',x'}(t,x)$ at the point \,$(t,x)=(t',x')$ for each point $(t',x')\in \mT\times M_X$. Then the function $\eta\in C(\mT\times M_X,M_Y)$ satisfies locally a Lipschitz condition with respect to $x$ on $\mT\times M_X$ and is a unique solution of equation \eqref{F} with respect to $y$ for each $(t,x)\in \mT\times M_X$.  The uniqueness of the solution follows from condition~\ref{Sogl_set}.
Indeed, let us assume that there exists another function $\sigma(t,x)$ which is defined in the same way as $\eta$ and, accordingly, has the same properties as $\eta$, but differs from it at some point $(t',x')\in \mT\times M_X$. Then $\sigma(t',x')=\eta(t',x')=y'$ due to condition~\ref{Sogl_set}, which contradicts the assumption. This holds for each $(t',x')\in \mT\times M_X$, and therefore $\eta(t,x)\equiv\sigma(t,x)$.

In general, the method for constructing the nonlocal implicit function \eqref{eta_y} and proving its uniqueness is similar to that used in \cite{RF1,Fil.MPhAG}, where the global implicit function is constructed in finite-dimensional spaces and require differentiability with respect to the phase variable instead of local Lipschitz continuity.
\end{proof}

 \begin{theorem}[existence and uniqueness of a nonlocal implicit function]\label{Th-Nonloc_ImplFunc-2}
Let $F\in C(\mT\times D_X\times D_Y,Z)$ and let there exist an open set $M_X\subseteq D_X$ and a set $M_Y\subseteq D_Y$ such that condition \ref{Sogl_F} of Theorem~\ref{Th-Nonloc_ImplFunc} and the following condition hold:
\begin{enumerate}[label={\upshape\arabic*.}, ref={\upshape\arabic*}, itemsep=4pt,parsep=0pt,topsep=4pt,leftmargin=0.6cm]
\addtocounter{enumi}{1}
\item\label{InvReg_F}
  A function $F(t,x,y)$  has continuous (strong) partial derivatives with respect to $x,\, y$ on $\mT\times D_X\times D_Y$.  \\
  For any fixed $(t',x',y')\in \mT\times M_X\times M_Y$ satisfying \eqref{F}, the operator
  $W_{t',x',y'}:=\partial_{y}F(t',x',y')\in \spL(Y,Z)$
  has the inverse $W_{t',x',y'}^{-1}\in \spL(Z,Y)$.
\end{enumerate}
Then equation \eqref{F} has a unique solution $y=\eta(t,x)$, i.e., $F(t,x,\eta(t,x))=0$, for each $(t,x)\in \mT\times M_X$ and the function $\eta\in C(\mT\times M_X,M_Y)$ is continuously differentiable with respect to $x$ on $\mT\times M_X$.
 \end{theorem}

\begin{proof}
The ideas for the proof of the given theorem (in particular, the method for constructing a nonlocal implicit function has been already mentioned in the proof of Theorem~\ref{Th-Nonloc_ImplFunc}) are similar to those used in the papers \cite{RF1,Fil.MPhAG} where regular DAEs in finite-dimensional spaces have been considered. Therefore, we will give a brief proof of the present theorem.

Choose an arbitrary fixed point $(t',x')\in \mT\times M_X$. Then, due to conditions \ref{Sogl_F} and \ref{InvReg_F}, there exists a unique ${y'\in M_Y}$ such that $(t',x',y')$ satisfies \eqref{F} and the operator $\partial_{y}F(t',x',y')$ has the bounded inverse $[\partial_{y}F(t',x',y')]^{-1}\in \spL(Z,Y)$. Using the theorems on the existence and differentiability of an implicit function \cite[Theorems~25,~28]{Schwartz1}, we obtain that there exist open neighborhoods $U_{\delta_1}(t')$, $U_{\delta_2}(x')$ and $U_\varepsilon(y')$ and a unique function $\upsilon_{t',x'}\in C\big(U_{\delta_1}(t')\times U_{\delta_2}(x'),U_\varepsilon(y')\big)$ such that $\upsilon_{t',x'}(t',x')=y'$, $\upsilon_{t',x'}(t,x)$ is a solution of equation \eqref{F} with respect to $y$, and it is continuously differentiable with respect to $x$ on $U_{\delta_1}(t')\times U_{\delta_2}(x')$.
Recall that $U_{\delta_1}(t')$ can be a semi-open interval in $\mT$. Generally, the implicit function theorems \cite{Schwartz1} assume that the set of variables is open, and therefore, to prove the existence of an implicit function when the interval $U_{\delta_1}(t')$ is semi-open, the fixed point theorems \cite[Theorems 46, $46_2$]{Schwartz1} as well as the proofs of the implicit function theorems \cite[Theorems~25,~28]{Schwartz1} (which in turn use the same fixed point theorems) are used.

Now, in the same way as in the proof of Theorem \ref{Th-Nonloc_ImplFunc}, we define the function $\eta\colon \mT\times M_X\to M_Y$ by $\eta(t,x)= \upsilon_{t',x'}(t,x)$ at the point \,$(t,x)=(t',x')$ for each point $(t',x')\in \mT\times M_X$, and obtain that $\eta(t,x)$ is continuous in $(t,x)$, continuously differentiable in $x$ and is a unique solution of equation \eqref{F} with respect to $y$ for any $(t,x)\in \mT\times M_X$.
\end{proof}

\begin{corollary}\label{Cor_Nonloc_ImplFunc-2}
Theorem~\ref{Th-Nonloc_ImplFunc} remains valid if condition \ref{Inv-Lipsch_F} of this theorem is replaced by condition \ref{InvReg_F} of Theorem~\ref{Th-Nonloc_ImplFunc-2}.
\end{corollary}
\begin{proof}
The corollary follows directly from Theorem~\ref{Th-Nonloc_ImplFunc-2}, since the continuous differentiability of $\eta$ in $x$ implies its local Lipschitz continuity with respect to $x$.
\end{proof}

\begin{remark}\label{Rem__Nonloc_ImplFunc-2}
It is easy to verify that if condition \ref{InvReg_F} of Theorem~\ref{Th-Nonloc_ImplFunc-2} holds, then  condition \ref{Inv-Lipsch_F} of Theorem~\ref{Th-Nonloc_ImplFunc} holds as well. This gives another way of proving Corollary \ref{Cor_Nonloc_ImplFunc-2}.

Let us show this. Since $F(t,x,y)$ has a continuous partial derivatives with respect to $(x,y)$ on $\mT\times D_X\times D_Y$, then it is locally Lipschitz continuous with respect to $(x,y)$. Choose an arbitrary fixed point $(t',x',y')\in \mT\times M_X\times M_Y$ at which equation \eqref{F} holds, and take the operator $W_{t',x',y'}=\partial_{y}F(t',x',y')$ as the operator $W_{t',x',y'}$ appearing in condition \ref{Inv-Lipsch_F} of Theorem~\ref{Th-Nonloc_ImplFunc}. Then, due to condition~\ref{InvReg_F} of Theorem~\ref{Th-Nonloc_ImplFunc-2}, there exist open neighborhoods $U_\delta(t',x')\subset \mT\times D_X$ and $U_\varepsilon(y')\subset D_Y$ such that for each $(t,x)\in \overline{U_\delta(t',x')}$, $y^1,\,y^2\in \overline{U_\varepsilon(y')}$ the following holds:
$\big\|F(t,x,y^1)- F(t,x,y^2)- W_{t',x',y'} \big[y^1-y^2\big]\big\| \le\int\limits_0^1\big\|\partial_{y}F\big(t,x,y^2+\theta (y^1-y^2)\big) - W_{t',x',y'}\big\| d\theta\, \|y^1-y^2\|\le k(\delta,\varepsilon)\|y^1-y^2\|$,
where  $k(\delta,\varepsilon)=\sup\limits_{(t,x)\in \overline{U_\delta(t',x')},\, y''\in \overline{U_\varepsilon(y')}} \big\|\partial_{y}F(t,x,y'') -W_{t',x',y'}\big\|\to 0$
as $\delta,\varepsilon\to 0$ since $\partial_{y}F$ is continuous at $(t',x',y')$ (cf. \cite[p.~422, the proof of the theorem]{Trenogin}). Hence, taking into account that $W_{t',x',y'}$ has a bounded inverse, condition~\ref{Inv-Lipsch_F} of Theorem~\ref{Th-Nonloc_ImplFunc} holds.
\end{remark}

\begin{theorem}[existence, uniqueness and differentiability of a nonlocal implicit function]\label{Th_Nonloc_ImplFunc-3}
Let conditions of Theorem~\ref{Th-Nonloc_ImplFunc-2} hold and, in addition, let $F\in C^n(\mT\times D_X\times D_Y,Z)$ and $M_Y$ be open. Then the solution $y=\eta(t,x)$, defined in Theorem~\ref{Th-Nonloc_ImplFunc-2}, belongs to $C^n(\mT\times M_X,M_Y)$.
\end{theorem}
\begin{proof}
The proof is similar to the proof of Theorem~\ref{Th-Nonloc_ImplFunc-2}, except that the theorem on higher derivatives of an implicit function \cite[Theorem~31]{Schwartz1} is also used here.
\end{proof}

Another nonlocal (global) implicit function theorems have been presented, e.g., in  \cite{AZM,Rheinboldt69,Clarke} (see also references therein). In \cite{AZM}, a nonlocal implicit function theorem for a closed mapping with a parameter from one Asplund space to another has been obtained in terms of the regular coderivative of the mapping.
In \cite{Rheinboldt69}, nonlocal implicit function theorems for mappings in Banach spaces has been obtained by using the concept of ``continuation property'' and local mapping relations. 
For mappings in real finite-dimensional spaces, a nonlocal implicit function theorem is presented in \cite{Clarke} with the use of the maximal rank condition.

As mentioned above, in \cite{RF1,Fil.MPhAG,Fil.Sing-GN} conditions of the existence and uniqueness of global implicit functions for mappings in finite-dimensional spaces have been obtained, but they were not presented in the form of separate theorems since the aims of these papers were to study DAEs and the results were formulated in the form of theorems immediately for the DAEs.

 \section{Global solvability, Lagrange stability and instability of the ADAEs}\label{Sec-Main}

 \subsection{First approach}\label{Approach_1}

\emph{In this subsection, we consider the case when the component $x_{20}=P_{20}x$ can be defined as an implicit function from equation \eqref{AE} given below} (i.e., when the projection $Q_{2*}\big[f(t,x)-B(x_{2\Sigma}+x_{20})\big]$ depends on $x_{20}$; for clarity, notice that $Q_{2*}Bx_{20}=Bx_{20}^{(1)}$ where $x_{20}^{(1)}=P_{20}^{(1)}x$). \emph{In the case when the pencil $P(\la)$ has index 1, this always holds,} since $D_{20}=D_2$, $x_{20}=x_2=P_2x$, $x_{2\Sigma}=0$, $Q_{2*}=Q_2$  ($P_{2\Sigma}=0$, $Q_{2\Sigma}=0$) and $B_2=(Q_2B)\big|_{D_2}$ has the inverse  $B_2^{-1}\in \spL(Y_2,D_2)$.  

Recall that the representation of $x$ in the form ${x=x_1+x_{2\Sigma}+x_{20}}$ (see \eqref{xrr-ind}) is uniquely determined for each $x\in X$.

\begin{remark}\label{Rem_A_D}
In what follows, we assume that $D_B\supseteq D_A$ (then $D=D_A$), but this condition is not necessary (we make this assumption to simplify the presentation of further results). If we do not require that $D_B\supseteq D_A$, then it is necessary to require that $\overline{D}=X$ and to replace $A$ by the restricted operator $A\big|_D\colon D\to Y$ and, accordingly, the semi-inverse operator $A^{(-1)}$ by the semi-inverse $\big(A\big|_D\big)^{(-1)}$ below (a semi-inverse operator \cite{Faddeev} is called an algebraic generalized inverse and, in Banach spaces, a topological generalized inverse in \cite{Nashed76}). Then the results of Section \ref{Sec-Main} remain valid.
\end{remark}

The system \eqref{ADAE_DE1+2}, \eqref{ADAE_AE} can be rewritten as
 \begin{align}
&\dot{x}_1+\dot{x}_{2\Sigma}=A^{(-1)}\big[f(t,x_1+x_{2\Sigma}+x_{20})- B(x_1+x_{2\Sigma}+x_{20})\big], \label{DE1+2} \\
&0=Q_{2*}\big[f(t,x_1+x_{2\Sigma}+x_{20})-B(x_{2\Sigma}+x_{20})\big], \label{AE}
 \end{align}
or
 \begin{align}
& \frac{d}{dt}[x_1+x_{2\Sigma}]=\Pi(t,x), \quad \Pi(t,x):=\widetilde{A}^{-1}(Q_1+Q_{2\Sigma})\big[f(t,x)-Bx\big],  \label{ADAE_DE_Pi} \\
& F_{2*}(t,x_1,x_{2\Sigma},x_{20})=0, \quad
  F_{2*}(t,x_1,x_{2\Sigma},x_{20}):=Q_{2*}f(t,x)-Q_{2*}B(x_{2\Sigma}+x_{20}),  \label{ADAE_AE_F2*}
 \end{align}
where the operator $\widetilde{A}^{-1}\in \spL(Y,X)$ is the inverse of
\begin{equation}\label{widetilde_A}
{\widetilde{A}=A+\sum\limits_{i=1}^n \langle \cdot,B^*q_i^1\rangle B\varphi_i^{m_i}\colon D\to Y},
\end{equation}
and $A^{(-1)}\in \spL(Y,X)$ is the semi-inverse operator to $A$, which can be obtained by ${A^{(-1)}=\widetilde{A}^{-1}(Q_1+Q_{2\Sigma})}$ (or by using the relations mentioned below).
The existence and continuity of $\widetilde{A}^{-1}$ is proved in a similar way as the Schmidt lemma (see \cite[p.~232]{Trenogin} or \cite{Vainberg-Trenogin}). It is readily verified that the operator $\widetilde{A}^{-1}=A^{(-1)}+\sum\limits_{i=1}^n \langle \cdot,q_i^{m_i}\rangle \varphi_i^1$ is the inverse of $\widetilde{A}$. Notice that $Y_1=\widetilde{A}D_1=AD_1$,\;  $Y_{2\Sigma}=\widetilde{A}D_{2\Sigma}=AD_{2\Sigma}$,\; $\widetilde{A}^{-1}Ax= (P_1+P_{2\Sigma})x$,  $x\in D$, and $A\widetilde{A}^{-1}y= (Q_1+Q_{2\Sigma})y$, $y\in Y$.
In addition, $Q_1\widetilde{A}x=Q_1Ax=AP_1x=\widetilde{A}P_1x$ and $Q_{2\Sigma}\widetilde{A}x=Q_{2\Sigma}Ax=AP_{2\Sigma}x=\widetilde{A}P_{2\Sigma}x$, $x\in D$.

The semi-inverse operator $A^{(-1)}$ is defined by the relations
 \begin{equation}\label{Semi-inverse}
\begin{split}
& A^{(-1)}A=(P_1+P_{2\Sigma})\big|_{D_A},\quad
AA^{(-1)}=Q_1+Q_{2\Sigma},  \\
& A^{(-1)}=(P_1+P_{2\Sigma}) A^{(-1)}\quad (\text{or } A^{(-1)}= A^{(-1)}(Q_1+Q_{2\Sigma})\,).
\end{split}
 \end{equation}
Note that from these relations the equalities $AA^{(-1)}A=A$ and $A^{(-1)}AA^{(-1)}=A^{(-1)}$ follows, and the converse is also true with respect to the direct decompositions \eqref{XDYrr}, \eqref{DY2rr} and the corresponding projectors.

 \smallskip
\emph{For the regular pencil $P(\la)$ of index 1, the system  \eqref{ADAE_DE_Pi}, \eqref{ADAE_AE_F2*} takes the form}
 \begin{align}
&\dot{x}_1=\widetilde{A}^{-1}Q_1\big[f(t,x_1+x_2)- Bx_1\big], \label{DE_ind-1} \\
&0=\widetilde{A}^{-1}Q_2f(t,x_1+x_2)-x_2, \label{AE_ind-1}
 \end{align}
where $\widetilde{A}=A+BP_2=Q_1A+Q_2B$ coincide with the operator  $G=Q_1A+Q_2B$ from \cite[Subsection~3.3]{Vlasenko1} which has the inverse $G^{-1}=\widetilde{A}^{-1}\in \spL(Y,X)$.

 \smallskip
We will consider the initial value problem (IVP) for equation \eqref{ADAE}, and, accordingly, for the system \eqref{ADAE_DE_Pi}, \eqref{ADAE_AE_F2*}, with the initial condition
 \begin{equation}\label{ini}
x(t_0)=x_0.
 \end{equation}

The ``algebraic part'' \eqref{AE} of the ADAE generates the manifold
 \begin{equation}\label{L_0reg}
L_0=\{(t,x)\in \R_+\times D \mid Q_{2*}[f(t,x)-Bx]=0\} = \{(t,x)\in \R_+\times D \mid \text{$(t,x)$ satisfies \eqref{AE} or \eqref{ADAE_AE_F2*}}\}.
 \end{equation}
It is clear that the graph of a solution (i.e., the set of points $(t,x(t))$ where $t$ from the domain of definition of the solution $x(t)$) of the IVP \eqref{ADAE}, \eqref{ini} must lie in the manifold \eqref{L_0reg}.

The initial condition for the system \eqref{ADAE_DE_Pi}, \eqref{ADAE_AE_F2*} (accordingly, for the DAE \eqref{ADAE}) can also be considered in the form
  \begin{equation*}
(P_1+P_{2\Sigma})x(t_0)=x_1(t_0)+x_{2\Sigma}(t_0)=\omega_0
 \end{equation*}
or $Ax(t_0)=w_0$ (then $(P_1+P_{2\Sigma})x(t_0)=A^{(-1)}w_0= \widetilde{A}^{-1}w_0$\,).
However, the point $(t_0,x_0)$, where $x_0=\omega_0+P_{20}x(t_0)$ if the initial condition has the form $(P_1+P_{2\Sigma})x(t_0)=\omega_0$, must belong to $L_0$ regardless of its specification.

\begin{remark}\label{Rem_Norms}
Note that the spaces $X_1\dotpl X_{2\Sigma}\dotpl X_{20}$ and $X_1\times X_{2\Sigma}\times X_{20}$ are isomorphic (there exists a one-to-one correspondence between an element $x=x_1+x_{2\Sigma}+x_{20}$ and $x=(x_1,x_{2\Sigma},x_{20})$) and we identify the representations of $x\in X$ in the form  $x=x_1+x_{2\Sigma}+x_{20}$ (see \eqref{xrr-ind}) and the form $x=(x_1,x_{2\Sigma},x_{20})$. Below, when proving the theorems, we use the norm $\|\cdot\|$ in $X=X_1\dotpl X_{2\Sigma}\dotpl X_{20}$, defined by
$\|x\|=\|x_1\|+\|x_{2\Sigma}\|+ \|x_{20}\|$ where $\|x_1\|=\|x_1\|_{X_1}$, $\|x_{2\Sigma}\|=\|x_{2\Sigma}\|_{X_{2\Sigma}}$ and $\|x_{20}\|=\|x_{20}\|_{X_{20}}$ denote the norms of the components $x_1$, $x_{2\Sigma}$ and $x_{20}$ in the subspaces $X_1$, $X_{2\Sigma}$ and $X_{20}$, respectively. The norm $\|x\|$ of $x\in X_1\times X_{2\Sigma}\times X_{20}$ is defined in the same way and it coincides with the above-defined norm of the corresponding element $x\in X_1\dotpl X_{2\Sigma}\dotpl X_{20}$.
\end{remark}

 \begin{theorem}[\textbf{Solvability}]\label{Th_Exist-Lip_set}
Let $f\in C(\R_+\times D,Y)$, $D_B\supseteq D_A=D$, $\dim\ker A=n<\infty$, and let $\la A+B$ be a regular pencil of index $\nu$. Assume that there exists an open set $\widetilde{M}\subseteq W:=Y_1\dotpl Y_{2\Sigma}$ and a set $M_{20}\subseteq D_{20}$ such that the following conditions where $M_{1,2\Sigma}=\widetilde{A}^{-1}\widetilde{M}\subseteq D_1\dotpl  D_{2\Sigma}$ hold:
 \begin{enumerate}[label={\upshape\arabic*.}, ref={\upshape\arabic*}, itemsep=4pt,parsep=0pt,topsep=4pt,leftmargin=0.6cm]
\item\label{Sogl_set}
 For each  ${t\in \R_+}$, ${x_1+x_{2\Sigma}\in M_{1,2\Sigma}}$ there exists a unique ${x_{20}\in M_{20}}$ such that ${(t,x_1+x_{2\Sigma}+x_{20})\in L_0}$.

\item\label{Inv-Lipsch_set}
 A function $f(t,x)$ satisfies locally a Lipschitz condition with respect to $x$ on $\R_+\times D$.  \\
 For any fixed $t'$, $x'=x'_1+x'_{2\Sigma}+x'_{20}$ such that $x'_1+x'_{2\Sigma}\in M_{1,2\Sigma}$, $x'_{20}\in M_{20}$ and ${(t',x')\in L_0}$, there exist open neighborhoods
 $U_\delta(t',x'_1,x'_{2\Sigma})= U_{\delta_1}(t')\times U_{\delta_2}(x'_1)\times U_{\delta_3}(x'_{2\Sigma})\subset \R_+\times D_1\times D_{2\Sigma}$ and $U_\varepsilon(x'_{20})\subset D_{20}$ and an operator $W_{t',x'}\in \spL(D_{20},Y_{2*})$ having the inverse $W_{t',x'}^{-1}\in \spL(Y_{2*},D_{20})$ such that for each $(t,x_1,x_{2\Sigma})\in U_\delta(t',x'_1,x'_{2\Sigma})$ and each $x_{20}^1,\, x_{20}^2\in U_\varepsilon(x'_{20})$ the mapping $F_{2*}$ defined in \eqref{ADAE_AE_F2*} satisfies the inequality
  \begin{equation}\label{ContractiveMap}
 \big\|F_{2*}(t,x_1,x_{2\Sigma},x_{20}^1)- F_{2*}(t,x_1,x_{2\Sigma},x_{20}^2)- W_{t',x'}\big[x_{20}^1-x_{20}^2\big]\big\|\le q(\delta,\varepsilon)\big\|x_{20}^1-x_{20}^2\big\|,
   \end{equation}
 where $q(\delta,\varepsilon)$ is such that $\lim\limits_{\delta,\,\varepsilon\to 0} q(\delta,\varepsilon)< \|W_{t',x'}^{-1}\|^{-1}$  \,\textup{(}the numbers $\delta, \varepsilon>0$ depend on the choice of $t',x'$\textup{)}.
 \end{enumerate}
Then for each initial point $(t_0,x_0)\in L_0$, for which $(P_1+P_{2\Sigma})x_0\in M_{1,2\Sigma}$ and $P_{20}x_0\in M_{20}$, there exists a $t_{max}\le \infty$ such that the IVP \eqref{ADAE}, \eqref{ini} has a unique solution $x(t)$ in $M_{1,2\Sigma}\dotpl  M_{20}$ on the maximal interval of existence $[t_0,t_{max})$.
 \end{theorem}

 \begin{corollary}[\textbf{Solvability}]\label{Corollary_Exist_set}
Theorem~\ref{Th_Exist-Lip_set} remains valid if condition \ref{Inv-Lipsch_set} is replaced by
\begin{enumerate}[label={\upshape\arabic*.}, ref={\upshape\arabic*}, itemsep=4pt,parsep=0pt,topsep=4pt,leftmargin=0.6cm]
\addtocounter{enumi}{1}
\item\label{InvReg} A function $f(t,x)$ has a continuous (strong) partial derivative with respect to $x$ on $\R_+\times D$.  \\
    For any fixed $t'$, $x'=x'_1+x'_{2\Sigma}+x'_{20}$ such that $x'_1+x'_{2\Sigma}\in M_{1,2\Sigma}$, $x'_{20}\in M_{20}$ and ${(t',x')\in L_0}$, the operator
     \begin{equation}\label{funcPhiInvReg}
    W_{t',x'}:=\partial_{x_{20}}F_{2*}(t',x'_1,x'_{2\Sigma},x'_{20})= \left[\partial_x (Q_{2*}f)(t',x')-Q_{2*}B\right]P_{20}\big|_{D_{20}} \colon D_{20}\to Y_{2*}
     \end{equation}
    has the inverse $W_{t',x'}^{-1}\in \spL(Y_{2*},D_{20})$.
\end{enumerate}
 \end{corollary}

In Lemma \ref{Lem-Alternative} and Theorems \ref{Th_GlobReg}, \ref{Th_GlobReg2}, \ref{Th_Lagr_Stab} and \ref{Th_Lagr_Instab} below, it is required that ${M_{1,2\Sigma}=D_1\dotpl D_{2\Sigma}}$ and ${M_{20}=D_{20}}$, and therefore the statements of the above theorem and corollary take the following form.

 \begin{corollary}[\textbf{Solvability}]\label{Corol_Exist-Lip_set}
Let the conditions of Theorem~\ref{Th_Exist-Lip_set} (or Corollary \ref{Corollary_Exist_set}), where ${\widetilde{M}=W}$ $($accordingly, ${M_{1,2\Sigma}=D_1\dotpl D_{2\Sigma}}$$)$ and ${M_{20}=D_{20}}$, hold.
Then for each initial point $(t_0,x_0)\in L_0$ there exists a $t_{max}\le \infty$ such that the IVP \eqref{ADAE}, \eqref{ini} has a unique solution $x(t)$ on the maximal interval of existence $[t_0,t_{max})$.
 \end{corollary}

 \begin{proof}[Proof of Theorem \ref{Th_Exist-Lip_set}]
As shown above, the DAE \eqref{ADAE} is equivalent to the system \eqref{ADAE_DE_Pi}, \eqref{ADAE_AE_F2*}, and $(t,x)\in L_0$  if and only if $(t,x)$ satisfies \eqref{ADAE_AE_F2*} (where $x=x_1+x_{2\Sigma}+x_{20}$).

Obviously, $F_{2*}\in C(\R_+\times D_1 \times D_{2\Sigma}\times D_{20},Y_{2*})$ satisfies locally a Lipschitz condition with respect to $(x_1,x_{2\Sigma},x_{20})$ on $\R_+\times D_1 \times D_{2\Sigma}\times D_{20}$.

Below, we provide the following lemma which is proved in a similar way as Lemma \ref{Lem-ImplicitFunc}.
  \begin{lemma}\label{Lem-ImplicitFuncReg_set}
For any fixed $t'$, $x'=x'_1+x'_{2\Sigma}+x'_{20}$ such that $x'_1+x'_{2\Sigma}\in M_{1,2\Sigma}$, $x'_{20}\in M_{20}$ and ${(t',x')\in L_0}$, there exist open neighborhoods
$U_r(t',x'_1,x'_{2\Sigma})=U_{r_1}(t')\times U_{r_2}(x'_1)\times U_{r_3}(x'_{2\Sigma})\subset U_\delta(t',x'_1,x'_{2\Sigma})$ and $U_\rho(x'_{20})\subset U_\varepsilon(x'_{20})$ \,$($\,$r,\rho>0$ depend on the choice of $t'$, $x'$\,$)$ and a unique function  $\varkappa\in C(U_r(t',x'_1,x'_{2\Sigma}),U_\rho(x'_{20}))$  such that $\varkappa(t,x_1,x_{2\Sigma})$ satisfies a Lipschitz condition with respect to $(x_1,x_{2\Sigma})$ on $U_r(t',x'_1,x'_{2\Sigma})$ and the equality $\varkappa(t',x'_1,x'_{2\Sigma})=x'_{20}$ and is a solution of equation \eqref{ADAE_AE_F2*} with respect to $x_{20}$, i.e., $F_{2*}(t,x_1,x_{2\Sigma},\varkappa(t,x_1,x_{2\Sigma}))=0$ for every $(t,x_1,x_{2\Sigma})\in U_r(t',x'_1,x'_{2\Sigma})$.
  \end{lemma}

Note that equation \eqref{ADAE_AE_F2*} can be written in the form $x_{20}=x_{20}-W_{t',x'}^{-1}F_{2*}(t,x_1,x_{2\Sigma},x_{20})$, where $(t,x_1,x_{2\Sigma})\in U_\delta(t',x'_1,x'_{2\Sigma})$, $x_{20}\in U_\varepsilon(x'_{20})$ and $U_\delta(t',x'_1,x'_{2\Sigma})$, $U_\varepsilon(x'_{20})$ and $W_{t',x'}$ are defined in condition~\ref{Inv-Lipsch_set}, or in the form
\begin{equation}\label{ADAE_AE_F2*equiv2}
x_{20}=x_{20}-W_{t,\,x}^{-1}F_{2*}(t,x_1,x_{2\Sigma},x_{20}),
 \end{equation}
where $t$, $x=x_1+x_{2\Sigma}+x_{20}$ are such that ${x_1+x_{2\Sigma}\in M_{1,2\Sigma}}$, ${x_{20}\in M_{20}}$ and ${(t,x)\in L_0}$. Generally, the above form of equation \eqref{ADAE_AE_F2*} is used in the proof of Lemma \ref{Lem-ImplicitFuncReg_set}, but equation \eqref{ADAE_AE_F2*equiv2} is also used later, in the proof of Lemma \ref{Lem-Alternative}, and therefore it is given here.

Further, in the same way as in the proof of Theorem \ref{Th-Nonloc_ImplFunc} on a nonlocal implicit function, we construct the function
 \begin{equation}\label{eta}
\eta\colon \R_+\times M_1\times M_{2\Sigma}\to M_{20}, \quad \text{where} \quad M_1=P_1M_{1,2\Sigma},\quad M_{2\Sigma}=P_{2\Sigma}M_{1,2\Sigma},
 \end{equation}
so that $\eta(t,x_1,x_{2\Sigma})$ is continuous in $(t,x_1,x_{2\Sigma})$, satisfies locally a Lipschitz condition with respect to $(x_1,x_{2\Sigma})$ on $\R_+\times M_1\times M_{2\Sigma}$ and is a unique solution of  \eqref{ADAE_AE_F2*} with respect to $x_{20}$  (for each $(t,x_1,x_{2\Sigma})\in \R_+\times M_1\times M_{2\Sigma}$).

Let us substitute $x_{20}=\eta(t,x_1,x_{2\Sigma})$ in equation \eqref{ADAE_DE_Pi}, multiply it by  $\widetilde{A}$ (defined in \eqref{widetilde_A}) and make the change of variables
 $$
w=\widetilde{A}(x_1+x_{2\Sigma}) \qquad  (x_1+x_{2\Sigma}\in M_{1,2\Sigma}).
 $$
Then \,$x_1=P_1\widetilde{A}^{-1}w$, \,$x_{2\Sigma}=P_{2\Sigma}\widetilde{A}^{-1}w$,\, $w\in \widetilde{M}$, and we obtain the equation
 \begin{equation}\label{ADAE_DE_eta_w}
\begin{split}
\dot{w}=\widehat{\Pi}(t,w),\qquad
 \widehat{\Pi}(t,w) & := \widetilde{A}\,\Pi(t,\widetilde{A}^{-1}w+\widehat{\eta}(t,w)) =  \\
 & = (Q_1+Q_{2\Sigma})f(t,\widetilde{A}^{-1}w+\widehat{\eta}(t,w))- (Q_1+Q_{2\Sigma})B\big[\widetilde{A}^{-1}w+\widehat{\eta}(t,w)\big],  \\
 & \widehat{\eta}(t,w) := \eta(t,P_1\widetilde{A}^{-1}w,P_{2\Sigma}\widetilde{A}^{-1}w).
\end{split}
 \end{equation}
Since $\widetilde{A}^{-1}\in \spL(Y,X)$, $Q_1+Q_{2\Sigma}\in \spL(Y)$, $(Q_1+Q_{2\Sigma})B\widetilde{A}^{-1}\in \spL(Y)$ and $(Q_1+Q_{2\Sigma})BP_{20}\in \spL(D,Y)$, then, due to the properties of $f$ and $\eta$, the function
$\widehat{\Pi}\colon \R_+\times \widetilde{M}\to W$ (\,${W=Y_1\dotpl  Y_{2\Sigma}}$) is continuous and satisfies locally a Lipschitz condition with respect to $w$ on $\R_+\times \widetilde{M}$.
By \cite[p.~9, Theorem~1]{Schwartz2} we infer that for each initial point $(t_0,x_0)\in L_0$ such that $(P_1+P_{2\Sigma})x_0\in M_{1,2\Sigma}$ and $P_{20}x_0\in M_{20}$ the IVP for equation \eqref{ADAE_DE_eta_w} with the initial condition
\begin{equation}\label{RegDEeta_ini}
w(t_0)=w_0,\quad w_0=\widetilde{A}(P_1+P_{2\Sigma})x_0=Ax_0\in \widetilde{M},
\end{equation}
has a unique solution $w(t)$ on some interval $J\subset \R_+$ that contains $t_0$.
Using \cite[p.~16-17, Corollary]{Schwartz2}, we also infer that the solution $w(t)$ can be extended over a maximal interval of existence $[t_0,t_{max})$ (${t_{max}\le \infty}$) of $w(t)$ in $\widetilde{M}$, and the extended solution is unique.  Notice that the function $x(t)=\widetilde{A}^{-1}w(t)+\widehat{\eta}(t,w(t))$ is continuous and  $Ax(t)=(Q_1+Q_{2\Sigma})w(t)=w(t)$ has the continuous derivative $\dfrac{dAx}{dt}(t)=\widehat{\Pi}(t,w(t))$ on $[t_0,t_{max})$ (for $t_0=0$, $Ax(t)$ has the derivative on the right at $t_0$).  Thus, the IVP \eqref{ADAE}, \eqref{ini} has the unique solution
$$
x(t)=\widetilde{A}^{-1}w(t)+\widehat{\eta}(t,w(t))
$$
in $M_{1,2\Sigma}\dotpl M_{20}$ on the maximal interval of existence $[t_0,t_{max})$.
 \end{proof}

 \begin{proof}[Proof of Corollary \ref{Corollary_Exist_set}]
Corollary \ref{Corollary_Exist_set} follows from Theorem \ref{Th-Nonloc_ImplFunc-2} or Corollary \ref{Cor_Nonloc_ImplFunc-2}.

In addition, it is readily verified that condition \ref{Inv-Lipsch_set} of Theorem~\ref{Th_Exist-Lip_set} holds since condition \ref{InvReg} of Corollary~\ref{Corollary_Exist_set} holds (see Remark \ref{Rem__Nonloc_ImplFunc-2}) and hence all the conditions of Theorem~\ref{Th_Exist-Lip_set} are fulfilled.
 \end{proof}

The lemma below will be used in the sequel. Its assertion is generally clear and can be deduced from the results of \cite{KK1956}, as can be seen from its proof. But such a lemma has not been found in the literature and therefore it is given here.

 \begin{lemma}\label{Lem_Exist-Blow_up}
Let $H\in C(\R_+\times Z,Z)$, where $Z$ is a Banach space, and let the IVP
 \begin{equation}\label{Lem_IVP}
\begin{split}
&\dot{z}=H(t,z), \\
&z(t_0)=z_0
\end{split}
 \end{equation}
have a solution for each initial point $(t_0,z_0)\in \R_+\times Z$.  Assume that $\sup\limits_{t\,\in\, [t_0,b_1],\, \|z\|\,\le\, b_2}\|H(t,z)\|<\infty$ for every fixed $b_1>t_0$, $b_2>0$.
Then there exists a maximal interval of existence $[t_0,t_{max})$ of a solution $z(t)$ of  the IVP \eqref{Lem_IVP}, and if $t_{max}<\infty$, then $\lim\limits_{t\to t_{max}-0}\|z(t)\|=\infty$.
 \end{lemma}

 \begin{proof}
The existence of the maximal interval of existence $[t_0,t_{max})$ of the solution $z(t)$ follows from \cite[Lemma 3.2]{KK1956}.
Let us prove that $\lim\limits_{t\to t_{max}-0}\|z(t)\|=\infty$ if $t_{max}<\infty$.

Note that from \cite[Lemma 3.3]{KK1956} it follows that $\lim\limits_{t\to t_{max}-0} z(t)$ does not exist (the same follows from \cite[p.~17, the proof of Corollary]{Schwartz2}), i.e., there is no point $z_*\in Z$ such that $\lim\limits_{t\to t_{max}-0} z(t)=z_*$.

Assume that $\|z(t)\|\not\to \infty$ as ${t\to t_{max}-0}$. Then there exists a number $r_0>0$ such that $z(t)\in \overline{B}_{r_0}(0)=\{z\in Z\mid \|z\|\le r_0\}$ for all $t\in [t_0,t_{max})$. Hence, $\|H(t,z(t))\|\le C$, where $C>0$ is some constant, for all $t\in [t_0,t_{max})$. Consider a sequence $\{t_n\}$ such that $t_0\le t_n <t_{max}$ and ${t_n\to t_{max}}$ as ${n\to \infty}$. Since the solution $z(t)$ of the IVP \eqref{Lem_IVP} satisfies the integral equation ${z(t)=z_0+\int\limits_{t_0}^t H(\tau,z(\tau)) d\tau}$, then $\|z(t_n)-z(t_m)\|=\bigg\|\int\limits_{t_m}^{t_n} H(\tau,z(\tau)) d\tau\bigg\|\le C|t_n-t_m|\to 0$ as ${m,n\to \infty}$ for any $t_m$, $t_n$. Hence, $z(t_n)$ is a Cauchy sequence and, therefore, it converges to some point $z_*\in Z$. Consequently, there exists $\lim\limits_{t\to t_{max}-0} z(t)=z_*\in Z$, which contradicts the fact mentioned above.
Thus, it is proved that for any $r>0$ there exists $t_*$ \,($t_0<t_*<t_{max}$) such that $z(t)\not\in \overline{B}_{r}(0)$ for $t\in [t_*,t_{max})$ and, therefore, ${\lim\limits_{t\to t_{max}-0}\|z(t)\|=\infty}$.
 \end{proof}

  \begin{lemma}\label{Lem-Alternative}
Let the conditions of Theorem~\ref{Th_Exist-Lip_set} \textup{(}or Corollary \ref{Corollary_Exist_set}\textup{)}, where ${\widetilde{M}=W}$ $($accordingly, ${M_{1,2\Sigma}=D_1\dotpl D_{2\Sigma}}$$)$ and ${M_{20}=D_{20}}$, and the following condition be satisfied:
 \begin{enumerate}[label={\upshape\arabic*.}, ref={\upshape\arabic*}, itemsep=4pt,parsep=0pt,topsep=4pt,leftmargin=0.6cm]
 \addtocounter{enumi}{2}
\item\label{Alternative} For every fixed ${a_1,a_2>0}$ the following holds:
  \begin{equation}\label{Alt-1}
 \sup\limits_{(t,x)\,\in\, L_0,\; t\,\in\, [0,a_1],\; \|x_1+x_{2\Sigma}\|\,\le\, a_2} \|x_{20}-W_{t,x}^{-1}F_{2*}(t,x_1,x_{2\Sigma},x_{20})\|\le K,
  \end{equation}
  \begin{equation}\label{Alt-2}
 \sup\limits_{(t,x)\,\in\, L_0,\; t\,\in\, [0,a_1],\; \|x_1+x_{2\Sigma}\|\,\le\, a_2,\; \|x_{20}\|\,\le\, K} \|(Q_1+Q_{2\Sigma})f(t,x)\|<\infty \qquad (x=x_1+x_{2\Sigma}+x_{20}),
  \end{equation}
  where ${K=K(a_1,a_2)>0}$ is some constant \textup{(}depending on the choice of  $a_1$, $a_2$\textup{)}.
 \end{enumerate}
Then for each initial point $(t_0,x_0)\in L_0$ there exists a $t_{max}\le \infty$ such that the IVP \eqref{ADAE}, \eqref{ini} has a unique solution $x(t)$ on the maximal interval of existence $[t_0,t_{max})$, and if $t_{max}<\infty$, then $\lim\limits_{t\to t_{max}-0}\|Ax(t)\|=\infty$. In addition, if $D=X$ or the operator $A$ is bounded, then $\lim\limits_{t\to t_{max}-0}\|x(t)\|=\infty$ when $t_{max}<\infty$.
 \end{lemma}

 \begin{remark}\label{Rem_Alternative}
The condition \eqref{Alt-1} in Lemma \ref{Lem-Alternative} can be replaced by the following:  \\
For any fixed  ${a_1,a_2>0}$ there exists a constant ${K=K(a_1,a_2)>0}$ (depending on $a_1$, $a_2$) such that ${\|x_{20}\|\le K}$ for each ${(t,x_1+x_{2\Sigma}+x_{20})\in L_0}$, for which $t\in [0,a_1]$ and $\|x_1+x_{2\Sigma}\|\le a_2$.
 \end{remark}

 \begin{theorem}[\textbf{Global solvability}]\label{Th_GlobReg}
Let the conditions of Lemma \ref{Lem-Alternative} and the following condition hold:
\begin{enumerate}[label={\upshape\arabic*.}, ref={\upshape\arabic*}, itemsep=4pt,parsep=0pt,topsep=4pt,leftmargin=0.6cm]
\addtocounter{enumi}{3}
\item\label{Extens} There exists a number ${R>0}$, a functional
  \begin{equation}\label{Vmax}
   V(w)=\max\limits_{i=1,...,m}V_i(w),
  \end{equation}
  where $w=\widetilde{A}(x_1+x_{2\Sigma})\in W=Y_1\dotpl Y_{2\Sigma}$ and $V_i\in C(W,\R_+)$, $i=1,...,m$, are continuously differentiable functionals on $U_R^c(0)=\{w\in W\mid \|w\|\ge R\}$, and functionals $U\in C(\R_+)$, $\psi\colon \R_+\to \R_+$ such that $\psi(t)$ is integrable on each finite interval in $\R_+$,
   \begin{equation}\label{Extens_lim}
  \int\limits^{\infty}\dfrac{du}{U(u)} =\infty,\qquad  \lim\limits_{\|w\|\to\infty}V(w)=\infty,
   \end{equation}
  and for each $(t,\widetilde{A}^{-1}w+x_{20})\in L_0$ for which $w\in U_R^c(0)$ the following inequality holds:
   \begin{equation}\label{ExtensGlobReg}
  \big\langle \widetilde{A}\,\Pi(t,\widetilde{A}^{-1}w+x_{20}), \partial_{w}V_{j(w)}(w) \big\rangle \le U\big(V(w)\big)\psi(t),
   \end{equation}
  where $j(w)\in \{1,...,m\}$ is any index for which $V(w)=V_{j(w)}(w)$ for a given $w$.
\end{enumerate}
Then for each initial point $(t_0,x_0)\in L_0$ there exists a unique global (i.e., $t_{max}=\infty$) solution $x(t)$ of the IVP \eqref{ADAE}, \eqref{ini}.
 \end{theorem}

  \begin{theorem}[\textbf{Global solvability}]\label{Th_GlobReg2}
Theorem~\ref{Th_GlobReg} remains valid if condition \ref{Extens} is replaced by
 \begin{enumerate}[label={\upshape\arabic*.}, ref={\upshape\arabic*}, itemsep=4pt,parsep=0pt,topsep=4pt,leftmargin=0.6cm]
\addtocounter{enumi}{3}
\item\label{Extens-2} There exists a number ${R>0}$, a functional $V\in C(W,\R_+)$, where $w=\widetilde{A}(x_1+x_{2\Sigma})\in W=Y_1\dotpl Y_{2\Sigma}$, and functionals $U\in C(\R_+)$, $\psi\colon \R_+\to \R_+$ such that $\psi(t)$ is integrable on each finite interval in $\R_+$, \eqref{Extens_lim} holds,
    $V(w)$ satisfies a Lipschitz condition on $U_R^c(0)=\{w\in W\mid \|w\|\ge R\}$   and for each $(t,\widetilde{A}^{-1}w+x_{20})\in L_0$ such that $w\in U_R^c(0)$ the following inequality holds:
   \begin{equation}\label{ExtensGlobReg-2}
  \big\| \widetilde{A}\,\Pi(t,\widetilde{A}^{-1}w+x_{20})\big\|\le U\big(V(w)\big)\psi(t).
   \end{equation}
 \end{enumerate}
 \end{theorem}

 \begin{proof}[Proof of Lemma \ref{Lem-Alternative}]
The existence of a unique solution of the IVP \eqref{ADAE}, \eqref{ini} on a maximal interval of existence follows from Theorem~\ref{Th_Exist-Lip_set} (recall that if the conditions of Corollary~\ref{Corollary_Exist_set} hold, then the conditions of Theorem~\ref{Th_Exist-Lip_set} also hold).  More precisely, from the proof of Theorem~\ref{Th_Exist-Lip_set} (in the case when ${\widetilde{M}=W}$, ${M_{1,2\Sigma}=D_1\dotpl D_{2\Sigma}}$ and ${M_{20}=D_{20}}$) it follows that for an arbitrary initial point $(t_0,x_0)\in L_0$ there exists the unique solution $x(t)=\widetilde{A}^{-1}w(t)+\widehat{\eta}(t,w(t))$ of the IVP \eqref{ADAE}, \eqref{ini} on the maximal interval of existence $[t_0,t_{max})$.

We will prove that if $t_{max}<\infty$, then $\lim\limits_{t\to t_{max}-0}\|w(t)\|=\infty$, where $w(t)$ is the solution of the IVP \eqref{ADAE_DE_eta_w}, \eqref{RegDEeta_ini} on $[t_0,t_{max})$.
Let us estimate the functions $\widehat{\eta}$ and $\widehat{\Pi}(t,w)$ that are defined in \eqref{ADAE_DE_eta_w}. Since $x_{20}=\eta(t,x_1,x_{2\Sigma})$ satisfies equation \eqref{ADAE_AE_F2*equiv2},
then \eqref{Alt-1} yields the estimate $\sup\limits_{t\,\in\, [0,a_1],\, \|x_1+x_{2\Sigma}\|\,\le\, a_2} \|\eta(t,x_1,x_{2\Sigma})\|\le K$ with some constant ${K=K(a_1,a_2)}$ for any fixed $a_1>0$, $a_2>0$.
Consequently, for any fixed $a_1>0$, $a_3>0$ the inequality $\|\widehat{\eta}(t,w)\|\le K_1$, where ${K_1=K_1(a_1,a_3)}$ is some constant, holds for every $t\in [0,a_1]$, $\|w\|\le a_3$ (since $\|x_1+x_{2\Sigma}\|=\|\widetilde{A}^{-1}w\|\le \|\widetilde{A}^{-1}\|a_3<\infty$).
Thus, $\sup\limits_{t\,\in\, [0,a_1],\, \|w\|\,\le\, a_3} \|\widehat{\eta}(t,w)\|\le K_1$ and, using \eqref{Alt-2}, we obtain   $\sup\limits_{t\,\in\,[0,a_1],\, \|w\|\,\le\, a_3} \|(Q_1+Q_{2\Sigma})f(t,\widetilde{A}^{-1}w+\widehat{\eta}(t,w))\|<\infty$ for any fixed $a_1>0$, $a_3>0$. Hence,
$$
\sup\limits_{t\,\in\, [0,b_1],\, \|w\|\,\le\, b_2} \|\widehat{\Pi}(t,w)\|<\infty \quad \text{for any fixed $b_1>0$, $b_2>0$}.
$$

It follows from Lemma~\ref{Lem_Exist-Blow_up} that $\lim\limits_{t\to t_{max}-0}\|w(t)\|=\infty$ if $t_{max}<\infty$.  As shown in the proof of Theorem \ref{Th_Exist-Lip_set}, $Ax(t)=w(t)$, $t\in [t_0,t_{max})$. Hence, $\|Ax(t)\|=\|w(t)\|$ and, therefore, $\lim\limits_{t\to t_{max}-0}\|Ax(t)\|=\infty$ when $t_{max}<\infty$.

Notice that if $D=X$, then $A$ is a bounded operator (since it is closed and $D_A=D$).
If the operator $A$ is bounded, then $\lim\limits_{t\to t_{max}-0}\|Ax(t)\|=\infty$ implies $\lim\limits_{t\to t_{max}-0}\|x(t)\|=\infty$.
\end{proof}

 \begin{proof}[Proof of Theorem \ref{Th_GlobReg}]
Let $(t_0,x_0)\in L_0$ is an arbitrary initial point.
Since all the conditions of Lemma \ref{Lem-Alternative}  (consequently, the conditions of Theorem~\ref{Th_Exist-Lip_set}) are satisfied, then the IVP \eqref{ADAE}, \eqref{ini} has the unique solution ${x(t)=\widetilde{A}^{-1}w(t)+\widehat{\eta}(t,w(t))}$ on the maximal interval of existence $[t_0,t_{max})$ (see the proof of Theorem~\ref{Th_Exist-Lip_set}). Here $w(t)$ is the solution of the IVP \eqref{ADAE_DE_eta_w}, \eqref{RegDEeta_ini} on $[t_0,t_{max})$, $\widetilde{A}^{-1}\in \spL(Y,X)$ and $\widehat{\eta}\in C(\R_+\times W,D_{20})$ is the function defined in \eqref{ADAE_DE_eta_w}. Also, it follows  from the proof of Lemma \ref{Lem-Alternative} that $\lim\limits_{t\to t_{max}-0}\|w(t)\|=\infty$ if $t_{max}<\infty$.

Below we will prove that $w(t)$ exists on $[t_0,t_{max})=[t_0,\infty)$. Then $x(t)$ exists on $[t_0,\infty)$  as well.

Suppose that $t_{max}<\infty$. Then there exists a sequence $\{t_n\}$, $t_0\le t_n <t_{max}$, such that $t_n\to t_{max}$ and $\|w(t_n)\|\to\infty$ as ${n\to\infty}$. The further proof is carried out, in general, in the same way as the main part of the proof of the theorem \cite[Theorem~3]{KK1956} on the global solvability of ODEs.  But in \cite[Theorem~3]{KK1956} the condition, similar to condition \ref{Extens} for the whole space $W$ (not only for $w\in U_R^c(0)$) and hence being more restrictive, was used. So we will give the proof for our case.

It follows from condition~\ref{Extens} that there exists a number ${R>0}$ such that the inequality  $\big\langle \widehat{\Pi}(t,w), \partial_{w}V_{j(w)}(w) \big\rangle \le U\big(V(w)\big)\psi(t)$, where $V\in C(W,\R_+)$, $U\in C(\R_+)$ and $\psi\colon \R_+\to \R_+$ are some functionals satisfying condition~\ref{Extens}, holds for each $t\in \R_+$, $w\in U_R^c(0)$. Hence, for any $\varepsilon>0$ and for each $t\in \R_+$, $w\in U_R^c(0)$,
\begin{equation}\label{Extens-widehat_Pi}
\big\langle \widehat{\Pi}(t,w), \partial_{w}V_{j(w)}(w) \big\rangle \le \big[U\big(V(w)\big)+\varepsilon\big]\psi(t).
\end{equation}
Since, by virtue of the above assumption, $\lim\limits_{n\to\infty}\|w(t_n)\|= \infty$, then $\lim\limits_{n\to\infty}V(w(t_n))=\infty$, and also there exists a number $N$ such that $w(t_n)\in U_R^c(0)$ for all $n\ge N$.
We deduce from \eqref{Extens-widehat_Pi} and \cite[Lemma~2.2]{KK1956} that
 $$
V(w(t_n))-V(w(t_N))\le \int\limits_{t_N}^{t_n}\Big[U\big(V(w(\tau))\big)+\varepsilon\Big]\psi(\tau)\, d\tau,
 $$
and therefore, due to \cite[Lemma~2.1]{KK1956},
 $$
\int\limits_{V(w(t_N))}^{V(w(t_n))} \dfrac{du}{U(u)+\varepsilon}\le \int\limits_{t_N}^{t_n}\psi(\tau)\, d\tau
 $$
for any $\varepsilon>0$ and $n\ge N$. Further, passing to the limit as ${\varepsilon\to +0}$ and ${n\to \infty}$, we see that
 $$
\infty=\int\limits_{V(w(t_N))}^\infty \dfrac{du}{U(u)}\le \int\limits_{t_N}^{t_{max}}\psi(\tau)\, d\tau < \infty,
 $$
since \eqref{Extens_lim} holds and $\psi$ is integrable on every finite interval in $\R_+$. Thus, we have been led to a contradiction, so that $t_{max}=\infty$.
 \end{proof}

 \begin{proof}[Proof of Theorem \ref{Th_GlobReg2}]
The beginning of the proof is the same as for Theorem \ref{Th_GlobReg} and, as well as in the proof of Theorem \ref{Th_GlobReg}, we suppose that $t_{max}<\infty$ and consider a sequence $\{t_n\}$, $t_0\le t_n <t_{max}$, such that $t_n\to t_{max}$ and $\|w(t_n)\|\to\infty$ as $n\to\infty$.

It follows from condition~\ref{Extens-2} of Theorem \ref{Th_GlobReg2} that there exists a number ${R>0}$ such that the inequality
\begin{equation}\label{Extens-2-widehat_Pi}
\big\| \widehat{\Pi}(t,w)\big\| \le \big[U\big(V(w)\big)+\varepsilon\big]\psi(t),
\end{equation}
where $V$, $U$ and $\psi$ are some functionals satisfying condition~\ref{Extens-2}, hods for each $t\in \R_+$, $w\in U_R^c(0)$ and any $\varepsilon>0$. Since $V$ satisfies a Lipschitz condition on $U_R^c(0)$, then there exists a constant $l>0$ such that $|V(w^1)-V(w^2)|\le l\|w^1-w^2\|$ for any $w^1,w^2\in U_R^c(0)$. In addition, $\lim\limits_{n\to\infty}\|w(t_n)\|= \infty$, $\lim\limits_{n\to\infty}V(w(t_n))=\infty$ and, therefore, there exists a number $N$ such that $w(t_n)\in U_R^c(0)$ and $V(w(t_n))\ge V(w(t_N))$ for all $n\ge N$. Then, using \eqref{Extens-2-widehat_Pi} and \cite[Lemma~2.3]{KK1956}, we conclude that
$$
V(w(t_n))-V(w(t_N))=\big|V(w(t_n))-V(w(t_N))\big|\le l\int\limits_{t_N}^{t_n}\Big[U\big(V(w(\tau))\big)+\varepsilon\Big]\psi(\tau)\, d\tau,
$$
for any $\varepsilon>0$ and $n\ge N$. Hence, due to \cite[Lemma~2.1]{KK1956},\, $\dfrac{1}{l}\int\limits_{V(w(t_N))}^{V(w(t_n))} \dfrac{du}{U(u)+\varepsilon}\le \int\limits_{t_N}^{t_n} \psi(\tau)\, d\tau$ for any $\varepsilon>0$, $n\ge N$.
Further, as in the proof of Theorem \ref{Th_GlobReg}, we obtain the contradiction:
 $$
\infty=\dfrac{1}{l}\int\limits_{V(w(t_N))}^\infty \dfrac{du}{U(u)}\le \int\limits_{t_N}^{t_{max}}\psi(\tau)\, d\tau < \infty.
 $$
This implies $t_{max}=\infty$. Thus, $w(t)$ and, hence, $x(t)$ exist on $[t_0,\infty)$.
 \end{proof}

 \begin{theorem}[\textbf{Lagrange stability}]\label{Th_Lagr_Stab}
Let the conditions of Theorem~\ref{Th_Exist-Lip_set} (or Corollary \ref{Corollary_Exist_set}), where ${\widetilde{M}=W}$ $($accordingly, ${M_{1,2\Sigma}=D_1\dotpl D_{2\Sigma}}$$)$ and ${M_{20}=D_{20}}$, and the following conditions be satisfied:
 \begin{enumerate}[label={\upshape\arabic*.}, ref={\upshape\arabic*}, itemsep=4pt,parsep=0pt,topsep=4pt,leftmargin=0.6cm]
 \addtocounter{enumi}{2}
\item\label{Lagr-Alternative} For every fixed $b>0$, $d>0$ the following holds:
  \begin{equation}\label{Lagr-Alt-1}
 \sup\limits_{(t,x)\,\in\, L_0,\; \|x_1+x_{2\Sigma}\|\,\le\, b} \|x_{20}-W_{t,x}^{-1}F_{2*}(t,x_1,x_{2\Sigma},x_{20})\|\le K,
  \end{equation}
  \begin{equation}\label{Lagr-Alt-2}
 \sup\limits_{(t,x)\,\in\, L_0,\; t\,\in\, [0,d],\; \|x_1+x_{2\Sigma}\|\,\le\, b,\; \|x_{20}\|\,\le\, K} \|(Q_1+Q_{2\Sigma})f(t,x)\|<\infty \qquad (x=x_1+x_{2\Sigma}+x_{20}),
  \end{equation}
  where ${K=K(b)>0}$ is some constant \textup{(}depending on the choice of $b$\textup{)}.

 \item  Let condition \ref{Extens} of Theorem~\ref{Th_GlobReg} or condition \ref{Extens-2} of Theorem \ref{Th_GlobReg2} hold and let the functional $\psi$ defined in these conditions be integrable on $\R_+$, i.e.,
  \begin{equation}\label{Stab-psi}
 \int\limits_0^\infty \psi(t)dt<\infty.
  \end{equation}
   \end{enumerate}
Then the ADAE \eqref{ADAE} is Lagrange stable \textup{(}for the initial points $(t_0,x_0)\in L_0$\textup{)}.
 \end{theorem}

 \begin{remark}\label{Lag-Alt-1_2}
The condition \eqref{Lagr-Alt-1} in Theorem \ref{Th_Lagr_Stab} can be replaced by the following:  \\
For any fixed ${b>0}$ there exists a constant ${K=K(b)>0}$ (depending on $b$) such that ${\|x_{20}\|\le K}$ for each ${(t,x_1+x_{2\Sigma}+x_{20})\in L_0}$ for which $\|x_1+x_{2\Sigma}\|\le b$.
 \end{remark}
 \begin{remark}\label{Alternative-connect}

Note that if condition \ref{Lagr-Alternative} of Theorem~\ref{Th_Lagr_Stab} holds, then condition \ref{Alternative} of Lemma \ref{Lem-Alternative} holds.
 \end{remark}

 \begin{proof}[Proof of Theorem \ref{Th_Lagr_Stab}]
Since condition \ref{Alternative} of Lemma \ref{Lem-Alternative} is fulfilled (see Remark \ref{Alternative-connect}), then all the conditions of Theorem~\ref{Th_GlobReg} or \ref{Th_GlobReg2} are satisfied. Hence, for an arbitrary initial point $(t_0,x_0)\in L_0$ the IVP \eqref{ADAE}, \eqref{ini} has the unique solution $x(t)=\widetilde{A}^{-1}w(t)+\widehat{\eta}(t,w(t))$ on $[t_0,\infty)$, where $w(t)$ is the solution of the IVP \eqref{ADAE_DE_eta_w}, \eqref{RegDEeta_ini} on $[t_0,\infty)$ and $\widehat{\eta}\in C(\R_+\times W,D_{20})$ is the function defined in \eqref{ADAE_DE_eta_w} (see the proofs of Theorem~\ref{Th_Exist-Lip_set} and Theorem \ref{Th_GlobReg} or \ref{Th_GlobReg2}).

Assume that $\sup\limits_{t\in [t_0,\infty)}\|w(t)\|=\infty$.
Then there exists a sequence $\{t_n\}$ such that $t_n\ge t_0$, $t_n\to \infty$ and $\|w(t_n)\|\to \infty$ as $n\to \infty$. Then, in the same way as in the proof of Theorem~\ref{Th_GlobReg} or \ref{Th_GlobReg2} (depending on the condition that we use), we obtain that $\infty= \int\limits_{0}^{\infty}\psi(\tau)\, d\tau$, which contradicts \eqref{Stab-psi}. Thus, $\sup\limits_{t\in [t_0,\infty)}\|w(t)\|<\infty$ and, hence, $\sup\limits_{t\in [t_0,\infty)}\|\widetilde{A}^{-1}w(t)\|<\infty$. From this estimate and \eqref{Lagr-Alt-1} it follows that $\sup\limits_{t\,\in\, [t_0,\infty)} \|\widehat{\eta}(t,w(t))\|\le K_1$, where $K_1$ is some constant. Hence, $\sup\limits_{t\in [t_0,\infty)}\|x(t)\|=\sup\limits_{t\in [t_0,\infty)}\|\widetilde{A}^{-1}w(t)+\widehat{\eta}(t,w(t))\| <\infty$.
Thus, the ADAE \eqref{ADAE} is Lagrange stable for each initial point $(t_0,x_0)\in L_0$.

Since a solution for the IVP \eqref{ADAE}, \eqref{ini} exists only for the consistent initial points, that is, $(t_0,x_0)\in L_0$, then equation \eqref{ADAE} is Lagrange stable (in the sense that each solution is Lagrange stable).
 \end{proof}

 \begin{remark}
Generally, it is not necessary to require that ${M_{20}=D_{20}}$ in Lemma \ref{Lem-Alternative} and Theorems \ref{Th_GlobReg}, \ref{Th_GlobReg2}, \ref{Th_Lagr_Stab}, but then it is necessary to add $x_{20}\in M_{20}$ to their conditions and statements. In particular, the statement of Theorem \ref{Th_Lagr_Stab} will change to the following:
Then the ADAE \eqref{ADAE} is Lagrange stable for the initial points $(t_0,x_0)\in L_0$ for which $x_{20}\in M_{20}$.
In addition, it is necessary to add the equality ${M_{20}=D_{20}}$ to the requirement $D=X$ in the statement of Lemma \ref{Lem-Alternative} (i.e., the last statement of the lemma will take the following form: if $M_{20}=D_{20}$ and $D=X$ or if the operator $A$ is bounded, then $\lim\limits_{t\to t_{max}-0}\|x(t)\|=\infty$ when $t_{max}<\infty$).
 \end{remark}

The lemmas below will be used in the sequel.

 \begin{lemma}\label{Lem_2.2-ineq_ge}
Let $H\in C(\R_+\times D,W)$, where $W$ is a Banach space and $D\subseteq W$. Let there exist a functional $V$ of the form \eqref{Vmin}, where $V_i\in C(W,\R_+)$, $i=1,...,m$, are continuously differentiable functionals on $D$, and functionals $U\in C(\R_+)$, $\psi\colon \R_+\to \R_+$ such that $\psi(t)$ is integrable on each finite interval in $\R_+$, and the inequality
$$
\langle H(t,w),\partial_{w}V_{j(w)}(w) \rangle \ge U(V(w))\psi(t),
$$
where $j(w)\in \{1,...,m\}$ is any index for which $V(w)=V_{j(w)}(w)$ for a given $w$, holds for each $w\in D$.
Assume that the equation $\dot{w}=H(t,w)$ has a solution $w(t)$ on an interval $[t_1,t_2]\subset \R_+$ and $w(t)\in D$ for all $t\in [t_1,t_2]$.
Then the following inequality holds:
 $$
V(w(t_2))-V(w(t_1))\ge \int\limits_{t_1}^{t_2}U\big(V(w(\tau))\big)\psi(\tau)\, d\tau.
 $$
 \end{lemma}
 \begin{proof}[Proof of Lemma \ref{Lem_2.2-ineq_ge}]
The proof is carried out in the same way as the proof of \cite[Lemma~2.2]{KK1956}. In the lemma \cite[Lemma~2.2]{KK1956}, a functional of the form \eqref{Vmax} instead of \eqref{Vmin} and other inequality signs are used. Therefore, it is not difficult to make appropriate changes to the proof of this lemma and to verify that the present lemma is true.
 \end{proof}

 \begin{lemma}\label{Lem_Blow_up}
Let $H\in C(\R_+\times D,Z)$, where $Z$ is a Banach space and $D\subset Z$ is a region, and let the IVP \eqref{Lem_IVP}, i.e.,
 $\left\{ \begin{array}{ll}
    \dot{z}=H(t,z), \\
    z(t_0)=z_0,
  \end{array} \right.$
have a solution for each initial point $(t_0,z_0)\in \R_+\times D$. Assume that
$$
\sup\limits_{t\,\in\, [t_0,b],\, z\,\in\,  \Theta}\|H(t,z)\|<\infty
$$
for any fixed $b>t_0$ and any closed bounded set $\Theta\subset D$ such that $\rho(\Theta,\partial D)>0$. Assume also that the solution $z(t)$ remains in $D$ for all $t$ from the maximal interval of existence of $z(t)$ \textup{(}i.e., it can never leave~$D$\textup{)}.
Then $\lim\limits_{t\to t_{max}-0}\|z(t)\|=\infty$ if $t_{max}<\infty$.
 \end{lemma}

 \begin{proof}
The existence of the maximal interval of existence $[t_0,t_{max})$ of the solution $z(t)$ in $D$ follows from \cite[Lemma 3.2]{KK1956}.
From \cite[Lemma 3.3]{KK1956} (or \cite[p.~17, the proof of Corollary]{Schwartz2}) and the fact that the solution cannot leave $D$, we conclude that there is no point $z_*\in \overline{D}$ such that $\lim\limits_{t\to t_{max}-0} z(t)=z_*$. In addition, if $D$ is bounded, then due to \cite[Theorem 2 and its Corollary]{KK1956} the solution exists on $[t_0,\infty)$. Also, note that if $\sup\limits_{t\,\in\, [t_0,t_{max}),\, z\,\in\, D}\|H(t,z)\|<\infty$, then there exists $\lim\limits_{t\to t_{max}-0} z(t)\in \partial D$.

Let us prove that $\lim\limits_{t\to t_{max}-0}\|z(t)\|=\infty$ if $t_{max}<\infty$. It follows from the above that $D$ is unbounded, since $t_{max}=\infty$ otherwise.

Assume that $\|z(t)\|\not\to \infty$ as $t\to t_{max}-0$. Then there exists a closed bounded set $\Theta\subset D$ such that $\rho(\Theta,\partial D)>0$ and $z(t)\in \Theta$ for all $t\in [t_0,t_{max})$. Hence, $\|H(t,z(t))\|\le C$, where $C>0$ is some constant, for all $t\in [t_0,t_{max})$. Consider a sequence $\{t_n\}$ such that $t_0\le t_n <t_{max}$, $t_n\to t_{max}$ as $n\to \infty$. Since $z(t)$ remains in $D$ for all $t$, then $z(t_n)\in D$ for all $n$. As in the proof of Lemma \ref{Lem_Exist-Blow_up}, we obtain that  $\|z(t_n)-z(t_m)\|\le C|t_n-t_m|\to 0$ as $m,n\to \infty$ for any $t_m$, $t_n$. Hence, $z(t_n)$ is a Cauchy sequence and, therefore, it converges to some point $z_*\in D$. Consequently, there exists $\lim\limits_{t\to t_{max}-0} z(t)=z_*\in D$, which contradicts the fact mentioned above.
Thus, it is proved that for any closed bounded set $\Theta\subset D$ such that $\rho(\Theta,\partial D)>0$ there exists $t_*$ \,($t_0<t_*<t_{max}$) such that $z(t)\not\in \Theta$ for $t\in [t_*,t_{max})$ and, therefore, $\lim\limits_{t\to t_{max}-0}\|z(t)\|=\infty$.
 \end{proof}

 \begin{theorem}[\textbf{Lagrange instability (blow-up of solutions)}]\label{Th_Lagr_Instab}
Let the conditions of Lemma \ref{Lem-Alternative} and the following conditions hold:
 \begin{enumerate}[label={\upshape\arabic*.}, ref={\upshape\arabic*}, itemsep=4pt,parsep=0pt,topsep=4pt,leftmargin=0.6cm]
\addtocounter{enumi}{3}
\item\label{InstLagr-a}  There exists an open set $\widetilde{\Omega}\subseteq W=Y_1\dotpl Y_{2\Sigma}$ such that the component $(P_1+P_{2\Sigma})x(t)=(x_1+x_{2\Sigma})(t)$ of each solution $x(t)$ with the initial point $(t_0,x_0)\in L_0$, for which $(P_1+P_{2\Sigma})x_0\in \Omega=\widetilde{A}^{-1}\widetilde{\Omega}$, remains in $\Omega$ for all $t$ from the maximal interval of existence of $x(t)$.

\item\label{InstLagr-b}  There exists a functional
  \begin{equation}\label{Vmin}
   V(w)=\min\limits_{i=1,...,m}V_i(w),
  \end{equation}
  where $w=\widetilde{A}(x_1+x_{2\Sigma})\in W$ and $V_i\in C(W,\R_+)$, $i=1,...,m$, are positive and continuously differentiable functionals on $\widetilde{\Omega}$, and functionals $U\in C(0,\infty)$, $\psi\colon \R_+\to \R_+$ such that $\psi(t)$ is integrable on each finite interval in $\R_+$,
   \begin{equation}\label{Blow-up}
  \int\limits^{\infty}\dfrac{du}{U(u)}<\infty,\qquad   \int\limits^{\infty}\psi(t)dt=\infty,
    \end{equation}
   and for each $(t,\widetilde{A}^{-1}w+x_{20})\in L_0$, for which $w\in \widetilde{\Omega}$, the following inequality holds:
   \begin{equation}\label{InstLagrReg}
  \big\langle \widetilde{A}\Pi(t,\widetilde{A}^{-1}w+x_{20}), \partial_{w}V_{j(w)}(w) \big\rangle \ge U\big(V(w)\big)\psi(t),
   \end{equation}
  where $j(w)\in \{1,...,m\}$ is any index for which $V(w)=V_{j(w)}(w)$ for a given $w$.
 \end{enumerate}
Then for each initial point $(t_0,x_0)\in L_0$ such that $(P_1+P_{2\Sigma})x_0\in \Omega$, there exists a $t_{max}<\infty$ such that the IVP \eqref{ADAE}, \eqref{ini} has a unique solution $x(t)$ on the maximal interval of existence $[t_0,t_{max})$ and $\lim\limits_{t\to t_{max}-0}\|Ax(t)\|=\infty$. In addition, if $D=X$ or the operator $A$ is bounded, then $\lim\limits_{t\to t_{max}-0}\|x(t)\|=\infty$ $($the solution is blow-up in the finite time $[t_0,t_{max})$$)$.
 \end{theorem}

 \begin{proof}
From Lemma \ref{Lem-Alternative} and its proof we deduce that for an arbitrary initial point $(t_0,x_0)\in L_0$  the IVP \eqref{ADAE}, \eqref{ini} has the unique solution $x(t)=\widetilde{A}^{-1}w(t)+\widehat{\eta}(t,w(t))$ on $[t_0,t_{max})$, where $w(t)$ is the solution of the IVP \eqref{ADAE_DE_eta_w}, \eqref{RegDEeta_ini} on $[t_0,t_{max})$ and $\widehat{\eta}\in C(\R_+\times W,D_{20})$ is the function defined in \eqref{ADAE_DE_eta_w}. Also, it follows from Lemma \ref{Lem-Alternative} that if $t_{max}<\infty$, then the theorem statement is fulfilled.

Assume that $t_{max}=\infty$. Consider a sequence $\{t_n\}$ such that $t_n> t_0$ and $t_n\to \infty$ as $n\to \infty$. Due to condition \ref{InstLagr-a}, $w(t_n)\in \widetilde{\Omega}$ for all $n$.
It follows from condition~\ref{InstLagr-b} that the inequality $\big\langle \widehat{\Pi}(t,w), \partial_{w}V_{j(w)}(w)\big\rangle\ge U\big(V(w)\big)\psi(t)$, where $V$, $U$ and $\psi$ are some functionals satisfying condition~\ref{InstLagr-b}, holds for each $t\in \R_+$, $w\in \widetilde{\Omega}$. Then
\begin{equation}\label{InstLagr-ineq}
V(w(t_n))-V(w(t_1))\ge \int\limits_{t_1}^{t_n}U\big(V(w(\tau))\big)\psi(\tau)\, d\tau
\end{equation}
for all $n\ge 1$ due to Lemma \ref{Lem_2.2-ineq_ge}.
From \eqref{InstLagr-ineq} we infer that
\begin{equation}\label{InstLagr-ineq-2}
\int\limits_{V(w(t_1))}^{V(w(t_n))} \dfrac{du}{U(u)}\ge \int\limits_{t_1}^{t_n}\psi(\tau)\, d\tau
\end{equation}
for all $n\ge 1$, which is proved in the same way as \cite[Lemma~2.1]{KK1956}.
Further, passing to the limit as $n\to \infty$, we obtain
\begin{equation}\label{InstLagr-ineq-3}
\infty>\int\limits_{V(w(t_1))}^\infty \dfrac{du}{U(u)}\ge \lim\limits_{n\to\infty}\int\limits_{V(w(t_1))}^{V(w(t_n))} \dfrac{du}{U(u)}\ge \int\limits_{t_1}^{\infty}\psi(\tau)\, d\tau = \infty
\end{equation}
by virtue of \eqref{Blow-up}, which is a contradiction. Hence $t_{max}<\infty$ and the statement of the theorem holds.
\end{proof}

 \begin{theorem}[\textbf{Lagrange instability}]\label{Th_Lagr_Instab_specify}
Let the conditions of Theorem~\ref{Th_Exist-Lip_set} (or Corollary \ref{Corollary_Exist_set}) hold, and let the following conditions, where $\widetilde{M}$, $M_{20}$ and $M_{1,2\Sigma}=\widetilde{A}^{-1}\widetilde{M}$  are defined in Theorem~\ref{Th_Exist-Lip_set}, be satisfied:
 \begin{enumerate}[label={\upshape\arabic*.}, ref={\upshape\arabic*}, itemsep=4pt,parsep=0pt,topsep=4pt,leftmargin=0.6cm]
\addtocounter{enumi}{2}
 \item\label{Alternative_set}
 For any fixed ${a>0}$ and any closed bounded set $\widetilde{S}\subset \widetilde{M}$ such that $\rho(\widetilde{S},\partial \widetilde{M})>0$ the following holds:
  \begin{equation}\label{Alt-1_set}
 \sup\limits_{(t,x)\,\in\, L_0,\; t\,\in\, [0,a],\; w\,\in\, \widetilde{S}} \big\|x_{20}-W_{t,x}^{-1} F_{2*}(t,P_1\widetilde{A}^{-1}w,P_{2\Sigma}\widetilde{A}^{-1}w,x_{20})\big\| \le K \qquad (w=\widetilde{A}(x_1+x_{2\Sigma})),
  \end{equation}
  \begin{equation}\label{Alt-2_set}
 \sup\limits_{(t,x)\,\in\, L_0,\; t\,\in\, [0,a],\; w\,\in\, \widetilde{S},\; \|x_{20}\|\,\le\, K} \|(Q_1+Q_{2\Sigma})f(t,x)\|<\infty \qquad (x=\widetilde{A}^{-1}w+x_{20}),
  \end{equation}
  where ${K=K\big(a,\rho(\widetilde{S},\partial \widetilde{M})\big)}$ is some constant \textup{(}depending on the choice of $a$, $\widetilde{S}$\textup{)}.

\item\label{InstLagr-a_set}
  The component $(P_1+P_{2\Sigma})x(t)=(x_1+x_{2\Sigma})(t)$ of each solution $x(t)$ with the initial point $(t_0,x_0)\in L_0$ such that $(P_1+P_{2\Sigma})x_0\in M_{1,2\Sigma}$, $P_{20}x_0\in M_{20}$ remains in $M_{1,2\Sigma}$ for all $t$ from the maximal interval of existence of $x(t)$.

\item\label{InstLagr-b_set}
  There exists a functional of the form \eqref{Vmin}, i.e., $V(w)=\min\limits_{i=1,...,m}V_i(w)$,
  where $w=\widetilde{A}(x_1+x_{2\Sigma})\in W$ and $V_i\in C(W,\R_+)$, $i=1,...,m$, are positive and continuously differentiable functionals on $\widetilde{M}$, and functionals $U\in C(0,\infty)$, $\psi\colon \R_+\to \R_+$ satisfying \eqref{Blow-up}, such that $\psi(t)$ is integrable on each finite interval in $\R_+$  and for each $(t,\widetilde{A}^{-1}w+x_{20})\in L_0$ for which $w\in \widetilde{M}$ the  inequality \eqref{InstLagrReg}, where $j(w)\in \{1,...,m\}$ is any index for which $V(w)=V_{j(w)}(w)$ for a given $w$, holds.
 \end{enumerate}
Then for each initial point $(t_0,x_0)\in L_0$ such that $(P_1+P_{2\Sigma})x_0\in M_{1,2\Sigma}$ and $P_{20}x_0\in M_{20}$, there exists a $t_{max}<\infty$ such that the IVP \eqref{ADAE}, \eqref{ini} has a unique solution $x(t)$ on the maximal interval of existence $[t_0,t_{max})$ and $\lim\limits_{t\to t_{max}-0}\|Ax(t)\|=\infty$. In addition, if $D=X$ or the operator $A$ is bounded, then $\lim\limits_{t\to t_{max}-0}\|x(t)\|=\infty$ $($the solution blows up in the finite time $[t_0,t_{max})$$)$.
 \end{theorem}

 \begin{proof}
It follows from the proof of Theorem~\ref{Th_Exist-Lip_set} that for an arbitrary initial point $(t_0,x_0)\in L_0$, for which $(P_1+P_{2\Sigma})x_0\in M_{1,2\Sigma}$ and $P_{20}x_0\in M_{20}$, there exists the unique solution $x(t)=\widetilde{A}^{-1}w(t)+\widehat{\eta}(t,w(t))$ of the IVP \eqref{ADAE}, \eqref{ini} on the maximal interval of existence $[t_0,t_{max})$. Here $w(t)$ is the solution of the IVP \eqref{ADAE_DE_eta_w}, \eqref{RegDEeta_ini} on $[t_0,t_{max})$ and $\widehat{\eta}\in C(\R_+\times \widetilde{M},M_{20})$ is the function defined in \eqref{ADAE_DE_eta_w}.

From \eqref{Alt-1_set} we obtain that  $\sup\limits_{t\,\in\, [0,a],\; w\,\in\, \widetilde{S}}\|\widehat{\eta}(t,w)\|\le K$ with some constant ${K=K\big(a,\rho(\widetilde{S},\partial \widetilde{M})\big)}$ for any fixed ${a>0}$ and any closed bounded set $\widetilde{S}\subset \widetilde{M}$ such that $\rho(\widetilde{S},\partial \widetilde{M})>0$.
Then, using \eqref{Alt-2_set}, we have
$\sup\limits_{t\,\in\,[0,a],\, w\,\in\, \widetilde{S}} \|(Q_1+Q_{2\Sigma})f(t,\widetilde{A}^{-1}w+\widehat{\eta}(t,w))\| <\infty$ and, hence,
$\sup\limits_{t\,\in\, [0,b],\, w\,\in\, \widetilde{S}} \|\widehat{\Pi}(t,w)\|<\infty$
for any fixed $b>0$ and any closed bounded set $\widetilde{S}\subset \widetilde{M}$ such that $\rho(\widetilde{S},\partial \widetilde{M})>0$.

Thus, it follows from Lemma \ref{Lem_Blow_up} that if $t_{max}<\infty$, then $\lim\limits_{t\to t_{max}-0}\|w(t)\|=\infty$.

Assume that $t_{max}=\infty$. Consider a sequence $\{t_n\}$ such that $t_n> t_0$ and $t_n\to \infty$ as $n\to \infty$. Due to condition \ref{InstLagr-a_set}, $w(t_n)\in \widetilde{M}$ for all $n$.
It follows from condition~\ref{InstLagr-b_set} that the inequality $\big\langle \widehat{\Pi}(t,w), \partial_{w}V_{j(w)}(w)\big\rangle\ge U\big(V(w)\big)\psi(t)$, where $V$, $U$ and $\psi$ are some functionals satisfying condition~\ref{InstLagr-b_set}, holds for each $t\in \R_+$, $w\in \widetilde{M}$. Then the inequality \eqref{InstLagr-ineq} holds for all $n\ge 1$ due to Lemma \ref{Lem_2.2-ineq_ge}.
From \eqref{InstLagr-ineq} we obtain that \eqref{InstLagr-ineq-2} holds for all $n\ge 1$, which is proved in the same way as \cite[Lemma~2.1]{KK1956}.
Further, passing to the limit as $n\to \infty$, we obtain \eqref{InstLagr-ineq-3} by virtue of \eqref{Blow-up}, which is a contradiction. Hence $t_{max}<\infty$ and the statement of the theorem holds.
\end{proof}

 \begin{remark}\label{Rem_Appr1_Semi-inv}
In the above theorems, lemmas and corollaries one can use the semi-inverse operator $A^{(-1)}$ defined by \eqref{Semi-inverse} instead  of the inverse $\widetilde{A}^{-1}$ of \eqref{widetilde_A}, taking into account that ${A^{(-1)}=\widetilde{A}^{-1}(Q_1+Q_{2\Sigma})}$,\; $\widetilde{A}^{-1}\widetilde{M}=A^{(-1)}\widetilde{M}$ \;($\widetilde{M}\subseteq Y_1\dotpl Y_{2\Sigma}$),\; $\widetilde{A}(x_1+x_{2\Sigma})=A(x_1+x_{2\Sigma})$,\; $\widetilde{A}^{-1}w=A^{(-1)}w$ \;($w\in Y_1\dotpl Y_{2\Sigma}$),\;
$\widetilde{A}\,\Pi(t,x)=A\,\Pi(t,x)$,\; $\widetilde{A}^{-1}\widetilde{\Omega}=A^{(-1)}\widetilde{\Omega}$ \;($\widetilde{\Omega}\subseteq Y_1\dotpl Y_{2\Sigma}$).
 \end{remark}

 \subsection{Second approach}\label{Approach_2}

In Section \ref{Approach_1}, the case when the component $x_{20}$ can be defined as an implicit function from equation \eqref{AE} (i.e., when $Q_{2*}f(t,x)$ depends on $x_{20}$) was considered.
In the present section this is not required, i.e., a more general case is considered, but on the other hand it is assumed that the function $f$ will have a special structure.

As mentioned in Remark \ref{Rem_A_D}, for convenience we assume that $D_B\supseteq D_A$.

Equation \eqref{1_1} is equivalent to $\frac{d}{dt}[P_1x]+\widetilde{A}^{-1}BP_1x=\widetilde{A}^{-1}Q_1f(t,x)$, or $\frac{d}{dt}[P_1x]+A^{(-1)}BP_1x=A^{(-1)}Q_1f(t,x)$, or $\frac{d}{dt}[P_1x]+G^{-1}BP_1x=G^{-1}Q_1f(t,x)$, where $\widetilde{A}^{-1}\in \spL(Y,X)$ is the inverse of
\eqref{widetilde_A}, $A^{(-1)}\in \spL(Y,X)$ is the semi-inverse of $A$, and $G^{-1}\in \spL(Y,X)$ is the inverse of $G=Q_1A+Q_2B$ \cite[Subsection~3.3]{Vlasenko1}, since $\widetilde{A}^{-1}AP_1=A^{(-1)}AP_1=G^{-1}AP_1=P_1$.

Denote
\begin{equation}\label{P_2wedge}
x_{2j}^{(j+1)}=P_{2j}^{(j+1)}x,\; j=0,...,\nu-1;\quad x_{2\wedge,j}=P_{2\wedge,j}\,x,\; j=0,...,\nu-2;\quad
x_{2j}=P_{2j}x,\; j=0,...,\nu-1,
\end{equation}
and note that $x_{2(\nu-1)}=x_{2(\nu-1)}^{(\nu)}$ and $x_{2j}=x_{2j}^{(j+1)}+x_{2\wedge,j}$, $j=0,...,\nu-2$ (see \eqref{Q_2Sigma-P_2wedge}). As above, $x_1=P_1x$. Thus,
 \begin{equation}\label{xrr-ind_Appr-2}
x=x_1+x_2,\quad
x_2=\sum\limits_{j=0}^{\nu-1} x_{2j}= \sum\limits_{j=0}^{\nu-1}x_{2j}^{(j+1)}+ \sum\limits_{j=0}^{\nu-2}x_{2\wedge,j}= x_{20}+x_{2\Sigma},\quad x_{2\Sigma}=\sum\limits_{j=1}^{\nu-1}x_{2j}^{(j+1)}+ \sum\limits_{j=1}^{\nu-2}x_{2\wedge,j},
 \end{equation}
This representation of $x_2$ correspond to the direct decomposition (cf. \eqref{DY2rr})
\begin{equation}\label{D2_Appr-2}
D_2=D_{20}\dotpl D_{2\Sigma}.
\end{equation}

Assume that the function $f$ has the following structure:
 \begin{equation}\label{f-Appr-2}
f(t,x)=\big(Q_1+Q_{20}\big)f(t,x)+  \sum\limits_{s=1}^{\nu-2}Q_{2\Sigma,s} f\bigg(t, \sum\limits_{j=1}^{\nu-1}x_{2j}^{(j+1)}+ \sum\limits_{j=s}^{\nu-2}x_{2\wedge,j}\bigg)+ \sum\limits_{s=1}^{\nu-1}Q_{2s}^{(s+1)} f\bigg(t,\sum\limits_{j=s}^{\nu-1}x_{2j}^{(j+1)}\bigg).
 \end{equation}
Then, taking into account the aforementioned, the system \eqref{1_1}, \eqref{2_20},  \eqref{2_2Sigma_s-}, \eqref{2_2Sigma_nu-2}, \eqref{3_2*_s-}, \eqref{3_2*_nu-1} takes the form
 \begin{align}
 \frac{d}{dt}x_1+\widetilde{A}^{-1}Bx_1 &=\widetilde{A}^{-1}Q_1f(t,x), \label{x_1}\\
 \frac{d}{dt}\Big[A\Big(x_{21}^{(2)}+x_{2\wedge,1}\Big)\Big]+ Bx_{20} &= Q_{20}f(t,x)   \qquad\qquad (\nu\ge 2), \label{x_20}
 \end{align}
\begin{subequations}\label{x_2wedge,sSyst}
\renewcommand{\theequation}{\theparentequation.s}
 \begin{align}
  \qquad\qquad  & \vdots  && \nonumber\\
 \frac{d}{dt}\Big[A\Big(x_{2(s+1)}^{(s+2)}+x_{2\wedge,s+1}\Big)\Big]+ Bx_{2\wedge,s} &= Q_{2\Sigma,s}f\bigg(t, \sum\limits_{j=1}^{\nu-1}x_{2j}^{(j+1)}+ \sum\limits_{i=s}^{\nu-2}x_{2\wedge,i}\bigg)\quad & (s=1,...,\nu-3) &
     \label{x_2wedge,s}\\
  \qquad\qquad  & \vdots  & (\nu\ge 4), &   \nonumber
 \end{align}
\end{subequations}
\begin{equation}
 \frac{d}{dt}\Big[Ax_{2(\nu-1)}\Big]+Bx_{2\wedge,\nu-2} = Q_{2\Sigma,\nu-2}f\bigg(t, \sum\limits_{j=1}^{\nu-1}x_{2j}^{(j+1)}+ x_{2\wedge,\nu-2}\bigg)
   \qquad(\nu\ge 3), \label{x_2wedge,nu-2}
\end{equation}
\begin{subequations}\label{x_2s_(s+1)Syst}
\renewcommand{\theequation}{\theparentequation.s}
 \begin{align}
   \qquad\qquad  & \vdots  \nonumber\\
 Bx_{2s}^{(s+1)} &=Q_{2s}^{(s+1)}f\bigg(t, \sum\limits_{j=s}^{\nu-1}x_{2j}^{(j+1)}\bigg) \qquad (s=1,...,\nu-2)
  \qquad\quad(\nu\ge 3),   \label{x_2s_(s+1)} \\
    \qquad\qquad  & \vdots  \nonumber
 \end{align}
\end{subequations}
 \begin{equation}
Bx_{2(\nu-1)} =Q_{2(\nu-1)}f\big(t,x_{2(\nu-1)}\big)
  \qquad\qquad(\nu\ge 2). \hspace{2cm} \label{x_2(nu-1)}
 \end{equation}
This system consists of $2\nu-1$ equations with $2\nu-1$ variables and, hence, is closed. It is equivalent to the DAE \eqref{ADAE} where $f$ is of the form \eqref{f-Appr-2}. In \eqref{x_1} we can take $\widetilde{A}^{-1}Q_1=A^{(-1)}Q_1$, $\widetilde{A}^{-1}Bx_1=A^{(-1)}Bx_1$.

Notice that the reference \eqref{x_2s_(s+1)} (respectively, (\ref{x_2s_(s+1)Syst}.i)\,) means that we consider the equation from \eqref{x_2s_(s+1)} number $s$ (respectively, number $i$), for example, \eqref{x_2s_(s+1)}, $s=\nu-2$, means $Bx_{2(\nu-2)}^{(\nu-1)} =Q_{2(\nu-2)}^{(\nu-1)}f\bigg(t,x_{2(\nu-2)}^{(\nu-1)}+x_{2(\nu-1)}\bigg)$, and (\ref{x_2s_(s+1)Syst}.i), $i=\nu-2$, means the same.  Similarly for the reference \eqref{x_2wedge,s} (respectively, (\ref{x_2wedge,sSyst}.i)\,).

Note that in the case when $P(\la)$ has index 1,  $P_{20}^{(1)}=P_2$ and $Q_{20}^{(1)}=Q_2$ (since $D_{20}^1=D_2$, $Y_{20}^{(1)}=Y_2$). Thus, \emph{for the pencil $P(\la)$ of index 1, the system  \eqref{x_1}--\eqref{x_2(nu-1)} becomes the system of equations \eqref{x_1},  $Bx_2=Q_2f(t,x)$} (\emph{that is, \eqref{x_20} where $x_{20}=x_2$, $Q_{20}=Q_2$, $x_{2\wedge,1}=0$, $x_{21}^{(2)}=0$}) and it is equivalent to the system \eqref{DE_ind-1}, \eqref{AE_ind-1}.

The initial condition for the system \eqref{x_1}--\eqref{x_2(nu-1)} (accordingly, for the DAE \eqref{ADAE} with $f$ of the form \eqref{f-Appr-2}) will be considered as
  \begin{equation}\label{ini-2}
P_1x(t_0)=x_{0,1},
 \end{equation}
or $Q_1Ax(t_0)=u_0$ (then $P_1x(t_0)=\widetilde{A}^{-1}u_0=A^{(-1)}u_0$\,).

Below we use the following functions and the following equations equivalent to \eqref{x_2s_(s+1)}, ${s=1,...,\nu-2}$, \eqref{x_2(nu-1)}:
  \begin{subequations}\label{x_2s_(s+1)Syst+}
  \renewcommand{\theequation}{\theparentequation.s}
 \begin{equation}
\begin{split}
F_{2s}^{(s+1)} & \big(t,x_{2j}^{(j+1)},j=s,...,\nu-1\big)=0,  \\
& F_{2s}^{(s+1)}\big(t,x_{2j}^{(j+1)},j=s,...,\nu-1\big):= Q_{2s}^{(s+1)}f\bigg(t, \sum\limits_{j=s}^{\nu-1}x_{2j}^{(j+1)}\bigg) -Bx_{2s}^{(s+1)} \quad (s=1,...,\nu-2),  \label{x_2s_(s+1)+}
\end{split}
 \end{equation}
  \end{subequations}
where $\big(t,x_{2j}^{(j+1)},j=s,...,\nu-1\big)= \big(t,x_{2s}^{(s+1)},x_{2(s+1)}^{(s+2)},...,x_{2\nu-1}\big)$,
 \begin{equation}
F_{2(\nu-1)}\big(t,x_{2(\nu-1)}\big)=0,\qquad
F_{2(\nu-1)}\big(t,x_{2(\nu-1)}\big):= Q_{2(\nu-1)}f\big(t,x_{2(\nu-1)}\big)-Bx_{2(\nu-1)}. \label{x_2(nu-1)+}
 \end{equation}
Also, we will use the following functions and the following equations equivalent to \eqref{x_2wedge,s}, ${s=1,...,\nu-3}$, \eqref{x_2wedge,nu-2}, \eqref{x_20} where $\frac{d}{dt}x_{2(s+1)}^{(s+2)}$, $\frac{d}{dt}x_{2\wedge,s+1}$ (${s=0,...,\nu-3}$) and $\frac{d}{dt}x_{2(\nu-1)}$ are replaced by $d_{2(s+1)}^{(s+2)}$, $d_{2\wedge,s+1}$ and $d_{2(\nu-1)}$:
   \begin{subequations}\label{x_2wedge,sSyst+}
   \renewcommand{\theequation}{\theparentequation.s}
 \begin{equation}
\begin{split}
F_{2\Sigma,s} & \Big(t, x_{2j}^{(j+1)},j=1,...,\nu-1, x_{2\wedge,i},i=s,...,\nu-2, d_{2(s+1)}^{(s+2)}, d_{2\wedge,s+1}\Big) =0, \\
& F_{2\Sigma,s} \Big(t, x_{2j}^{(j+1)},j=1,...,\nu-1, x_{2\wedge,i},i=s,...,\nu-2, d_{2(s+1)}^{(s+2)}, d_{2\wedge,s+1}\Big):=  \\
& :=Q_{2\Sigma,s}f\bigg(t, \sum\limits_{j=1}^{\nu-1}x_{2j}^{(j+1)}+ \sum\limits_{i=s}^{\nu-2}x_{2\wedge,i}\bigg) -Bx_{2\wedge,s} -A\Big[d_{2(s+1)}^{(s+2)}+d_{2\wedge,s+1}\Big] \quad (s=1,...,\nu-3), \label{x_2wedge,s+}
\end{split}
 \end{equation}
   \end{subequations}
 \begin{equation}
\begin{split}
F_{2\Sigma,\nu-2} \big(t,x_{2j}^{(j+1)}, &\, j=1,...,\nu-1, x_{2\wedge,\nu-2}, d_{2(\nu-1)}\big)=0, \\
F_{2\Sigma,\nu-2} & \big(t, x_{2j}^{(j+1)},j=1,...,\nu-1,x_{2\wedge,\nu-2}, d_{2(\nu-1)}\big):= \\
& :=Q_{2\Sigma,\nu-2}f\bigg(t, \sum\limits_{j=1}^{\nu-1}x_{2j}^{(j+1)}+x_{2\wedge,\nu-2}\bigg) -Bx_{2\wedge,\nu-2} -Ad_{2(\nu-1)},  \label{x_2wedge,nu-2+}
\end{split}
 \end{equation}
 \begin{equation}
\begin{split}
F_{20} & \big(t,x_1,x_{20},x_{2\Sigma},d_{21}^{(2)},d_{2\wedge,1}\big)=0, \\
& F_{20}\big(t,x_1,x_{20},x_{2\Sigma},d_{21}^{(2)},d_{2\wedge,1}\big):= Q_{20}f\big(t,x_1+x_{20}+x_{2\Sigma}\big) - Bx_{20} -Q_{20}A\Big[d_{21}^{(2)}+d_{2\wedge,1}\Big].  \label{x_20+}
\end{split}
 \end{equation}

Note that in this subsection we use norms defined in the same way as in Remark~\ref{Rem_Norms}.

 \begin{theorem}[\textbf{Solvability}]\label{Th_Exist-Lip_set-2}
Let the function $f\in C(\R_+\times D,Y)$ have the structure \eqref{f-Appr-2}, ${D_B\supseteq D_A=D}$, ${\dim\ker A=n<\infty}$, and let ${\la A+B}$ be a regular pencil of index $\nu$.
Assume that there exists open sets ${\widetilde{M}_1\subseteq Y_1}$ and ${M_{20}\subseteq D_{20}}$, ${M_{2j}^{(j+1)}\subseteq D_{2j}^{(j+1)}}$, ${M_{2\wedge,j}\subseteq D_{2\wedge,j}}$, ${j=1,...,\nu-2}$, ${M_{2(\nu-1)}=M_{2(\nu-1)}^{(\nu)}\subseteq D_{2(\nu-1)}}$ such that
 \begin{align*}
& Q_{2s}^{(s+1)}f\in C^{\nu-2}\bigg(\R_+\times \bigg(\dotpl\limits_{i=s}^{\nu-1}M_{2i}^{(i+1)}\bigg),Y_{2s}^{(s+1)}\bigg)\!,\quad
 Q_{2\Sigma,s}f\in C^s\bigg(\R_+\times \bigg(\dotpl\limits_{i=1}^{\nu-1}M_{2i}^{(i+1)}\dotpl \dotpl\limits_{i=s}^{\nu-2}M_{2\wedge,i}\bigg),Y_{2\Sigma,s}\bigg)\!,   \\
&  s=1,...,\nu-2, \qquad
 Q_{2(\nu-1)}f\in C^{\nu-1}(\R_+\times M_{2(\nu-1)},Y_{2(\nu-1)}),
 \end{align*}
and the following conditions where \,${M_1=\widetilde{A}^{-1}\widetilde{M}_1\subseteq D_1}$,
\,${M_{2\Sigma}= \dotpl\limits_{i=1}^{\nu-1}M_{2i}^{(i+1)}}\dotpl  \dotpl\limits_{i=1}^{\nu-2}M_{2\wedge,i}$ and
${M=M_1\dotpl M_{20}\dotpl M_{2\Sigma}}$\,  hold:
  \begin{enumerate}[label={\upshape{1.\alph*.}}, ref={\upshape{1.\alph*}},itemsep=4pt,parsep=0pt,topsep=5pt,leftmargin=1cm]
 \item\label{Sogl_set-1a}  For each  ${t\in \R_+}$ there exists a unique ${x_{2(\nu-1)}\in M_{2(\nu-1)}}$ such that \,$t,\,x_{2(\nu-1)}$ satisfy \eqref{x_2(nu-1)}.

 \item\label{Sogl_set-1b} For each ${t\in \R_+}$, ${x_{2j}^{(j+1)}\in M_{2j}^{(j+1)}}$, $j=s+1,...,\nu-1$, there exists a unique ${x_{2s}^{(s+1)}\in M_{2s}^{(s+1)}}$ such that \,$t,\,x_{2j}^{(j+1)}$, $j=s,...,\nu-1$, satisfy \eqref{x_2s_(s+1)}. This holds for each $s$, $s=1,...,\nu-2$.

 \item\label{Sogl_set-1c} For each ${t\in \R_+}$, ${x_{2j}^{(j+1)}\in M_{2j}^{(j+1)}}$, $j=1,...,\nu-1$, $d_{2(\nu-1)}\in D_{2(\nu-1)}$ there exists a unique $x_{2\wedge,\nu-2}\in M_{2\wedge,\nu-2}$ such that $t$, $x_{2\wedge,\nu-2}$, $x_{2j}^{(j+1)}$, $j=1,...,\nu-1$, $d_{2(\nu-1)}$ satisfy  \eqref{x_2wedge,nu-2+}.

 \item\label{Sogl_set-1d}  For each ${t\in \R_+}$, ${x_{2j}^{(j+1)}\in M_{2j}^{(j+1)}}$, $j=1,...,\nu-1$, $x_{2\wedge,i}\in M_{2\wedge,i}$, $i=s+1,...,\nu-2$, $d_{2\wedge,s+1}\in D_{2\wedge,s+1}$, $d_{2(s+1)}^{(s+2)}\in D_{2(s+1)}^{(s+2)}$ there exists a unique $x_{2\wedge,s}\in M_{2\wedge,s}$ such that $t$, $x_{2\wedge,i}$, $i=s,...,\nu-2$, $x_{2j}^{(j+1)}$, $j=1,...,\nu-1$, $d_{2\wedge,s+1}$, $d_{2(s+1)}^{(s+2)}$  satisfy  \eqref{x_2wedge,s+}. This holds for each $s$, $s=1,...,\nu-3$.

 \item\label{Sogl_set-1e}  For each ${t\in \R_+}$, ${x_1\in M_1}$, ${x_{2\Sigma}\in M_{2\Sigma}}$, ${d_{2\wedge,1}\in D_{2\wedge,1}}$, ${d_{21}^{(2)}\in D_{21}^{(2)}}$ there exists a unique ${x_{20}\in M_{20}}$ such that \,$t$, $x_1$, $x_{20}$, $x_{2\Sigma}$, $d_{2\wedge,1}$, $d_{21}^{(2)}$  satisfy  \eqref{x_20+}.
  \end{enumerate}

    \begin{enumerate}[label={\upshape{2.\alph*.}}, ref={\upshape{2.\alph*}}, itemsep=4pt,parsep=0pt,topsep=5pt,leftmargin=1cm]
 \item\label{Inv-set-2a} For any fixed ${\hat{t}\in \R_+}$, ${\hat{x}_{2(\nu-1)}\in M_{2(\nu-1)}}$ satisfying \eqref{x_2(nu-1)}, the operator
   \begin{equation*}
  \partial_{x_{2(\nu-1)}}F_{2(\nu-1)}\big(\hat{t},\hat{x}_{2(\nu-1)}\big)\!=\! \left[\partial_x (Q_{2(\nu-1)}f)(\hat{t},\hat{x}_{2(\nu-1)})-Q_{2(\nu-1)}B\right] P_{2(\nu-1)}\big|_{D_{2(\nu-1)}} \colon D_{2(\nu-1)}\to Y_{2(\nu-1)}
   \end{equation*}
  has the inverse $\left[\partial_{x_{2(\nu-1)}}F_{2(\nu-1)} \big(\hat{t},\hat{x}_{2(\nu-1)}\big)\right]^{-1} \in \spL(Y_{2(\nu-1)},D_{2(\nu-1)})$.

 \item\label{Inv-set-2b} For any fixed ${\hat{t}\in \R_+}$, ${\hat{x}_{2i}^{(i+1)}\in M_{2i}^{(i+1)}}$, $i=s,...,\nu-1$ \textup{(}recall that  $\hat{x}_{2(\nu-1)}= \hat{x}_{2(\nu-1)}^{(\nu)}$\textup{)},  satisfying \eqref{x_2s_(s+1)}, the operator
   \begin{multline*}
 \partial_{x_{2s}^{(s+1)}}F_{2s}^{(s+1)} \big(\hat{t},\hat{x}_{2i}^{(i+1)},i=s,...,\nu-1\big) = \\ =\left[\partial_x (Q_{2s}^{(s+1)}f)\bigg(\hat{t}, \sum\limits_{i=s}^{\nu-1}\hat{x}_{2i}^{(i+1)}\bigg)- Q_{2s}^{(s+1)}B\right] P_{2s}^{(s+1)}\big|_{D_{2s}^{(s+1)}} \colon D_{2s}^{(s+1)}\to Y_{2s}^{(s+1)}
   \end{multline*}
  has an inverse belonging to $\spL\big(Y_{2s}^{(s+1)},D_{2s}^{(s+1)}\big)$.
  This holds for each $s$, $s=1,...,\nu-2$.

 \item\label{Inv-set-2c} For any fixed ${\hat{t}\in \R_+}$, ${\hat{x}_{2j}^{(j+1)}\in M_{2j}^{(j+1)}}$, $j=1,...,\nu-1$,  $\hat{x}_{2\wedge,\nu-2}\in M_{2\wedge,\nu-2}$, $\hat{d}_{2(\nu-1)}\in D_{2(\nu-1)}$ satisfying \eqref{x_2wedge,nu-2+},   the operator
   \begin{multline*}
  \partial_{x_{2\wedge,\nu-2}} F_{2\Sigma,\nu-2} \Big(\hat{t}, \hat{x}_{2j}^{(j+1)},j=1,...,\nu-1, \hat{x}_{2\wedge,\nu-2}, \hat{d}_{2(\nu-1)}\Big)= \\ \left[\partial_x (Q_{2\Sigma,\nu-2}f) \bigg(\hat{t}, \sum\limits_{j=1}^{\nu-1}\hat{x}_{2j}^{(j+1)}+\hat{x}_{2\wedge,\nu-2}\bigg)- Q_{2\Sigma,\nu-2}B\right] P_{2\wedge,\nu-2}\big|_{D_{2\wedge,\nu-2}} \colon D_{2\wedge,\nu-2}\to Y_{2\Sigma,\nu-2}
   \end{multline*}
  has an inverse belonging to $\spL(Y_{2\Sigma,\nu-2},D_{2\wedge,\nu-2})$.

 \item\label{Inv-set-2d} For any fixed ${\hat{t}\in \R_+}$, ${\hat{x}_{2j}^{(j+1)}\in M_{2j}^{(j+1)}}$, $j=1,...,\nu-1$,  $\hat{x}_{2\wedge,i}\in M_{2\wedge,i}$, $i=s,...,\nu-2$, $\hat{d}_{2\wedge,s+1}\in D_{2\wedge,s+1}$, $\hat{d}_{2(s+1)}^{(s+2)}\in D_{2(s+1)}^{(s+2)}$ satisfying \eqref{x_2wedge,s+},   the operator
   \begin{multline*}
 \partial_{x_{2\wedge,s}} F_{2\Sigma,s}\Big(\hat{t}, \hat{x}_{2j}^{(j+1)},j=1,...,\nu-1, \hat{x}_{2\wedge,i},i=s,...,\nu-2, \hat{d}_{2(s+1)}^{(s+2)},\hat{d}_{2\wedge,s+1}\Big) = \\ =\left[\partial_x (Q_{2\Sigma,s}f)\bigg(\hat{t}, \sum\limits_{j=1}^{\nu-1}\hat{x}_{2j}^{(j+1)}+  \sum\limits_{i=s}^{\nu-2}\hat{x}_{2\wedge,i}\bigg)- Q_{2\Sigma,s}B\right] P_{2\wedge,s}\big|_{D_{2\wedge,s}} \colon D_{2\wedge,s}\to Y_{2\Sigma,s}
   \end{multline*}
  has an inverse belonging to $\spL(Y_{2\Sigma,s},D_{2\wedge,s})$.
  This holds for each $s$, $s=1,...,\nu-3$.

 \item\label{Inv-Lipsch_set-2e}
 The components $Q_1f(t,x_1+x_{20}+x_{2\Sigma})$ and $Q_{20}f(t,x_1+x_{20}+x_{2\Sigma})$ of the function $f(t,x)$ satisfy locally a Lipschitz condition with respect to $x_1+x_{20}$ on $\R_+\times M$.  \\
 For any fixed $\hat{t}\in \R_+$, $\hat{x}_1\in M_1$, $\hat{x}_{20}\in M_{20}$, $\hat{x}_{2\Sigma}\in M_{2\Sigma}$ \textup{(}accordingly,  $\hat{x}=\hat{x}_1+\hat{x}_{20}+\hat{x}_{2\Sigma}\in M$\textup{)},
 $\hat{d}_{21}^{(2)}\in D_{21}^{(2)}$, $\hat{d}_{2\wedge,1}\in D_{2\wedge,1}$ satisfying \eqref{x_20+}
 there exist open neighborhoods
 $U_{\delta}\big(\hat{t},\hat{x}_1,\hat{x}_{2\Sigma}, \hat{d}_{21}^{(2)},\hat{d}_{2\wedge,1}\big)= U_{\delta_1}\big(\hat{t}\big) \times U_{\delta_2}\big(\hat{x}_1\big) \times U_{\delta_3}\big(\hat{x}_{2\Sigma}\big)  \times U_{\delta_4}\big(\hat{d}_{21}^{(2)}\big)\times U_{\delta_5}\big(\hat{d}_{2\wedge,1}\big)\subset
 \R_+\times D_1\times D_{2\Sigma}\times D_{21}^{(2)}\times D_{2\wedge,1}$ and $U_\varepsilon\big(\hat{x}_{20}\big)\subset D_{20}$ and an operator $\mathrm{Z}= \mathrm{Z}_{\hat{t},\hat{x},\hat{d}_{21}^{(2)},\hat{d}_{2\wedge,1}} \in \spL(D_{20},Y_{20})$ having the inverse $\mathrm{Z}^{-1}\in \spL(Y_{20},D_{20})$ such that for each $t\in U_{\delta_1}\big(\hat{t}\big)$,
 $x_1\in U_{\delta_2}\big(\hat{x}_1\big)$,
 $x_{2\Sigma}\in U_{\delta_3}\big(\hat{x}_{2\Sigma}\big)$,
 $d_{21}^{(2)}\in U_{\delta_4}\big(\hat{d}_{21}^{(2)}\big)$, $d_{2\wedge,1}\in U_{\delta_5}\big(\hat{d}_{2\wedge,1}\big)$ and each $x'_{20},\, x''_{20}\in U_\varepsilon\big(\hat{x}_{20}\big)$ the mapping $F_{20}$ satisfies the inequality
  \begin{equation*}
 \big\|F_{20}\big(t,x_1,x'_{20},x_{2\Sigma},d_{21}^{(2)},d_{2\wedge,1}\big)-  F_{20}\big(t,x_1,x''_{20},x_{2\Sigma},d_{21}^{(2)},d_{2\wedge,1}\big)- \mathrm{Z}[x'_{20}-x''_{20}]\big\|\le k(\delta,\varepsilon)\big\|x'_{20}-x''_{20}\big\|,
  \end{equation*}
  where $k(\delta,\varepsilon)$ is such that $\lim\limits_{\delta,\,\varepsilon\to 0} k(\delta,\varepsilon)< \|\mathrm{Z}^{-1}\|^{-1}$  \,\textup{(}the numbers $\delta, \varepsilon>0$ depend on the choice of $\hat{t},\hat{x}$, $\hat{d}_{21}^{(2)}$, $\hat{d}_{2\wedge,1}$\textup{)}.
    \end{enumerate}
Then for each initial value $t_0\in \R_+$ and $x_{0,1}\in M_1$ there exists a $t_{max}\le \infty$ such that the IVP \eqref{ADAE}, \eqref{f-Appr-2}, \eqref{ini-2} has a unique solution $x(t)$ in $M$ on the maximal interval of existence $[t_0,t_{max})$.
 \end{theorem}

 \begin{corollary}[\textbf{Solvability}]\label{Corollary_Exist_set-2}
Theorem~\ref{Th_Exist-Lip_set-2} remains valid if condition \ref{Inv-Lipsch_set-2e} is replaced by
\begin{enumerate}[label={\upshape{2.\alph*.}}, ref={\upshape{2.\alph*}}, itemsep=4pt,parsep=0pt,topsep=5pt,leftmargin=1cm]
\addtocounter{enumi}{4}
\item\label{Inv_set-2e} The components $Q_1f(t,x)$ and $Q_{20}f(t,x)$ of the function $f(t,x)$ have the continuous partial derivatives with respect to $x$ on $\R_+\times M$.  \\
 For any fixed $\hat{t}\in \R_+$, $\hat{x}_1\in M_1$, $\hat{x}_{20}\in M_{20}$, $\hat{x}_{2\Sigma}\in M_{2\Sigma}$ \textup{(}accordingly,  $\hat{x}=\hat{x}_1+\hat{x}_{20}+\hat{x}_{2\Sigma}\in M$\textup{)},
  $\hat{d}_{21}^{(2)}\in D_{21}^{(2)}$, $\hat{d}_{2\wedge,1}\in D_{2\wedge,1}$ satisfying \eqref{x_20+}, the operator
     \begin{equation}\label{funcPhiInvReg-2}
 \mathrm{Z}_{\hat{t},\hat{x},\hat{d}_{21}^{(2)},\hat{d}_{2\wedge,1}}\!:=\! \partial_{x_{20}}F_{20}\big(\hat{t},\hat{x}_1,\hat{x}_{20},\hat{x}_{2\Sigma}, \hat{d}_{21}^{(2)},\hat{d}_{2\wedge,1}\big)\!=\! \Big[\partial_x (Q_{20}f)(\hat{t},\hat{x})-Q_{20}B\Big]P_{20}\big|_{D_{20}}\colon D_{20}\to Y_{20}
     \end{equation}
 has the inverse $\mathrm{Z}_{\hat{t},\hat{x},\hat{d}_{21}^{(2)},\hat{d}_{2\wedge,1}}^{-1}\in \spL(Y_{20},D_{20})$.
\end{enumerate}
 \end{corollary}

  \begin{proof}[Proof of Theorem \ref{Th_Exist-Lip_set-2}]
As shown above, the DAE \eqref{ADAE}, \eqref{f-Appr-2} is equivalent to the system the system \eqref{x_1}--\eqref{x_2(nu-1)}.  First, note that
 \begin{equation}\label{smooth_F_(2s)_(s+1)}
 \begin{split}
& F_{2(\nu-1)}\in C^{\nu-1}(\R_+\times M_{2(\nu-1)},Y_{2(\nu-1)}),\quad F_{2s}^{(s+1)}\in C^{\nu-2}\left(\R_+\times \timesO\limits_{i=s}^{\nu-1}M_{2i}^{(i+1)},Y_{2s}^{(s+1)}\right),  \\
& F_{2\Sigma,s}\in C^s\left(\R_+\times \timesO\limits_{i=1}^{\nu-1}M_{2i}^{(i+1)}\times \timesO\limits_{i=s}^{\nu-2}M_{2\wedge,i}\times D_{2(s+1)}^{(s+2)}\times D_{2\wedge,s+1},Y_{2\Sigma,s}\right),\quad s=1,...,\nu-2,
 \end{split}
 \end{equation}
due to smoothness of $Q_{2(\nu-1)}f$, $Q_{2s}^{(s+1)}f$ and $Q_{2\Sigma,s}f$.

Due to \eqref{smooth_F_(2s)_(s+1)}, conditions \ref{Sogl_set-1a},  \ref{Inv-set-2a}, \ref{Sogl_set-1b} and \ref{Inv-set-2b} and Theorem \ref{Th_Nonloc_ImplFunc-3}, there exist the functions $\eta_{2(\nu-1)}\in C^{\nu-1}(\R_+,M_{2(\nu-1)})$, $\eta_{2s}^{(s+1)}\in C^{\nu-2}\Big(\R_+\times M_{2(s+1)}^{(s+2)}\times ...\times M_{2(\nu-1)},M_{2s}^{(s+1)}\Big)$, $s=1,...,\nu-2$, which are unique solutions of equations \eqref{x_2(nu-1)+}, \eqref{x_2s_(s+1)+} (accordingly, \eqref{x_2(nu-1)}, \eqref{x_2s_(s+1)}), $s=1,...,\nu-2$, with respect to $x_{2(\nu-1)}$, $x_{2s}^{(s+1)}$, respectively.

Let us successively define
 $$
\widetilde{\eta}_{2(\nu-1)}(t)=\eta_{2(\nu-1)}(t),\; \widetilde{\eta}_{2(\nu-2)}^{(\nu-1)}(t)= \eta_{2(\nu-2)}^{(\nu-1)}\big(t,\widetilde{\eta}_{2(\nu-1)}(t)\big),\; ...,\; \widetilde{\eta}_{21}^{(2)}(t)= \eta_{21}^{(2)}\big(t,\widetilde{\eta}_{2i}^{(i+1)},i=s+1,...,\nu-1\big),
 $$
$t\in \R_+$, and set $\widetilde{\eta}_{2(\nu-1)}\equiv \widetilde{\eta}_{2(\nu-1)}^{(\nu)}$ \,(since $x_{2(\nu-1)}=x_{2(\nu-1)}^{(\nu)}$).
Then the functions $\widetilde{\eta}_{2s}^{(s+1)}\in C^{\nu-2}\big(\R_+,M_{2s}^{(s+1)}\big)$, $s=1,...,\nu-2$, $\widetilde{\eta}_{2(\nu-1)}\in C^{\nu-1}(\R_+,M_{2(\nu-1)})$ satisfy the system \eqref{x_2s_(s+1)+}, $s=1,...,\nu-2$, \eqref{x_2(nu-1)+},  \,i.e., $F_{2s}^{(s+1)}\big(t,\widetilde{\eta}_{2j}^{(j+1)}(t),j=s,...,\nu-1\big)=0$, $s=1,...,\nu-2$, $F_{2(\nu-1)}\big(t,\widetilde{\eta}_{2(\nu-1)}(t)\big)=0$,\, and, accordingly, the system \eqref{x_2s_(s+1)}, $s=1,...,\nu-2$, \eqref{x_2(nu-1)} for each $t\in \R_+$.

Substitute $x_{2j}^{(j+1)}=\widetilde{\eta}_{2j}^{(j+1)}(t)$, $j=1,...,\nu-1$, in equations \eqref{x_2wedge,nu-2+}, \eqref{x_2wedge,s+}, $s=1,...,\nu-3$.

Consider the equation
 \begin{equation}\label{x_2wedge,nu-2_Tilde}
\begin{split}
\widetilde{F}_{2\Sigma,\nu-2}(t,x_{2\wedge,\nu-2}) & =0,  \\
& \widetilde{F}_{2\Sigma,\nu-2}(t,x_{2\wedge,\nu-2}):= F_{2\Sigma,\nu-2} \!\left(t, \widetilde{\eta}_{2j}^{(j+1)}(t),j\!=\!1,...,\nu-1,x_{2\wedge,\nu-2}, \textstyle{\frac{d}{dt}}\widetilde{\eta}_{2(\nu-1)}(t)\right)\!,
\end{split}
 \end{equation}
equivalent to \eqref{x_2wedge,nu-2}, where $x_{2j}^{(j+1)}=\widetilde{\eta}_{2j}^{(j+1)}(t)$, $j=1,...,\nu-1$, and to \eqref{x_2wedge,nu-2+}, where $x_{2j}^{(j+1)}=\widetilde{\eta}_{2j}^{(j+1)}(t)$, $j=1,...,\nu-1$, and $d_{2(\nu-1)}=\frac{d}{dt}\widetilde{\eta}_{2(\nu-1)}(t)$.  $\widetilde{F}_{2\Sigma,\nu-2}\in C^{\nu-2}(\R_+\times M_{2\wedge,\nu-2}, Y_{2\Sigma,\nu-2})$ due to the mentioned properties of $F_{2\Sigma,\nu-2}$ and $\widetilde{\eta}_{2j}^{(j+1)}$. Here $\widetilde{F}_{2\Sigma,\nu-2}\in C^{\nu-2}(\R_+\times M_{2\wedge,\nu-2}, Y_{2\Sigma,\nu-2})$ due to the mentioned properties of $F_{2\Sigma,\nu-2}$ (see \eqref{smooth_F_(2s)_(s+1)}) and $\widetilde{\eta}_{2j}^{(j+1)}$.
By virtue of condition~\ref{Inv-set-2c}, for each (fixed) ${\hat{t}\in \R_+}$, $\hat{x}_{2\wedge,\nu-2}\in M_{2\wedge,\nu-2}$ satisfying \eqref{x_2wedge,nu-2_Tilde},  the operator
$$
\partial_{x_{2\wedge,\nu-2}} \widetilde{F}_{2\Sigma,\nu-2}(\hat{t},\hat{x}_{2\wedge,\nu-2})= \partial_{x_{2\wedge,\nu-2}} F_{2\Sigma,\nu-2} \Big(\hat{t}, \widetilde{\eta}_{2j}^{(j+1)}(\hat{t}),j=1,...,\nu-1, \hat{x}_{2\wedge,\nu-2}, \frac{d}{dt}\widetilde{\eta}_{2(\nu-1)}(\hat{t})\Big)
$$
has a bounded inverse.
By virtue of condition~\ref{Sogl_set-1c}, for each ${\hat{t}\in \R_+}$ there exists a unique ${\hat{x}_{2\wedge,\nu-2}\in M_{2\wedge,\nu-2}}$ such that $\hat{t}$, $\hat{x}_{2\wedge,\nu-2}$ satisfy \eqref{x_2wedge,nu-2_Tilde}.  Thus, by Theorem \ref{Th_Nonloc_ImplFunc-3} there exists the function
$$
\eta_{2\wedge,\nu-2}\in C^{\nu-2}(\R_+,M_{2\wedge,\nu-2})
$$
which is a unique solution of equation \eqref{x_2wedge,nu-2_Tilde} with respect to $x_{2\wedge,\nu-2}$.

Take an arbitrary fixed ${\hat{t}\in \R_+}$. As proved above, $\hat{t}$, $\hat{x}_{2j}^{(j+1)}=\widetilde{\eta}_{2j}^{(j+1)}\big(\hat{t}\big)$, $j=1,...,\nu-1$, and $\hat{x}_{2\wedge,\nu-2}=\eta_{2\wedge,\nu-2}\big(\hat{t}\big)$ satisfy the system \eqref{x_2(nu-1)}, \eqref{x_2s_(s+1)}, $s=1,...,\nu-2$, \eqref{x_2wedge,nu-2}.
Let us prove that there exist functions $\eta_{2\wedge,i}\in C^i(\R_+,M_{2\wedge,j})$, $i=1,...,\nu-3$, such that $\hat{t}$,  $\hat{x}_{2j}^{(j+1)}=\widetilde{\eta}_{2j}^{(j+1)}\big(\hat{t}\big)$, $j=1,...,\nu-1$, $\hat{x}_{2\wedge,i}=\eta_{2\wedge,i}\big(\hat{t}\big)$, $i=1,...,\nu-2$, satisfy the system \eqref{x_2(nu-1)}, \eqref{x_2s_(s+1)}, $s=1,...,\nu-2$, \eqref{x_2wedge,nu-2}, \eqref{x_2wedge,s}, $s=1,...,\nu-3$. These functions can be successively defined in a similar way as $\eta_{2\wedge,\nu-2}$. To do this, we will use the method of mathematical induction. Let us assume that there exist the functions $\eta_{2\wedge,i}$, $i=s+1,...,\nu-3$, with the above properties and prove that there exists the function $\eta_{2\wedge,s}$ that also satisfies them. First, consider the equations
\begin{subequations}\label{x_2wedge,sSyst_Tilde}
   \renewcommand{\theequation}{\theparentequation.s}
 \begin{equation}
\begin{split}
& \widetilde{F}_{2\Sigma,s} (t,x_{2\wedge,s})=0,   \qquad  \widetilde{F}_{2\Sigma,s} (t,x_{2\wedge,s}):=    \\
& := F_{2\Sigma,s} \!\left(\! t, \widetilde{\eta}_{2j}^{(j+1)}(t),j=1,...,\nu-1, x_{2\wedge,s},  \eta_{2\wedge,i}(t),i=s+1,...,\nu-2, \textstyle{\frac{d}{dt}}\widetilde{\eta}_{2(s+1)}^{(s+2)}(t), \textstyle{\frac{d}{dt}}\eta_{2\wedge,s+1}(t) \!\right) \\
&  \hspace{0.7\textwidth} (s=1,...,\nu-3).  \label{x_2wedge,s_Tilde}
\end{split}
 \end{equation}
   \end{subequations}
equivalent to \eqref{x_2wedge,s}, where $x_{2j}^{(j+1)}=\widetilde{\eta}_{2j}^{(j+1)}(t)$, $j=1,...,\nu-1$, and  $x_{2\wedge,i}=\eta_{2\wedge,i}(t)$, $i=s,...,\nu-2$, and to \eqref{x_2wedge,s+}, where $x_{2j}^{(j+1)}=\widetilde{\eta}_{2j}^{(j+1)}(t)$, $j=1,...,\nu-1$, $x_{2\wedge,i}=\eta_{2\wedge,i}(t)$, $i=s,...,\nu-2$, $d_{2\wedge,s+1}=\frac{d}{dt}\eta_{2\wedge,s+1}(t)$ and $d_{2(s+1)}^{(s+2)}=\frac{d}{dt}\widetilde{\eta}_{2(s+1)}^{(s+2)}(t)$.
As proved above, $\eta_{2\wedge,\nu-2}$ is a unique solution of equation \eqref{x_2wedge,nu-2_Tilde} with respect to $x_{2\wedge,\nu-2}$, and by assumption, $\eta_{2\wedge,j}$ ($j=\nu-3,\nu-4...,s+1$) is a unique solution of equation (\ref{x_2wedge,sSyst_Tilde}.j) with respect to $x_{2\wedge,j}$. The smoothness of $F_{2\Sigma,s}$, $\widetilde{\eta}_{2j}^{(j+1)}$ and $\eta_{2\wedge,i}$ yields $\widetilde{F}_{2\Sigma,s}\in C^s(\R_+\times M_{2\wedge,s}, Y_{2\Sigma,s})$.
Due to condition~\ref{Inv-set-2d}, for each (fixed) ${\hat{t}\in \R_+}$, $\hat{x}_{2\wedge,s}\in M_{2\wedge,s}$ satisfying \eqref{x_2wedge,s_Tilde},  the operator
 \begin{align*}
& \partial_{x_{2\wedge,s}} \widetilde{F}_{2\Sigma,s}(\hat{t},\hat{x}_{2\wedge,s})=   \\
& =\partial_{x_{2\wedge,s}} F_{2\Sigma,s} \Big(\hat{t}, \widetilde{\eta}_{2j}^{(j+1)}(\hat{t}),{j=1,...,\nu-1},  \hat{x}_{2\wedge,s},  \eta_{2\wedge,i}(\hat{t}),{i=s+1,...,\nu-2}, \textstyle{\frac{d}{dt}}\widetilde{\eta}_{2(s+1)}^{(s+2)}(\hat{t}), \textstyle{\frac{d}{dt}}\eta_{2\wedge,s+1}(\hat{t}) \Big)
 \end{align*}
has an inverse belonging to $\spL(Y_{2\Sigma,s},D_{2\wedge,s})$.
By virtue of condition~\ref{Sogl_set-1d}, for each ${\hat{t}\in \R_+}$ there exists a unique ${\hat{x}_{2\wedge,s}\in M_{2\wedge,s}}$ such that $\hat{t}$, $\hat{x}_{2\wedge,s}$ satisfy \eqref{x_2wedge,s_Tilde}.
Thus, using Theorem \ref{Th_Nonloc_ImplFunc-3}, we obtain that there exists the function
$$
\eta_{2\wedge,s}\in C^s(\R_+,M_{2\wedge,s})
$$
which is a unique solution of equation \eqref{x_2wedge,s_Tilde} with respect to $x_{2\wedge,s}$ for every $t\in\R_+$.
Thus, by induction, there exist the uniquely determined functions $\eta_{2\wedge,s}\in C^s(\R_+,M_{2\wedge,s})$, $s=1,...,\nu-3$, such that ${x_{2j}^{(j+1)}=\widetilde{\eta}_{2j}^{(j+1)}(t)}$, ${j=1,...,\nu-1}$, ${x_{2\wedge,i}=\eta_{2\wedge,i}(t)}$, ${i=1,...,\nu-2}$, satisfy the system \eqref{x_2(nu-1)}, \eqref{x_2s_(s+1)}, ${s=1,...,\nu-2}$, \eqref{x_2wedge,nu-2}, \eqref{x_2wedge,s}, ${s=1,...,\nu-3}$, for each $t\in \R_+$.

Introduce the function
 \begin{equation}\label{eta_2Sigma}
\eta_{2\Sigma}(t)=\sum\limits_{j=1}^{\nu-1}\widetilde{\eta}_{2j}^{(j+1)}(t)+ \sum\limits_{i=1}^{\nu-2}\eta_{2\wedge,i}(t),\quad \eta_{2\Sigma}\colon \R_+\to  M_{2\Sigma}.
 \end{equation}
Consider the equation
 \begin{equation}\label{x_20_Tilde}
\widetilde{F}_{20}(t,x_1,x_{20})=0,\quad
\widetilde{F}_{20}(t,x_1,x_{20}):= F_{20}\big(t,x_1,x_{20},\eta_{2\Sigma}(t), \textstyle{\frac{d}{dt}}\widetilde{\eta}_{21}^{(2)}(t), \textstyle{\frac{d}{dt}}\eta_{2\wedge,1}(t)\big),
 \end{equation}
where $\widetilde{F}_{20}\in C(\R_+\times M_1\times M_{20},Y_{20})$, which is equivalent to equation \eqref{x_20} where ${x_{2j}^{(j+1)}=\widetilde{\eta}_{2j}^{(j+1)}(t)}$, ${j=1,...,\nu-1}$, ${x_{2\wedge,i}=\eta_{2\wedge,i}(t)}$, ${i=1,...,\nu-2}$, and equation \eqref{x_20+} where ${x_{2\Sigma}=\eta_{2\Sigma}(t)}$, ${d_{2\wedge,1}=\frac{d}{dt}\eta_{2\wedge,1}(t)}$, $d_{21}^{(2)}=\frac{d}{dt}\widetilde{\eta}_{21}^{(2)}(t)$.
Take arbitrary fixed ${\hat{t}\in \R_+}$ and ${\hat{x}_1\in M_1}$. Then, due to condition \ref{Sogl_set-1e}, there exists a unique ${\hat{x}_{20}\in M_{20}}$ such that \,$\hat{t}$, $\hat{x}_1$, $\hat{x}_{20}$ satisfy  \eqref{x_20_Tilde}.
From condition \ref{Inv-Lipsch_set-2e} (taking into account the properties of $\eta_{2\wedge,i}$ and $\widetilde{\eta}_{2j}^{(j+1)}$) it follows that $\widetilde{F}_{20}(t,x_1,x_{20})$ satisfy locally a Lipschitz condition with respect to $(x_1,x_{20})$ on $\R_+\times M_1\times M_{20}$ and that there exist open neighborhoods
$U_\alpha\big(\hat{t},\hat{x}_1\big)=U_{\alpha_1}\big(\hat{t}\big)\times U_{\alpha_2}\big(\hat{x}_1\big)$ and $U_\beta\big(\hat{x}_{20}\big)$ and an operator $\widetilde{\mathrm{Z}}_{\hat{t},\hat{x},\hat{x}_{20}} \in \spL(D_{20},Y_{20})$ having the inverse $\widetilde{\mathrm{Z}}_{\hat{t},\hat{x},\hat{x}_{20}}^{-1}\in \spL(Y_{20},D_{20})$ such that for each $(t,x_1)\in U_\alpha\big(\hat{t},\hat{x}_1\big)$ and each $x'_{20},\, x''_{20}\in U_\beta\big(\hat{x}_{20}\big)$ the mapping $\widetilde{F}_{20}$ satisfies the inequality
 $$
\big\|\widetilde{F}_{20}(t,x_1,x'_{20})-\widetilde{F}_{20}(t,x_1,x''_{20})- \widetilde{\mathrm{Z}}_{\hat{t},\hat{x},\hat{x}_{20}}[x'_{20}-x''_{20}]\big\|\le k(\alpha,\beta)\big\|x'_{20}-x''_{20}\big\|,
 $$
where $k(\alpha,\beta)$ is such that $\lim\limits_{\alpha,\,\beta\to 0} k(\alpha,\beta)< \big\|\widetilde{\mathrm{Z}}_{\hat{t},\hat{x},\hat{x}_{20}}^{-1}\big\|^{-1}$.
Thus, by Theorem \ref{Th-Nonloc_ImplFunc}, there exists the function
\begin{equation}\label{eta_20}
\eta_{20}\in C(\R_+\times M_1,M_{20})
\end{equation}
which satisfies locally a Lipschitz condition with respect to $x_1$ on $\R_+\times M_1$ and is a unique solution of \eqref{x_20_Tilde} with respect to $x_{20}$.

Now we substitute ${x_{2\Sigma}=\eta_{2\Sigma}(t)}$ and $x_{20}=\eta_{20}(t,x_1)$ in equation \eqref{x_1}, multiply it by  $\widetilde{A}$ and obtain the equation
$\widetilde{A}\frac{d}{dt}x_1= Q_1f(t,x_1+\eta_{20}(t,x_1)+\eta_{2\Sigma}(t))-Bx_1$. We make the change of variables
$$
w_1=\widetilde{A}x_1,
$$
where $x_1\in M_1$, then $w_1\in \widetilde{M}_1$, $x_1=\widetilde{A}^{-1}w_1$ and we obtain the equation
 \begin{equation}\label{ADAE_DE_eta_w1}
\dot{w_1}=\breve{\Pi}(t,w_1),\quad \breve{\Pi}(t,w_1)= Q_1f\big(t,\widetilde{A}^{-1}w_1+\widehat{\eta}_{20}(t,w_1)+\eta_{2\Sigma}(t)\big)- Q_1B\widetilde{A}^{-1}w_1,
 \end{equation}
where $\widehat{\eta}_{20}(t,w_1):=\eta_{20}(t,\widetilde{A}^{-1}w_1)$. Since $\widetilde{A}^{-1}\in \spL(Y,X)$, $Q_1\in \spL(Y)$, $B\widetilde{A}^{-1}\in \spL(Y)$, then $\breve{\Pi}\in C(\R_+\times \widetilde{M}_1,Y_1)$ and, due to the properties of $Q_1f$ and $\eta_{20}$,   it satisfies locally a Lipschitz condition with respect to $w_1$ on $\R_+\times \widetilde{M}_1$. Using \cite[p.~9, Theorem~1]{Schwartz2}, we infer that for each initial point $(t_0,x_{0,1})\in \R_+\times M_1$ the IVP for equation \eqref{ADAE_DE_eta_w1} with the initial condition
 \begin{equation}\label{RegDEeta_ini-2}
w_1(t_0)=w_{0,1},\quad w_{0,1}=\widetilde{A}x_{0,1}=Ax_{0,1}\in \widetilde{M}_1,
 \end{equation}
has a unique solution $w_1(t)$ on some interval $J\subset \R_+$ that contains $t_0$. Further, by \cite[p.~16-17, Corollary]{Schwartz2} we obtain that the solution $w_1(t)$ can be extended over a maximal interval of existence $[t_0,t_{max})$ \,($t_{max}\le \infty$) of $w_1(t)$ in $\widetilde{M}_1$, and the extended solution is unique.  Notice that the function $x(t)=\widetilde{A}^{-1}w_1(t)+\widehat{\eta}_{20}(t,w_1(t))+\eta_{2\Sigma}(t)$ is continuous and  $Ax(t)=w_1(t)+A\eta_{2\Sigma}(t)$ has the continuous derivative $\dfrac{dAx}{dt}(t)= \breve{\Pi}(t,w_1(t))+AP_{2\Sigma}\dfrac{d\eta_{2\Sigma}}{dt}(t)$ on $[t_0,t_{max})$ (for $t_0=0$, $Ax(t)$ has the derivative on the right at $t_0$).  Thus, the IVP \eqref{ADAE}, \eqref{f-Appr-2}, \eqref{ini-2}  has the unique solution $x(t)=\widetilde{A}^{-1}w_1(t)+\widehat{\eta}_{20}(t,w_1(t))+\eta_{2\Sigma}(t)$  in $M$ on the maximal interval of existence $[t_0,t_{max})$.
 \end{proof}

 \begin{proof}[Proof of Corollary \ref{Corollary_Exist_set-2}]
Corollary \ref{Corollary_Exist_set-2} is proved in the same way as Corollary~\ref{Corollary_Exist_set}.
 \end{proof}

  \begin{lemma}\label{Lem_Alt-2Approach} 
Let the conditions of Theorem~\ref{Th_Exist-Lip_set-2} (or Corollary \ref{Corollary_Exist_set-2}), where
${\widetilde{M}_1= Y_1}$ \textup{(}accordingly, ${M_1=D_1}$ and ${M=D_1\dotpl M_{20}\dotpl M_{2\Sigma}}$\textup{)},
and the following condition be satisfied:
 \begin{enumerate}[label={\upshape\arabic*.}, ref={\upshape\arabic*}, itemsep=4pt,parsep=0pt,topsep=4pt,leftmargin=0.6cm]
 \addtocounter{enumi}{2}
\item\label{Alt-2Appr}
For any fixed ${a_1,a_2>0}$,
  \begin{equation}\label{Alt-2Appr-2}
 \sup\limits_{t\,\in\,[0,a_1],\; x\,\in\,M,\;  \|x_1\|\,\le\, a_2} \|Q_1f(t,x)\|<\infty \qquad (x=x_1+x_{20}+x_{2\Sigma}),
  \end{equation}
 \end{enumerate}
Then for each $t_0\in \R_+$,\, $x_{0,1}\in D_1$ there exists a $t_{max}\le \infty$ such that the IVP \eqref{ADAE}, \eqref{f-Appr-2}, \eqref{ini-2} has a unique solution $x(t)$ in $M$ on the maximal interval of existence $[t_0,t_{max})$, and if $t_{max}<\infty$, then $\lim\limits_{t\to t_{max}-0}\|AP_1x(t)\|=\infty$. In addition, if $D=X$ or the operator $A$ is bounded, then $\lim\limits_{t\to t_{max}-0}\|x(t)\|=\infty$ when $t_{max}<\infty$.
 \end{lemma}

 \begin{theorem}[\textbf{Global solvability}]\label{Th_GlobReg-2Approach}
Let the conditions of Lemma \ref{Lem_Alt-2Approach} and the following condition hold:
\begin{enumerate}[label={\upshape\arabic*.}, ref={\upshape\arabic*}, itemsep=4pt,parsep=0pt,topsep=4pt,leftmargin=0.6cm]
\addtocounter{enumi}{3}
\item\label{Extens-2Appr}   
  There exists a number ${R>0}$, a functional
  \begin{equation}\label{Vmax-2Appr} 
   V(w_1)=\max\limits_{i=1,...,m}V_i(w_1),
  \end{equation}
  where $w_1=\widetilde{A}x_1\in Y_1$ and $V_i\in C(Y_1,\R_+)$, $i=1,...,m$, are continuously differentiable functionals on $U_R^c(0)=\{w_1\in Y_1\mid \|w_1\|\ge R\}$, and functionals $U\in C(\R_+)$, $\psi\colon \R_+\to \R_+$ such that $\psi(t)$ is integrable on each finite interval in $\R_+$,
   \begin{equation}\label{Extens_lim-2Appr} 
  \int\limits^{\infty}\dfrac{du}{U(u)} =\infty,\qquad  \lim\limits_{\|w_1\|\to\infty}V(w_1)=\infty,
   \end{equation}
  and for each ${t\in \R_+}$, ${x_{20}\in M_{20}}$, ${x_{2\Sigma}\in M_{2\Sigma}}$, $w_1\in U_R^c(0)$ the following inequality holds:
   \begin{equation}\label{ExtensGlobReg-2Appr} 
  \big\langle Q_1f\big(t,\widetilde{A}^{-1}w_1+x_{20}+x_{2\Sigma}\big)- Q_1B\widetilde{A}^{-1}w_1,\, \partial_{w_1}V_{j(w_1)}(w_1) \big\rangle \le U\big(V(w_1)\big)\psi(t),
   \end{equation}
  where $j(w_1)\in \{1,...,m\}$ is any index for which $V(w_1)=V_{j(w_1)}(w_1)$ for a given $w_1$.
\end{enumerate}
Then for each $t_0\in \R_+$, $x_{0,1}\in D_1$ there exists a unique global solution $x(t)$ of the IVP \eqref{ADAE},~\eqref{f-Appr-2},~\eqref{ini-2}.
 \end{theorem}

  \begin{theorem}[\textbf{Global solvability}]\label{Th_GlobReg2-2Approach}
Theorem~\ref{Th_GlobReg-2Approach} remains valid if condition \ref{Extens-2Appr} is replaced by
 \begin{enumerate}[label={\upshape\arabic*.}, ref={\upshape\arabic*}, itemsep=4pt,parsep=0pt,topsep=4pt,leftmargin=0.6cm]
\addtocounter{enumi}{3}
\item\label{Extens2-2Appr}
  There exists a number ${R>0}$, a functional $V\in C(Y_1,\R_+)$, where $w_1=\widetilde{A}x_1\in Y_1$, and functionals $U\in C(\R_+)$, $\psi\colon \R_+\to \R_+$ such that $\psi(t)$ is integrable on each finite interval in $\R_+$, \eqref{Extens_lim-2Appr} holds,
  $V(w_1)$ satisfies a Lipschitz condition on $U_R^c(0)=\{w_1\in Y_1\mid \|w_1\|\ge R\}$   and for each ${t\in \R_+}$, ${x_{20}\in M_{20}}$, ${x_{2\Sigma}\in M_{2\Sigma}}$, $w_1\in U_R^c(0)$ the following inequality holds:
   \begin{equation}\label{ExtensGlobReg2-2Appr}
  \big\| Q_1f\big(t,\widetilde{A}^{-1}w_1+x_{20}+x_{2\Sigma}\big)- Q_1B\widetilde{A}^{-1}w_1\big\|\le U\big(V(w_1)\big)\psi(t).
   \end{equation}
 \end{enumerate}
 \end{theorem}

 \begin{proof}[Proof of Lemma \ref{Lem_Alt-2Approach}]
The proof is carried out by analogy with the proof of Lemma \ref{Lem-Alternative}. We provide it since it will be needed in what follows.

It follows from the proof of Theorem~\ref{Th_Exist-Lip_set-2} for the case when ${\widetilde{M}_1=Y_1}$ (note that if the conditions of Corollary~\ref{Corollary_Exist_set-2} hold, then the conditions of Theorem~\ref{Th_Exist-Lip_set-2} also hold) that for arbitrary initial values $t_0\in \R_+$, $x_{0,1}\in D_1$ the IVP \eqref{ADAE}, \eqref{f-Appr-2}, \eqref{ini-2} has the unique solution $x(t)=\widetilde{A}^{-1}w_1(t)+\widehat{\eta}_{20}(t,w_1(t))+\eta_{2\Sigma}(t)$ in ${M=D_1\dotpl M_{20}\dotpl M_{2\Sigma}}$ on the maximal interval of existence $[t_0,t_{max})$.

We will prove that if ${t_{max}<\infty}$, then ${\lim\limits_{t\to t_{max}-0}\|w_1(t)\|=\infty}$, where $w_1(t)$ is the solution of the IVP \eqref{ADAE_DE_eta_w1}, \eqref{RegDEeta_ini-2} on $[t_0,t_{max})$. The estimate \eqref{Alt-2Appr-2} implies that  ${\sup\limits_{t\,\in\,[0,b_1],\, \|w_1\|\,\le\,b_2} \|Q_1f\big(t, \widetilde{A}^{-1}w_1+ \widehat{\eta}_{20}(t,w_1)+ \eta_{2\Sigma}(t)\big)\| <\infty}$ for any fixed ${b_1,\,b_2>0}$. Hence,\,
 \begin{equation}\label{Bound_brevePi}
{\sup\limits_{t\,\in\, [0,b_1],\, \|w_1\|\,\le\, b_2} \|\breve{\Pi}(t,w_1)\|<\infty}\, \quad \text{for any fixed ${b_1,\,b_2>0}$}.
 \end{equation}
It follows from Lemma~\ref{Lem_Exist-Blow_up} that ${\lim\limits_{t\to t_{max}-0}\|w_1(t)\|=\infty}$ if ${t_{max}<\infty}$.  Since ${w_1(t)=AP_1x(t)}$, then ${\lim\limits_{t\to t_{max}-0}\|AP_1x(t)\|=\infty}$ if ${t_{max}<\infty}$.
When $A$ is bounded, this fact gives ${\lim\limits_{t\to t_{max}-0}\|x(t)\|=\infty}$.
\end{proof}

 \begin{proof}[Proofs of Theorem \ref{Th_GlobReg-2Approach} and  \ref{Th_GlobReg2-2Approach}]
The proofs of Theorem \ref{Th_GlobReg-2Approach} and  \ref{Th_GlobReg2-2Approach} are carried out in the same way as the proofs of Theorem \ref{Th_GlobReg} and \ref{Th_GlobReg2}, respectively.
\end{proof}

Note that if we multiply equations \eqref{x_2(nu-1)+}, \eqref{x_2s_(s+1)+} (or \eqref{x_2(nu-1)}, \eqref{x_2s_(s+1)}), $s=1,...,\nu-2$, equations \eqref{x_2wedge,nu-2+}, \eqref{x_2wedge,s+}, $s=1,...,\nu-3$, and equation \eqref{x_20+} on the left by the semi-inverse operator $\EuScript B_2^{(-1)}$ of $\EuScript B_2=BP_2$, we obtain the equivalent equations, respectively,
 \begin{equation}\label{Stab_x_2s_(s+1)}
x_{2s}^{(s+1)}=\EuScript B_2^{(-1)}Q_{2s}^{(s+1)} f\bigg(t,\sum\limits_{j=s}^{\nu-1}x_{2j}^{(j+1)}\bigg),\quad s=1,...,\nu-1,
 \end{equation}
 \begin{equation}\label{Stab_x_2wedge,s}
\begin{split}
x_{2\wedge,\nu-2} & =\EuScript B_2^{(-1)} \left(Q_{2\Sigma,\nu-2}f\bigg(t, \sum\limits_{j=1}^{\nu-1}x_{2j}^{(j+1)}+x_{2\wedge,\nu-2}\bigg)- Ad_{2(\nu-1)}\right),   \\
x_{2\wedge,s} & =\EuScript B_2^{(-1)} \left(Q_{2\Sigma,s}f\bigg(t, \sum\limits_{j=1}^{\nu-1}x_{2j}^{(j+1)}+ \sum\limits_{i=s}^{\nu-2}x_{2\wedge,i}\bigg)- A\Big[d_{2(s+1)}^{(s+2)}+d_{2\wedge,s+1}\Big]\right),\quad s=1,...,\nu-3.
\end{split}
 \end{equation}
 \begin{equation}\label{Stab_x_20}
x_{20}=\EuScript B_2^{(-1)} \left(Q_{20}f\big(t,x_1+x_{20}+x_{2\Sigma}\big)- Q_{20}A\Big[d_{21}^{(2)}+d_{2\wedge,1}\Big]\right).
 \end{equation}
The operator ${\EuScript B_2^{(-1)}\in \spL(Y,X)}$ is defined by the relations ${\EuScript B_2^{(-1)}\EuScript B_2=P_2}$,
${\EuScript B_2\EuScript B_2^{(-1)}=Q_2}$ and
${\EuScript B_2^{(-1)}= P_2\EuScript B_2^{(-1)}}$.

 \begin{theorem}[\textbf{Lagrange stability}]\label{Th_Lagr_Stab-2Appr}
Let the conditions of Theorem~\ref{Th_Exist-Lip_set-2} (or Corollary \ref{Corollary_Exist_set-2}), where   ${\widetilde{M}_1= Y_1}$ \textup{(}accordingly, ${M_1=D_1}$ and ${M=D_1\dotpl M_{20}\dotpl M_{2\Sigma}}$\textup{)}, and the following conditions be fulfilled:
 \begin{enumerate}[label={\upshape\arabic*.},ref={\upshape\arabic*}, itemsep=4pt,parsep=0pt,topsep=4pt,leftmargin=0.6cm]
 \addtocounter{enumi}{2}
\item\label{Lagr-Alt-2Appr}  %
  There exist constants $K_1,\,K_2\ge 0$, $C_1\ge \|\EuScript B_2^{(-1)}\| K_1$, $C_2\ge \|\EuScript B_2^{(-1)}\|K_2$ such that
  \begin{equation}\label{Lagr-Alt-x_2s_s+1}
 \sup\limits_{t\,\in\,\R_+,\;\; x_{2j}^{(j+1)}\,\in\, M_{2j}^{(j+1)},\; j\,=\,s,...,\nu-1} \left\|Q_{2s}^{(s+1)}f\bigg(t, \sum\limits_{j=s}^{\nu-1}x_{2j}^{(j+1)}\bigg)\right\|\!\le K_1,\;\; s=1,...,\nu-1;
  \end{equation}
   \begin{equation}\label{Lagr-Alt-x_2wedge,s-nu-2}
 \begin{split}
\hspace{-5mm}\sup_{\substack{ t\,\in\,\R_+,\;\;
 x_{2j}^{(j+1)}\,\in\, M_{2j}^{(j+1)},\; \|x_{2j}^{(j+1)}\|\,\le\, C_1,\; j\,=\,1,...,\nu-1,  \\
 x_{2\wedge,\nu-2}\,\in\, M_{2\wedge,\nu-2},\;\; d_{2(\nu-1)}\,\in\, D_{2(\nu-1)} } }
 \left\|Q_{2\Sigma,\nu-2}f\bigg(t, \sum\limits_{j=1}^{\nu-1}x_{2j}^{(j+1)}+x_{2\wedge,\nu-2}\bigg)- Ad_{2(\nu-1)}\right\|\!\le\! K_2,   \\
\hspace{-5mm}\sup_{\substack{ t\,\in\,\R_+, \\
 x_{2j}^{(j+1)}\,\in\, M_{2j}^{(j+1)},\; \|x_{2j}^{(j+1)}\|\,\le\,C_1,\; j\,=\,1,...,\nu-1,  \\
 x_{2\wedge,i}\,\in\, M_{2\wedge,i},\; \|x_{2\wedge,i}|\,\le\,C_2,\; i\,=\,s+1,...,\nu-2, \\
 x_{2\wedge,s}\,\in\, M_{2\wedge,s},  \\ d_{2\wedge,s+1}\,\in\, D_{2\wedge,s+1},\;\; d_{2(s+1)}^{(s+2)}\,\in\, D_{2(s+1)}^{(s+2)} } }
 \!\left\|Q_{2\Sigma,s}f\bigg(\!t, \sum\limits_{j=1}^{\nu-1}x_{2j}^{(j+1)}\!+\! \sum\limits_{i=s}^{\nu-2}x_{2\wedge,i}\bigg)\!-\! A\Big[d_{2(s+1)}^{(s+2)}\!+ d_{2\wedge,s+1}\Big]\right\|\!\le\! K_2,\! \\
 s=1,...,\nu-3.
 \end{split}
  \end{equation}
  For every fixed ${b>0}$, ${d>0}$ and some constants ${C_3\ge (\nu-1)C_1+(\nu-2)C_2}$ and ${K=K(b)>0}$, ${C_4\ge \|\EuScript B_2^{(-1)}\|K}$ \,\textup{(}$K$, $C_4$ depend on the choice of $b$\textup{)}, the following holds:
  \begin{equation}\label{Lagr-Alt-x_20} 
 \sup\limits_{t\,\in\,\R_+,\; x\,\in\,M,\; \|x_1\|\,\le\, b,\;  \|x_{2\Sigma}\|\,\le\,C_3} \!\left\|Q_{20}f\big(t,x_1+x_{20}+x_{2\Sigma}\big)- Q_{20}A\Big[d_{21}^{(2)}+d_{2\wedge,1}\Big]\right\|\le K,
  \end{equation}
  \begin{equation}\label{Lagr-Alt2-w1}
 \sup\limits_{t\,\in\,[0,d],\; x\,\in\,M,\; \|x_1\|\,\le\,b,\; \|x_{2\Sigma}\|\,\le\,C_3,\; \|x_{20}\|\,\le\,C_4} \|Q_1f(t,x)\|<\infty \qquad (x=x_1+x_{20}+x_{2\Sigma}).
  \end{equation}

 \item Let condition \ref{Extens-2Appr} of Theorem~\ref{Th_GlobReg-2Approach} or condition \ref{Extens2-2Appr} of Theorem \ref{Th_GlobReg2-2Approach} hold and let the functional $\psi$ defined in these conditions be integrable on $\R_+$, i.e., \eqref{Stab-psi} is fulfilled.
 \end{enumerate}
Then the ADAE \eqref{ADAE},~\eqref{f-Appr-2} is Lagrange stable \textup{(}for the initial points $(t_0,x_{0,1})\in \R_+\times D_1$\textup{)}.
 \end{theorem}

\begin{remark}\label{Rem_Lagr_Stab-2Appr}
The condition \ref{Lagr-Alt-2Appr} in Theorem \ref{Th_Lagr_Stab-2Appr} can be replaced by the following:
 \begin{enumerate}[label={\upshape\arabic*.},ref={\upshape\arabic*}, itemsep=4pt,parsep=0pt,topsep=4pt,leftmargin=0.6cm]
 \addtocounter{enumi}{2}
\item\label{Lagr-Alt-2Appr_2}
  There exist constants $C_1,\, C_2,\, C_3\ge 0$ such that
\begin{itemize}
  \item $\big\|x_{2s}^{(s+1)}\big\|\le C_1$ for each $t\in\R_+$, $x_{2j}^{(j+1)}\in M_{2j}^{(j+1)}$, $j=s,...,\nu-1$, $s=1,...,\nu-1$;
  \item $\big\|x_{2\wedge,\nu-2}\big\|\le C_2$ for each $t\in\R_+$, $x_{2j}^{(j+1)}\in M_{2j}^{(j+1)}$, $\|x_{2j}^{(j+1)}\|\le C_1$, $j=1,...,\nu-1$, $x_{2\wedge,\nu-2}\in M_{2\wedge,\nu-2}$, $d_{2(\nu-1)}\in D_{2(\nu-1)}$;
  \item $\big\|x_{2\wedge,s}\big\|\le C_2$ for each $t\in\R_+$, $x_{2j}^{(j+1)}\in M_{2j}^{(j+1)}$, $\|x_{2j}^{(j+1)}\|\le C_1$, $j=1,...,\nu-1$, $x_{2\wedge,i}\in M_{2\wedge,i}$, $\|x_{2\wedge,i}|\le C_2$, $i=s+1,...,\nu-2$, $x_{2\wedge,s}\in M_{2\wedge,s}$, $d_{2\wedge,s+1}\in D_{2\wedge,s+1}$, $d_{2(s+1)}^{(s+2)}\in D_{2(s+1)}^{(s+2)}$, $s=1,...,\nu-3$.
\end{itemize}
 For every fixed ${b>0}$ there exists constants ${C_3\ge (\nu-1)C_1+(\nu-2)C_2}$ and ${C_4=C_4(b)>0}$ such that $\big\|x_{20}\big\|\le C_4$ for each $t\in\R_+$, $x\in M$, $\|x_1\|\le b$, $\|x_{2\Sigma}\|\le C_3$. \\
 For every fixed ${b>0}$, ${d>0}$ the estimate \eqref{Lagr-Alt2-w1} holds.
 \end{enumerate}
\end{remark}

 \begin{lemma}\label{Lem_Alt_Lagr-2Approach}
Lemma~\ref{Lem_Alt-2Approach} remains valid if condition \ref{Alt-2Appr} of this lemma is replaced by condition \ref{Lagr-Alt-2Appr} of Theorem~\ref{Th_Lagr_Stab-2Appr}.
 \end{lemma}

 \begin{proof}[Proof of Lemma \ref{Lem_Alt_Lagr-2Approach}]
First recall that (as stated in the proof of Theorem~\ref{Th_Exist-Lip_set-2}) for arbitrary $t_0\in \R_+$, $x_{0,1}\in D_1$ the IVP \eqref{ADAE}, \eqref{f-Appr-2}, \eqref{ini-2} has the unique solution $x(t)=\widetilde{A}^{-1}w_1(t)+\widehat{\eta}_{20}(t,w_1(t))+\eta_{2\Sigma}(t)$
in ${M=D_1\dotpl M_{20}\dotpl M_{2\Sigma}}$ on the maximal interval of existence $[t_0,t_{max})$.
Here $w_1(t)$ is the solution of the IVP \eqref{ADAE_DE_eta_w1}, \eqref{RegDEeta_ini-2} on $[t_0,\infty)$, $\widehat{\eta}_{20}\in C(\R_+\times Y_1,M_{20})$ (see \eqref{ADAE_DE_eta_w1} and \eqref{eta_20}) and $\eta_{2\Sigma}\in C^1(\R_+,M_{2\Sigma})$ (see \eqref{eta_2Sigma}).

The representations \eqref{Stab_x_2s_(s+1)} and estimate  \eqref{Lagr-Alt-x_2s_s+1} imply that $\sup\limits_{t\,\in\,\R_+}\big\|\widetilde{\eta}_{2s}^{(s+1)}(t)\big\|\le \|\EuScript B_2^{(-1)}\|K_1\le C_1=\const$, $s=1,...,\nu-1$. Therefore, the representations \eqref{Stab_x_2wedge,s} and estimate  \eqref{Lagr-Alt-x_2wedge,s-nu-2} yield  $\sup\limits_{t\,\in\,\R_+}\|\eta_{2\wedge,s}(t)\|\le \|\EuScript B_2^{(-1)}\|K_2 \le C_2=\const$, $s=1,...,\nu-2$. Then $\sup\limits_{t\,\in\,\R_+}\|\eta_{2\Sigma}(t)\|\le (\nu-1)C_1+(\nu-2)C_2\le C_3=\const$.
From this estimate, \eqref{Lagr-Alt-x_20} and the representation \eqref{Stab_x_20} we obtain that
$\sup\limits_{t\,\in\,\R_+,\; \|w_1\|\,\le\,b_2} \|\widehat{\eta}_{20}(t,w_1)\|\le \|\EuScript B_2^{(-1)}\|K\le C_4$ for any fixed $b_2>0$ ($\|x_1\|\le \|\widetilde{A}^{-1}\|b_2$).

It follows from the obtained estimates and \eqref{Lagr-Alt2-w1} that $\sup\limits_{t\,\in\,[0,b_1],\, \|w_1\|\,\le\,b_2} \|Q_1f\big(t, \widetilde{A}^{-1}w_1+ \widehat{\eta}_{20}(t,w_1)+ \eta_{2\Sigma}(t)\big)\| <\infty$  for any fixed ${b_1=d>0}$, ${b_2>0}$, and hence \eqref{Bound_brevePi} holds. The rest of the proof coincides with the part of the proof of Lemma~\ref{Lem_Alt-2Approach} after \eqref{Bound_brevePi}.
 \end{proof}

 \begin{proof}[Proof of Theorem \ref{Th_Lagr_Stab-2Appr}]
Due to Lemma \ref{Lem_Alt_Lagr-2Approach}, all conditions of Theorem~\ref{Th_GlobReg-2Approach} (or Theorem \ref{Th_GlobReg2-2Approach}) are satisfied.
Hence, for arbitrary $t_0\in \R_+$, $x_{0,1}\in D_1$ the IVP \eqref{ADAE}, \eqref{f-Appr-2}, \eqref{ini-2} has the unique solution $x(t)=\widetilde{A}^{-1}w_1(t)+\widehat{\eta}_{20}(t,w_1(t))+\eta_{2\Sigma}(t)$ on $[t_0,\infty)$.

Assume that $\sup\limits_{t\in [t_0,\infty)}\|w_1(t)\|=\infty$.
Then there exists a sequence $\{t_n\}$ such that $t_n\ge t_0$, $t_n\to \infty$ and $\|w_1(t_n)\|\to \infty$ as $n\to \infty$. Then, in the same way as in the proof of Theorem~\ref{Th_Lagr_Stab}, we obtain that $\infty=\int\limits_{0}^{\infty}\psi(\tau)\, d\tau$, which contradicts \eqref{Stab-psi}. Thus, $\sup\limits_{t\in [t_0,\infty)}\|w_1(t)\|<\infty$ and, hence, $\sup\limits_{t\in [t_0,\infty)}\|\widetilde{A}^{-1}w_1(t)\|<\infty$.
Therefore, taking into account the estimates obtained for $\widehat{\eta}_{20}(t,w_1)$ and $\eta_{2\Sigma}(t)$ in the proof of Lemma \ref{Lem_Alt_Lagr-2Approach}, we obtain that $\sup\limits_{t\in [t_0,\infty)}\|x(t)\|=\sup\limits_{t\in [t_0,\infty)}\|\widetilde{A}^{-1}w_1(t)+ \widehat{\eta}_{20}(t,w_1(t))+\eta_{2\Sigma}(t)\| <\infty$.

Thus, the ADAE \eqref{ADAE},~\eqref{f-Appr-2} is Lagrange stable for each initial point $(t_0,x_{0,1})\in \R_+\times D_1$.
 \end{proof}

 \begin{theorem}[\textbf{Lagrange instability (blow-up of solutions)}]\label{Th_Lagr_Instab-2Appr} 
Let the conditions of Lemma \ref{Lem_Alt-2Approach} and the following conditions hold:
 \begin{enumerate}[label={\upshape\arabic*.}, ref={\upshape\arabic*}, itemsep=4pt,parsep=0pt,topsep=4pt,leftmargin=0.6cm]
\addtocounter{enumi}{3}
\item\label{InstLagr-a-2Appr}
  There exists an open set $\widetilde{\Omega}_1\subseteq Y_1$ such that the component $P_1x(t)=x_1(t)$ of each solution $x(t)$ with the initial values $t_0\in \R_+$, $x_{0,1}\in \Omega_1= \widetilde{A}^{-1}\widetilde{\Omega}_1$, remains in $\Omega_1$ for all $t$ from the maximal interval of existence of $x(t)$.

\item\label{InstLagr-b-2Appr}
  There exists a functional
  \begin{equation}\label{Vmin-2Appr}
   V(w_1)=\min\limits_{i=1,...,m}V_i(w_1),
  \end{equation}
  where $w_1=\widetilde{A}x_1\in Y_1$  and $V_i\in C(Y_1,\R_+)$, $i=1,...,m$, are positive and continuously differentiable functionals on $\widetilde{\Omega}_1$, and functionals $U\in C(0,\infty)$, $\psi\colon \R_+\to \R_+$ such that $\psi(t)$ is integrable on each finite interval in $\R_+$,
   \begin{equation}\tag{\ref{Blow-up}}
  \int\limits^{\infty}\dfrac{du}{U(u)}<\infty,\qquad   \int\limits^{\infty}\psi(t)dt=\infty,
    \end{equation}
   and for each ${t\in \R_+}$, ${x_{20}\in M_{20}}$, ${x_{2\Sigma}\in M_{2\Sigma}}$, $w_1\in \widetilde{\Omega}_1$, the following inequality holds:
   \begin{equation}\label{InstLagrReg-2Appr}
  \big\langle Q_1f\big(t,\widetilde{A}^{-1}w_1+x_{20}+x_{2\Sigma}\big)- Q_1B\widetilde{A}^{-1}w_1,\, \partial_{w_1}V_{j(w_1)}(w_1) \big\rangle \ge U\big(V(w_1)\big)\psi(t),
   \end{equation}
  where $j(w_1)\in \{1,...,m\}$ is any index for which $V(w_1)=V_{j(w_1)}(w_1)$ for a given $w_1$.
 \end{enumerate}
Then for each $t_0\in \R_+$,\, $x_{0,1}\in \Omega_1$, there exists a $t_{max}<\infty$ such that the IVP \eqref{ADAE}, \eqref{f-Appr-2}, \eqref{ini-2} has a unique solution $x(t)$ on the maximal interval of existence $[t_0,t_{max})$ and $\lim\limits_{t\to t_{max}-0}\|AP_1x(t)\|=\infty$. In addition, if $D=X$ or the operator $A$ is bounded, then $\lim\limits_{t\to t_{max}-0}\|x(t)\|=\infty$ $($the solution blows up in the finite time $[t_0,t_{max})$$)$.
 \end{theorem}

 \begin{proof}
The proof is carried out by analogy with the proof of Theorem \ref{Th_Lagr_Instab}.

It follows from the proof of Lemma \ref{Lem_Alt-2Approach} that for arbitrary $t_0\in \R_+$ and $x_{0,1}\in D_1$ the IVP \eqref{ADAE}, \eqref{f-Appr-2}, \eqref{ini-2} has the unique solution $x(t)=\widetilde{A}^{-1}w_1(t)+\widehat{\eta}_{20}(t,w_1(t))+\eta_{2\Sigma}(t)$ on $[t_0,t_{max})$, where $w_1(t)$ is the solution of the IVP \eqref{ADAE_DE_eta_w1}, \eqref{RegDEeta_ini-2} on $[t_0,t_{max})$. Also, by Lemma \ref{Lem_Alt-2Approach}, the theorem statement is fulfilled if $t_{max}<\infty$.

Assume that $t_{max}=\infty$. Consider a sequence $\{t_n\}$ such that $t_n> t_0$ and $t_n\to \infty$ as $n\to \infty$. Due to condition \ref{InstLagr-a}, $w_1(t_n)\in \widetilde{\Omega}_1$ for all $n$.
It follows from condition~\ref{InstLagr-b-2Appr} that the inequality $\big\langle \breve{\Pi}(t,w_1), \partial_{w_1}V_{j(w_1)}(w_1)\big\rangle\ge U\big(V(w_1)\big)\psi(t)$, where $V$, $U$ and $\psi$ are some functionals satisfying condition~\ref{InstLagr-b-2Appr}, holds for each $t\in \R_+$, $w_1\in \widetilde{\Omega}_1$. Then, due to Lemma \ref{Lem_2.2-ineq_ge},
$V(w_1(t_n))-V(w_1(t_1))\ge \int\limits_{t_1}^{t_n}U\big(V(w_1(\tau))\big)\psi(\tau)\, d\tau$
for all $n\ge 1$.  From this inequality we obtain that
$\int\limits_{V(w_1(t_1))}^{V(w_1(t_n))} \dfrac{du}{U(u)}\ge \int\limits_{t_1}^{t_n}\psi(\tau)\, d\tau$
for all $n\ge 1$, which is proved in the same way as \cite[Lemma~2.1]{KK1956}.
Further, passing to the limit as $n\to \infty$, we obtain
$$
\infty>\int\limits_{V(w_1(t_1))}^\infty \dfrac{du}{U(u)}\ge \lim\limits_{n\to\infty}\int\limits_{V(w_1(t_1))}^{V(w_1(t_n))} \dfrac{du}{U(u)}\ge \int\limits_{t_1}^{\infty}\psi(\tau)\, d\tau = \infty
$$
by virtue of \eqref{Blow-up}, which is a contradiction. Hence $t_{max}<\infty$ and the statement of the theorem holds.
\end{proof}

 \begin{theorem}[\textbf{Lagrange instability}]\label{Th_Lagr_Instab_spec-2Appr} 
Let the conditions of Theorem~\ref{Th_Exist-Lip_set-2} (or Corollary \ref{Corollary_Exist_set-2}) hold, and let the following conditions, where $\widetilde{M}_1$, $M_{20}$, $M_{2\Sigma}$, ${M_1=\widetilde{A}^{-1}\widetilde{M}_1}$ and ${M=M_1\dotpl M_{20}\dotpl M_{2\Sigma}}$ are defined in Theorem~\ref{Th_Exist-Lip_set-2}, be satisfied:
 \begin{enumerate}[label={\upshape\arabic*.}, ref={\upshape\arabic*}, itemsep=4pt,parsep=0pt,topsep=4pt,leftmargin=0.6cm]
\addtocounter{enumi}{2}
 \item\label{Alternative_set-2Appr}
 For any fixed ${a>0}$ and any closed bounded set $\widetilde{S}_1\subset \widetilde{M}_1$ such that $\rho(\widetilde{S}_1,\partial \widetilde{M}_1)>0$ the following holds:
  \begin{equation}\label{Alt-2_set-2Appr}
 \sup\limits_{t\,\in\,[0,a],\; x\,\in\, M,\; w_1\,\in\, \widetilde{S}_1} \|Q_1f(t,x)\|<\infty \qquad (x=\widetilde{A}^{-1}w_1+x_{20}+x_{2\Sigma}).
  \end{equation}

\item\label{InstLagr-a_set-2Appr}
  The component $P_1x(t)=x_1(t)$ of each solution $x(t)$ with the initial values $t_0\in \R_+$, $x_{0,1}\in M_1$ remains in $M_1$ for all $t$ from the maximal interval of existence of $x(t)$.

\item\label{InstLagr-b_set-2Appr}
  There exists a functional of the form \eqref{Vmin-2Appr}, where ${w_1=\widetilde{A}x_1\in Y_1}$  and ${V_i\in C(Y_1,\R_+)}$, $i=1,...,m$, are positive and continuously differentiable functionals on $\widetilde{M}_1$, and functionals ${U\in C(0,\infty)}$, ${\psi\colon \R_+\to \R_+}$ such that $\psi(t)$ is integrable on each finite interval in $\R_+$, \eqref{Blow-up} is fulfilled, and for each ${t\in \R_+}$, ${x_{20}\in M_{20}}$, ${x_{2\Sigma}\in M_{2\Sigma}}$, $w_1\in \widetilde{M}_1$ the  inequality \eqref{InstLagrReg-2Appr}, where $j(w_1)\in \{1,...,m\}$ is any index for which $V(w_1)=V_{j(w_1)}(w_1)$ for a given $w_1$, holds.
 \end{enumerate}
Then for each $t_0\in \R_+$,\, $x_{0,1}\in M_1$ there exists a $t_{max}<\infty$ such that the IVP \eqref{ADAE}, \eqref{f-Appr-2}, \eqref{ini-2} has a unique solution $x(t)$ on the maximal interval of existence $[t_0,t_{max})$ and $\lim\limits_{t\to t_{max}-0}\|AP_1x(t)\|=\infty$. In addition, if $D=X$  or the operator $A$ is bounded, then $\lim\limits_{t\to t_{max}-0}\|x(t)\|=\infty$ $($the solution blows up in the finite time $[t_0,t_{max})$$)$.
 \end{theorem}

 \begin{proof}
It follows from the proof of Theorem~\ref{Th_Exist-Lip_set-2} that for an arbitrary initial point $(t_0,x_{0,1})\in \R_+\times M_1$ the IVP \eqref{ADAE},~\eqref{f-Appr-2},~\eqref{ini-2} has the unique solution $x(t)=\widetilde{A}^{-1}w_1(t)+\widehat{\eta}_{20}(t,w_1(t))+\eta_{2\Sigma}(t)$ in $M$ on $[t_0,t_{max})$, where $w_1(t)$ is the solution of the IVP \eqref{ADAE_DE_eta_w1}, \eqref{RegDEeta_ini-2} on $[t_0,t_{max})$.

Using \eqref{Alt-2_set-2Appr}, we obtain that
$\sup\limits_{t\,\in\,[0,b],\; w\,\in\,\widetilde{S}_1} \|Q_1f\big(t, \widetilde{A}^{-1}w_1+ \widehat{\eta}_{20}(t,w_1)+ \eta_{2\Sigma}(t)\big)\| <\infty$  and, hence,
$\sup\limits_{t\,\in\, [0,b],\; w\,\in\,\widetilde{S}_1} \|\breve{\Pi}(t,w_1)\|<\infty$ for any fixed ${b>0}$ and any closed bounded set $\widetilde{S}_1\subset \widetilde{M}_1$ such that $\rho(\widetilde{S}_1,\partial \widetilde{M}_1)>0$.
Thus, it follows from Lemma \ref{Lem_Blow_up} that if $t_{max}<\infty$, then $\lim\limits_{t\to t_{max}-0}\|w_1(t)\|=\infty$.

Further, in the same way as in the proof of Theorem \ref{Th_Lagr_Instab_specify}, we prove that $t_{max}<\infty$.
\end{proof}

\begin{remark}\label{Rem_f-2-Appr-2}
Instead of the function $f$ of the form \eqref{f-Appr-2} we can consider the one having the structure
 \begin{equation}\label{f-2-Appr-2}
f(t,x)=\big(Q_1+Q_{20}\big)f(t,x)+  \sum\limits_{s=1}^{\nu-2}Q_{2\Sigma,s} f\bigg(t,x_{2\Sigma}^{(1)}+\sum\limits_{j=s}^{\nu-2}x_{2\wedge,j}\bigg)+ Q_{2*}^{(2)}f\bigg(t,x_{2\Sigma}^{(1)}\bigg)
 \end{equation}
where $Q_{2*}^{(2)}=\sum\limits_{s=1}^{\nu-1}Q_{2s}^{(s+1)}$ and $x_{2\Sigma}^{(1)}=\sum\limits_{j=1}^{\nu-1}x_{2j}^{(j+1)}=P_{2\Sigma}^{(1)}x$ (see \eqref{Q_2*_(i),P_2Sigma_(i)}),
that is, the components $Q_{2s}^{(s+1)}f$  \,($s=2,...,\nu-1$) of the function $f$ depend on the sum $\sum\limits_{j=1}^{\nu-1}x_{2j}^{(j+1)}$ instead of the sum from $s$ to $\nu-1$ in \eqref{f-Appr-2}.
Therefore, the system \eqref{1_1}, \eqref{2_20},  \eqref{2_2Sigma_s-}, \eqref{2_2Sigma_nu-2}, \eqref{3_2*_s-}, \eqref{3_2*_nu-1} for $f$ of the form  \eqref{f-2-Appr-2} takes the form \eqref{x_1}--\eqref{x_2(nu-1)} where $Q_{2s}^{(s+1)}f$, $s=2,...,\nu-1$, depend on $t,x_{2\Sigma}^{(1)}$. In the rest of the equations of this system $\sum\limits_{j=1}^{\nu-1}x_{2j}^{(j+1)}$ can also be redesignated as $x_{2\Sigma}^{(1)}$. Further, we sum those equations of the resulting system that correspond to equations \eqref{x_2s_(s+1)}, $s=1,...,\nu-2$, \eqref{x_2(nu-1)}, and obtain the equation
 \begin{equation}\label{x_2Sigma_(1)}
Bx_{2\Sigma}^{(1)}=Q_{2*}^{(2)}f\big(t,x_{2\Sigma}^{(1)}\big).
 \end{equation}

Thus, the system \eqref{x_1}--\eqref{x_2wedge,nu-2}, \eqref{x_2Sigma_(1)} is equivalent to the ADAE  \eqref{ADAE} with $f$ of the form \eqref{f-2-Appr-2}.

Equation \eqref{x_2Sigma_(1)} can be represented as the equation
 \begin{equation}\label{x_2Sigma_(1)+}
F_{2\Sigma}^{(1)}\big(t,x_{2\Sigma}^{(1)})=0,\qquad
F_{2\Sigma}^{(1)}\big(t,x_{2\Sigma}^{(1)})):= Q_{2*}^{(2)}f\big(t,x_{2\Sigma}^{(1)}\big)-Bx_{2\Sigma}^{(1)},
 \end{equation}
which is equivalent to the sum of equation \eqref{x_2s_(s+1)+}, $s=1,...,\nu-2$, \eqref{x_2(nu-1)+} where $Q_{2s}^{(s+1)}f$, $s=2,...,\nu-1$, depend on $t,x_{2\Sigma}^{(1)}$. Notice that $P_{2j}^{(j+1)}x_{2\Sigma}^{(1)} =x_{2j}^{(j+1)}$.

It can be readily verified that for the IVP \eqref{ADAE}, \eqref{f-2-Appr-2}, \eqref{ini-2}, Theorem \ref{Th_Exist-Lip_set-2} with the following changes is true:
\begin{itemize}
  \item $Q_{2*}^{(2)}f\big(t,x_{2\Sigma}^{(1)}\big)\in C^{\nu-1}\bigg(\R_+\times M_{2\Sigma}^{(1)},Y_{2s}^{(s+1)}\bigg)$ where ${M_{2\Sigma}^{(1)}=\dotpl\limits_{i=1}^{\nu-1}M_{2i}^{(i+1)}}$, that is,  \\
      $Q_{2s}^{(s+1)}f\in C^{\nu-1}\bigg(\R_+\times  \bigg(\dotpl\limits_{i=1}^{\nu-1}M_{2i}^{(i+1)}\bigg),Y_{2s}^{(s+1)}\bigg)$, $s=2,...,\nu-1$.

  \item Instead of conditions \ref{Sogl_set-1a}, \ref{Sogl_set-1b}, it is required that for each  ${t\in \R_+}$ there exists a unique ${x_{2\Sigma}^{(1)}\in M_{2\Sigma}^{(1)}}$ such that \,$t,\,x_{2\Sigma}^{(1)}$ satisfy \eqref{x_2Sigma_(1)+}.
  \item Instead of conditions \ref{Inv-set-2a}, \ref{Inv-set-2b}, it is required that for any fixed ${\hat{t}\in \R_+}$, ${\hat{x}_{2\Sigma}^{(1)}\in M_{2\Sigma}^{(1)}}$ satisfying \eqref{x_2Sigma_(1)+}, the operator
 \begin{equation*}
\partial_{x_{2\Sigma}^{(1)}}F_{2\Sigma}^{(1)}\big(\hat{t}, \hat{x}_{2\Sigma}^{(1)}\big)=\! \left[\partial_x \big(Q_{2*}^{(2)}f\big)\big(\hat{t},\hat{x}_{2\Sigma}^{(1)}\big)- Q_{2*}^{(2)}B\right]\! P_{2\Sigma}^{(1)}\big|_{D_{2\Sigma}^{(1)}}\colon D_{2\Sigma}^{(1)}\to Y_{2*}^{(2)}
\end{equation*}
 has an inverse belonging to $\spL\big(Y_{2*}^{(2)},D_{2\Sigma}^{(1)}\big)$.
\end{itemize}

Analogues of the remaining theorems and corollaries which were proved for the ADAE  \eqref{ADAE}, \eqref{f-Appr-2} are also easy to obtain for the ADAE  \eqref{ADAE}, \eqref{f-2-Appr-2}. Therefore, we will not provide them here.
\end{remark}

 \begin{remark}\label{Rem_Appr2_Semi-inv}
As in Section \ref{Approach_1} (see Remark \ref{Rem_Appr1_Semi-inv}), in the above theorems, lemmas and corollaries one can use the semi-inverse operator $A^{(-1)}$ defined by \eqref{Semi-inverse} instead  of the inverse $\widetilde{A}^{-1}$ of \eqref{widetilde_A}, taking into account that ${A^{(-1)}=\widetilde{A}^{-1}(Q_1+Q_{2\Sigma})}$,\; ${\widetilde{A}^{-1}\widetilde{M}_1=A^{(-1)}\widetilde{M}_1}$ \;($\widetilde{M}_1\subseteq Y_1$),\; $\widetilde{A}x_1=Ax_1$,\; $\widetilde{A}^{-1}w_1=A^{(-1)}w_1$ \;($w\in Y_1$),\;
$\widetilde{A}^{-1}\widetilde{\Omega}_1=A^{(-1)}\widetilde{\Omega}_1$ \;($\widetilde{\Omega}_1\subseteq Y_1$).
 \end{remark}

 \subsection{Cases of the characteristic pencil of index 1 and 2}

In applications, DAEs of index 1 most often occur. DAEs of index 2 often arise in gas industry, control problems, mechanics of multilink mechanisms (e.g., robotics) and chemical kinetics (see, e.g., \cite{Brenan-C-P,Vlasenko1,Riaza,FR999,Rutkas2007,Lamour-Marz-Tisch}). In robotics, chemical kinetics and control problems, DAEs of index 3 arise as well. DAEs with an arbitrarily high index occur, e.g., in the study of electrical networks \cite{Riaza}.

Below we specify what form the theorems and other statements from Section \ref{Sec-Main} will have for the ADAE \eqref{ADAE} with the regular pencil of index 1 and~2.

 \paragraph{The index-2 case.} \quad

Consider the case when the regular pencil \eqref{Pencil} has index 2 (that is, the order of a pole of $R(\la)$ at $\la=\infty$ is $r=1$, or $\widehat{R}(\mu)$ has a pole of order $\nu=2$ at $\mu=0$). Then the canonical system $\{\varphi_i^j\}_{i=1,...,n}^{j=1,...,m_i}$ ($1\le m_i\le 2$) contains the eigenvectors (satisfying $A\varphi_i^1=0$, $i=1,...,n$) and the adjoined vectors of order 1 (satisfying $A\varphi_l^2=-B\varphi_l^1$,  where $l\in \{1,...,n\}$ are those indices for which $m_l=2$).
In this case, the subspaces from the direct decompositions of $D$ and $Y$, which are described in Section \ref{Sec_DirDec}, satisfy the equalities
 \begin{equation}\label{D2_Y2_ind2}
\begin{split}
& D_{20}=D_{20}^{(1)}\dotpl D_{20}^{(2)}, \quad  D_{20}^{(2)}=D_{2\wedge,0}, \quad D_{21}=D_{2\Sigma}=D_{2\Sigma}^{(1)},    \\
& Y_{20}=Y_{20}^{(1)}\dotpl Y_{20}^{(2)}, \quad Y_{20}^{(2)}=Y_{2\Sigma}=Y_{2\Sigma,0}, \quad Y_{20}^{(1)}=Y_{2*}^{(1)}, \quad  Y_{21}=Y_{21}^{(2)}=Y_{2*}^{(2)}.
\end{split}
\end{equation}

Hence, ${P_{21}=P_{2\Sigma}=P_{2\Sigma}^{(1)}}$,\;
${P_{20}^{(2)}=P_{2\wedge,0}}$,\; ${Q_{20}^{(2)}=Q_{2\Sigma}=Q_{2\Sigma,0}}$,\;
${Q_{20}^{(1)}=Q_{2*}^{(1)}}$,\;
${Q_{21}=Q_{21}^{(2)}=Q_{2*}^{(2)}}$ and,
accordingly,
$$
x_{2\Sigma}=x_{2\Sigma}^{(1)}=x_{21}, \quad x_{20}=x_{20}^{(1)}+x_{20}^{(2)},\quad x_{20}^{(2)}=x_{2\wedge,0}.
$$
In addition, ${A_{2\Sigma}=A_{21}=A\big|_{D_{21}}\colon D_{21}\to Y_{20}^{(2)}}$, ${Q_{2\Sigma}Bx=Q_{20}^{(2)}Bx=BP_{20}^{(2)}x=BP_{2\wedge,0}x}$ and ${Q_{2*}^{(2)}Bx=Q_{21}Bx=BP_{21}x=BP_{2\Sigma}x}$. Also, recall that by construction $x_{21}=x_{21}^{(2)}$ and $Q_{21}=Q_{21}^{(2)}$.

 \medskip
First, we specify the results of \textbf{Section \ref{Approach_1}}.

The ADAE \eqref{ADAE} with the regular pencil $P(\la)$ of index 2 is equivalent to the system  \eqref{ADAE_DE_Pi},~\eqref{ADAE_AE_F2*} (or \eqref{DE1+2}, \eqref{AE}), where ${x_{2\Sigma}=x_{21}}$, ${Q_{2*}=Q_{20}^{(1)}+Q_{21}}$ and ${Q_{2\Sigma}=Q_{20}^{(2)}}$, i.e.,
 \begin{align}
& \frac{d}{dt}[x_1+x_{21}]=\Pi(t,x), \quad \Pi(t,x):=\widetilde{A}^{-1}(Q_1+Q_{20}^{(2)})\big[f(t,x)-Bx\big],  \label{ADAE_DE_Pi_ind2} \\
& F_{2*}(t,x_1,x_{21},x_{20})=0, \quad
  F_{2*}(t,x_1,x_{21},x_{20}):=(Q_{20}^{(1)}+Q_{21})[f(t,x)-B(x_{21}+x_{20})].  \label{ADAE_AE_F2*_ind2}
 \end{align}
In addition,  for the theorems and other results of Section \ref{Approach_1} we have ${D_{2\Sigma}=D_{21}}$, ${Y_{2*}= Y_{20}^{(1)}\dotpl Y_{21}}$,  ${Y_{2\Sigma}=Y_{20}^{(2)}}$, ${P_{2\Sigma}=P_{21}}$, and it is convenient to denote ${M_{21}:=M_{2\Sigma}}$. Recall that in Section \ref{Approach_1} we consider the case when the component $x_{20}$ can be defined as an implicit function from \eqref{ADAE_AE_F2*}, thus the function $F_{2*}$ should depend on $x_{20}$. Equation \eqref{ADAE_AE_F2*} takes the form \eqref{ADAE_AE_F2*_ind2} for the ADAE with the index-2 regular pencil.  \emph{Taking into account the mentioned relations, it is easy to formulate the results (theorems, corollaries, etc.) of Section \ref{Approach_1} in an appropriate way.}

 \medskip
Now let us specify the results of \textbf{Section \ref{Approach_2}}.

The function $f$ having the structure \eqref{f-Appr-2} as well as \eqref{f-2-Appr-2} (recall that by construction $x_{21}=x_{21}^{(2)}$ and $Q_{21}=Q_{21}^{(2)}$) takes the form
\begin{equation}\label{f-Appr-2_ind2}
f(t,x)=\big(Q_1+Q_{20}\big)f(t,x)+Q_{21}f(t,x_{21}),
\end{equation}
and the system \eqref{x_1}--\eqref{x_2(nu-1)} as well as the system  \eqref{x_1}--\eqref{x_2wedge,nu-2}, \eqref{x_2Sigma_(1)} takes the form
 \begin{align}
\frac{d}{dt}x_1+\widetilde{A}^{-1}Bx_1 &=\widetilde{A}^{-1}Q_1f(t,x), \label{x_1_ind2}\\
\frac{d}{dt}\big[Ax_{21}\big]+ Bx_{20} &= Q_{20}f(t,x), \label{x_20_ind2}  \\
Bx_{21} &=Q_{21}f\big(t,x_{21}\big). \label{x_2(nu-1)_ind2}
 \end{align}
Correspondingly, equation \eqref{x_2(nu-1)+}, equivalent to \eqref{x_2(nu-1)} and, accordingly, to \eqref{x_2(nu-1)_ind2}, has the form
\begin{equation}\label{x_2(nu-1)_ind2+}
F_{21}\big(t,x_{21}\big)=0,\qquad
F_{21}\big(t,x_{21}\big)= Q_{21}f\big(t,x_{21}\big)-Bx_{21}.
\end{equation}
In equation \eqref{x_20+} $d_{21}^{(2)}=d_{21}$, $d_{2\wedge,1}=0$ and ${x_{2\Sigma}=x_{21}}$, and the corresponding equation is equivalent to \eqref{x_20} and to \eqref{x_20_ind2} where $\frac{d}{dt}x_{21}^{(2)}=\frac{d}{dt}x_{21}$ is replaced by $d_{21}^{(2)}=d_{21}$ \,($x_{2\wedge,1}=0$), that is,
\begin{equation}\label{x_20_ind2+}
F_{20} \big(t,x_1,x_{20},x_{21},d_{21}\big)=0, \;
 F_{20}\big(t,x_1,x_{20},x_{21},d_{21}\big)\!= Q_{20}f\big(t,x_1+x_{20}+x_{21}\big) - Bx_{20} -Q_{20}A\,d_{21}.
\end{equation}

As mentioned in Section \ref{Intro}, an ADAE of the form \eqref{ADAE} with the pencil of index 2 has been considered in \cite{Rut-Khudoshin}. This ADAE also has the right-hand side of a special structure and was reduced to the system that can be transformed into the system \eqref{x_1_ind2}, \eqref{x_20_ind2}, \eqref{x_2(nu-1)_ind2} where $Q_{20}f$ depends only on $t,\, x_2$.

Theorem \ref{Th_Exist-Lip_set-2} and Corollary \ref{Corollary_Exist_set-2} for the ADAE \eqref{ADAE} with the regular pencil of index 2 take the following form.

 \begin{theorem}\label{Th_Exist-Lip_set-2_ind2}
Let the function $f\in C(\R_+\times D,Y)$ have the structure \eqref{f-Appr-2_ind2}, ${D_B\supseteq D_A=D}$, ${\dim\ker A=n<\infty}$, and let ${\la A+B}$ be a regular pencil of index 2.
Assume that there exists open sets ${\widetilde{M}_1\subseteq Y_1}$ and ${M_{20}\subseteq D_{20}}$, ${M_{21}\subseteq D_{21}}$ such that
\,$Q_{21}f\in C^1(\R_+\times M_{21},Y_{21})$\, and the following conditions where \,${M_1=\widetilde{A}^{-1}\widetilde{M}_1\subseteq D_1}$ and
${M=M_1\dotpl M_{20}\dotpl M_{21}}$\,  hold:
  \begin{enumerate}[label={\upshape{1.\alph*.}}, ref={\upshape{1.\alph*}},itemsep=4pt,parsep=0pt,topsep=5pt,leftmargin=1cm]
 \item\label{Sogl_set-1a_ind2}  For each  ${t\in \R_+}$ there exists a unique ${x_{21}\in M_{21}}$ such that \,$t,\,x_{21}$ satisfy \eqref{x_2(nu-1)_ind2}.
 \end{enumerate}

 \begin{enumerate}[label={\upshape{1.\alph*.}}, ref={\upshape{1.\alph*}},itemsep=4pt,parsep=0pt,topsep=5pt,leftmargin=1cm]
   \addtocounter{enumi}{4}
 \item\label{Sogl_set-1e_ind2}  For each ${t\in \R_+}$, ${x_1\in M_1}$, ${x_{21}\in M_{21}}$, ${d_{21}\in D_{21}}$ there exists a unique ${x_{20}\in M_{20}}$ such that \,$t$, $x_1$, $x_{20}$, $x_{21}$, $d_{21}$  satisfy  \eqref{x_20_ind2+}.
 \end{enumerate}

    \begin{enumerate}[label={\upshape{2.\alph*.}}, ref={\upshape{2.\alph*}}, itemsep=4pt,parsep=0pt,topsep=5pt,leftmargin=1cm]
 \item\label{Inv-set-2a_ind2} For any fixed ${\hat{t}\in \R_+}$, ${\hat{x}_{21}\in M_{21}}$ satisfying \eqref{x_2(nu-1)_ind2}, the operator
   \begin{equation*}
  \partial_{x_{21}}F_{21}\big(\hat{t},\hat{x}_{21}\big)\!=\! \left[\partial_x (Q_{21}f)(\hat{t},\hat{x}_{21})-Q_{21}B\right] P_{21}\big|_{D_{21}} \colon D_{21}\to Y_{21}
   \end{equation*}
  has the inverse $\left[\partial_{x_{21}}F_{21} \big(\hat{t},\hat{x}_{21}\big)\right]^{-1} \in \spL(Y_{21},D_{21})$.
    \end{enumerate}

  \begin{enumerate}[label={\upshape{2.\alph*.}}, ref={\upshape{2.\alph*}}, itemsep=4pt,parsep=0pt,topsep=5pt,leftmargin=1cm]
    \addtocounter{enumi}{4}
 \item\label{Inv-Lipsch_set-2e_ind2}
 The components $Q_1f(t,x_1+x_{20}+x_{21})$ and $Q_{20}f(t,x_1+x_{20}+x_{21})$ of the function $f(t,x)$ satisfy locally a Lipschitz condition with respect to $x_1+x_{20}$ on $\R_+\times M$.  \\
 For any fixed $\hat{t}\in \R_+$, $\hat{x}_1\in M_1$, $\hat{x}_{20}\in M_{20}$, $\hat{x}_{21}\in M_{21}$ \textup{(}accordingly,  $\hat{x}=\hat{x}_1+\hat{x}_{20}+\hat{x}_{21}\in M$\textup{)},
 $\hat{d}_{21}\in D_{21}$ satisfying \eqref{x_20_ind2+}
 there exist open neighborhoods
 $U_{\delta}\big(\hat{t},\hat{x}_1,\hat{x}_{21},\hat{d}_{21}\big)= U_{\delta_1}\big(\hat{t}\big) \times U_{\delta_2}\big(\hat{x}_1\big) \times U_{\delta_3}\big(\hat{x}_{21}\big)  \times U_{\delta_4}\big(\hat{d}_{21}\big)\subset
 \R_+\times D_1\times D_{21}\times D_{21}$ and $U_\varepsilon\big(\hat{x}_{20}\big)\subset D_{20}$ and an operator $\mathrm{Z}= \mathrm{Z}_{\hat{t},\hat{x},\hat{d}_{21}} \in \spL(D_{20},Y_{20})$ having the inverse $\mathrm{Z}^{-1}\in \spL(Y_{20},D_{20})$ such that for each $t\in U_{\delta_1}\big(\hat{t}\big)$,
 $x_1\in U_{\delta_2}\big(\hat{x}_1\big)$,
 $x_{21}\in U_{\delta_3}\big(\hat{x}_{21}\big)$,
 $d_{21}\in U_{\delta_4}\big(\hat{d}_{21}\big)$ and each $x'_{20},\, x''_{20}\in U_\varepsilon\big(\hat{x}_{20}\big)$ the mapping $F_{20}$ satisfies the inequality
  \begin{equation*}
 \big\|F_{20}\big(t,x_1,x'_{20},x_{21},d_{21}\big)-  F_{20}\big(t,x_1,x''_{20},x_{21},d_{21}\big)- \mathrm{Z}[x'_{20}-x''_{20}]\big\|\le k(\delta,\varepsilon)\big\|x'_{20}-x''_{20}\big\|,
  \end{equation*}
  where $k(\delta,\varepsilon)$ is such that $\lim\limits_{\delta,\,\varepsilon\to 0} k(\delta,\varepsilon)< \|\mathrm{Z}^{-1}\|^{-1}$  \,\textup{(}the numbers $\delta, \varepsilon>0$ depend on the choice of $\hat{t},\hat{x}$, $\hat{d}_{21}$\textup{)}.
    \end{enumerate}
Then for each initial value $t_0\in \R_+$ and $x_{0,1}\in M_1$ there exists a $t_{max}\le \infty$ such that the IVP \eqref{ADAE}, \eqref{f-Appr-2_ind2}, \eqref{ini-2} has a unique solution $x(t)$ in $M$ on the maximal interval of existence $[t_0,t_{max})$.
 \end{theorem}

 \begin{corollary}\label{Cor_Exist_set-2_ind2}
Theorem~\ref{Th_Exist-Lip_set-2_ind2} remains valid if condition \ref{Inv-Lipsch_set-2e_ind2} is replaced by
\begin{enumerate}[label={\upshape{2.\alph*.}}, ref={\upshape{2.\alph*}}, itemsep=4pt,parsep=0pt,topsep=5pt,leftmargin=1cm]
\addtocounter{enumi}{4}
\item\label{Inv_set-2e_ind2} The components $Q_1f(t,x)$ and $Q_{20}f(t,x)$ of the function $f(t,x)$ have the continuous partial derivatives with respect to $x$ on $\R_+\times M$.  \\
 For any fixed $\hat{t}\in \R_+$, $\hat{x}_1\in M_1$, $\hat{x}_{20}\in M_{20}$, $\hat{x}_{21}\in M_{21}$ \textup{(}accordingly,  $\hat{x}=\hat{x}_1+\hat{x}_{20}+\hat{x}_{21}\in M$\textup{)}, $\hat{d}_{21}\in D_{21}$ satisfying \eqref{x_20_ind2+}, the operator
     \begin{equation*}
 \mathrm{Z}_{\hat{t},\hat{x},\hat{d}_{21}}\!:=\! \partial_{x_{20}}F_{20}\big(\hat{t},\hat{x}_1,\hat{x}_{20},\hat{x}_{21}, \hat{d}_{21}\big)\!=\! \Big[\partial_x (Q_{20}f)(\hat{t},\hat{x})-Q_{20}B\Big]P_{20}\big|_{D_{20}}\colon D_{20}\to Y_{20}
     \end{equation*}
 has the inverse $\mathrm{Z}_{\hat{t},\hat{x},\hat{d}_{21}}^{-1}\in \spL(Y_{20},D_{20})$.
\end{enumerate}
 \end{corollary}

Given the above, \emph{it is easy to formulate the remaining results of Section \ref{Approach_2} for the ADAE \eqref{ADAE} with the regular pencil of index 2.}

   \paragraph{The index-1 case.} \quad

As stated in both Section \ref{Approach_1} and \ref{Approach_2}, the ADAE \eqref{ADAE} with the regular pencil $P(\la)$ of index 1 is equivalent to the system \eqref{DE_ind-1}, \eqref{AE_ind-1}. Thus, in the index-1 case, the results (theorems, corollaries, etc.) of Section \ref{Sec-Main} have the form of the results of Section \ref{Approach_1} where $D_{20}=D_2$, $D_{2\Sigma}=\emptyset$, $Y_{2*}=Y_2$,  $P_{20}=P_2$, $P_{2\Sigma}=0$,  $Q_{2*}=Q_2$, $Q_{2\Sigma}=0$, $x_{20}=x_2$ and $x_{2\Sigma}=0$.

\section{Applications}

 \subsection{Example}

Consider the initial boundary value problem (IBVP) from \cite[Section 5.2]{Rut-Khudoshin}:
 \begin{equation}\label{Ex1}
\begin{split}
& \partial_t \partial^2_x u + \left(\partial^2_x +1\right)u=f(t,u) \\
&2u(t,1)-2u(t,0)-\partial_x u(t,1)=0, \\
&\partial_x u(t,0)=0,  \\
&u(0,x)=0,
\end{split}
 \end{equation}
where $t\in \R_+$, $x\in [0,1]$. It can be represented as the following IVP for the ADAE of the type \eqref{ADAE} in $L_2[0,1]$:
 \begin{equation}\label{Ex1-IVP}
\begin{split}
& \frac{d}{dt}[Au]+Bu=f(t,u), \quad A=\partial^2_x,\quad B=\partial^2_x+1,  \\
& u(0)=0,
\end{split}
 \end{equation}
where $u\in L_2[0,1]$ and $A$, $B$ are defined on $D_A=D_B=D=\{{y\in L_2[0,1]}\mid y'(x) \text{ is absolutely continuous on $[0,1]$}, \; y'(0)=0, \; 2y(1)-2y(0)-y'(1)=0\}$.

The resolvent $(A+\mu B)^{-1}$ has a pole of order 2 at the point $\mu=0$. Therefore, $P(\la)$ is a \emph{regular pencil of index 2}. Since $\ker A= \Span\{\varphi_1^1\}$, where $\varphi_1^1=1$, and the index $\nu=2$, then the canonical system of root vectors of the pencil $A+\mu B$ at $\mu =0$ consists of the vectors $\varphi_1^1$, $\varphi_1^2$, where $\varphi_1^2$ can be found from the relation $A\varphi_1^2=-B\varphi_1^1$  (see \eqref{eavA}).
The projectors take form
\begin{align*}
& P_{20}u=(u,p^1_1)\varphi_1^1,\quad P_{21}u=(u,p_1^2)\varphi_1^2, \quad P_2=P_{20}+P_{21}, \quad P_1=I_X-P_2, \\
& Q_{20}v=(v,q_1^1)B\varphi_1^1,\quad Q_{21}v=(v,q_1^2)B\varphi_1^2, \quad Q_2=Q_{20}+Q_{21}, \quad Q_1=I_Y-Q_2,
\end{align*}
where
\begin{align*}
& \varphi_1^1=1, \quad \varphi_1^2=\sqrt{5} x^2,  \\
& p_1^1=B^*q_1^1=4x^3-6x^2-\frac{8}{5}\left(x-\frac{1}{2}\right)+2,  \\
& p_1^2=B^*q_1^2=\frac{12}{\sqrt{5}}\left(x-\frac{1}{2}\right)=q_1^2,  \\
& B\varphi_1^1=1, \quad B\varphi_1^2=\sqrt{5} (x^2+2),  \\
& q_1^1=4x^3-6x^2-\frac{128}{5}\left(x-\frac{1}{2}\right)+2, \\
& q_1^2=\frac{12}{\sqrt{5}}\left(x-\frac{1}{2}\right).
\end{align*}
and the subspaces from the direct decompositions \eqref{XDYrr}, \eqref{DY2rr}, \eqref{D2_Appr-2}, \eqref{D2_Y2_ind2} take the form $D_{20}=\Span\{\varphi_1^1\}$,  $D_{21}=D_{2\Sigma}=\Span\{\varphi_1^2\}$,  $Y_{20}=Y_{2\Sigma}=\Span\{B\varphi_1^1\}$, and $Y_{21}=Y_{2*}=\Span\{B\varphi_1^2\}$.

Using the relations \eqref{Semi-inverse}, we obtain that the semi-inverse operator of $A$ has the form $A^{(-1)}y(v)=\int\limits_0^v\int\limits_0^s y(\zeta)\,d\zeta\, ds -(v^4+12v^2)\int\limits_0^1\left(s-\frac{1}{2}\right)y(s)\,ds$. Thus, the inverse of \eqref{widetilde_A} has the form $\widetilde{A}^{-1}=A^{(-1)}+\langle \cdot,q_1^2\rangle \varphi_1^1$.

Further, verifying the conditions of the theorems proved above, we find the requirements necessary for the global solvability and Lagrange stability of the ADAE \eqref{Ex1}, or the conditions under which the solutions blow up in finite time.

\section*{Acknowledgments}
This work has received funding from the European Research Council (ERC) under the European Union's Horizon 2020 research and innovation programme, ERC Advanced Grant NEUROMORPH, no. 101018153.

 \small{

}
\end{document}